\documentclass[a4paper,11pt,draft, reqno]{amsart}

\usepackage[margin=18mm]{geometry}
\usepackage{amssymb, amsmath, amsthm, amscd}

\usepackage[T1]{fontenc}
\usepackage{eucal,mathrsfs,dsfont}
\usepackage{color}

\renewcommand{\leq}{\leqslant}
\renewcommand{\geq}{\geqslant}
\renewcommand{\le}{\leqslant}
\renewcommand{\ge}{\geqslant}

\definecolor{mno}{rgb}{0.5,0.1,0.5}

\newcommand{\R}{\mathds R}
\newcommand{\bS}{\mathds S}

\newcommand{\e}{\varepsilon}
\newcommand{\T}{\mathds T}

\newcommand{\Pp}{\mathds P}
\newcommand{\Ee}{\mathds E}

\newcommand{\I}{\mathds 1}

\newcommand{\D}{\mathscr{D}}
\newcommand{\Z}{\mathds Z}

\newcommand{\F}{\mathscr{F}}

\def\eps{{\varepsilon}}

\newcommand{\sL}{\mathcal{L}}
\newcommand{\sA}{\mathcal{A}}

\def\<{\langle}
\def\>{\rangle}

\def\wt{\widetilde}

\newtheorem{theorem}{Theorem}[section]
\newtheorem{lemma}[theorem]{Lemma}
\newtheorem{proposition}[theorem]{Proposition}
\newtheorem{corollary}[theorem]{Corollary}

\theoremstyle{definition}

\newtheorem{example}[theorem]{Example}
\newtheorem{remark}[theorem]{Remark}

\numberwithin{equation}{section}

\begin{document}
\allowdisplaybreaks

\title[Periodic homogenization of non-symmetric L\'evy-type processes]
{\bfseries Periodic homogenization of non-symmetric L\'evy-type
processes}
\author{Xin Chen\qquad Zhen-Qing Chen\qquad Takashi Kumagai\qquad Jian Wang}

 \date{July 6, 2020}

\maketitle

\begin{abstract}
In this paper, we study
 homogenization problem for strong Markov processes on $\R^d$
 having infinitesimal generators
 $$
 \sL f(x)=\int_{\R^d}\left(f(x+z)-f(x)-\langle \nabla f(x), z\rangle \I_{\{|z|\le 1\}} \right) k(x,z)\,
 \Pi (dz)  +\langle b(x), \nabla f(x) \rangle, \quad f\in C^2_b (\R^d)
 $$
 in periodic media, where $\Pi$ is a
non-negative
 measure on  $\R^d$ that does not charge the origin $0$,
 satisfies $\int_{\R^d}  (1 \wedge |z|^2)\, \Pi (dz)<\infty$,
 and can be singular with respect to the Lebesgue measure on $\R^d$.
 Under a proper scaling, we show the scaled processes
converge weakly to L\'evy  processes on $\R^d$.
The results are a counterpart of the celebrated work
\cite{BLP,Bh} in the jump-diffusion setting.
In particular,
we completely characterize the homogenized limiting processes
when $b(x) $ is a bounded continuous      multivariate 1-periodic $\R^d$-valued  function,
  $k(x,z)$ is  a non-negative  bounded  continuous function that is
 multivariate 1-periodic in  both $x$ and $z$ variables,
 and, in spherical coordinate $z=(r, \theta) \in \R_+\times \bS^{d-1}$,
$$
\I_{\{|z|>1\}}\,\Pi (dz) = \I_{\{  r>1\}}  \varrho_0(d\theta) \, \frac{ dr }{r^{1+\alpha}}
$$
with  $\alpha\in (0,\infty)$ and $\varrho_0$
being  any
finite measure  on the unit sphere $\bS^{d-1}$ in $\R^d$.
Different phenomena occur depending on  the values of $\alpha$; there are five cases:
 $\alpha \in (0, 1)$, $\alpha =1$, $\alpha \in (1, 2)$,
$\alpha =2$ and $\alpha \in (2, \infty)$.

\bigskip

\noindent\textbf{Keywords:}  homogenization; L\'evy-type process;
martingale problem; corrector

\bigskip

\noindent \textbf{MSC 2020:} 60F17; 60G51;  60J76;  60J25.

\end{abstract}

\allowdisplaybreaks

\section{Introduction}\label{section1}
In the celebrated work
\cite{BLP,Bh} the authors studied the
periodic
homogenization of a
diffusion $X :=( X_t)_{t\ge0}$ on $\R^d$ generated  by the following second-order elliptic operator
 $$
\hat \sL
f(x)=\frac{1}{2}\sum_{i,j=1}^d a_{ij}(x ) \frac{\partial ^2f(x)}{\partial x_i\partial
x_j}+ \sum_{i=1}^d b_i(x)\frac{\partial f(x)}{\partial x_i},
$$
where the coefficients
$(a_{ij}(x))_{1\le i,j\le d}$
and $b(x):=(b_i(x))_{1\le i\le d}$ are bounded and  multivariate 1-periodic (that is, they can be viewed as bounded functions
defined on the $d$-dimensional torus $\T^d:=(\R/\Z)^d$).  Under the  assumptions that
each $a_{ij}(\cdot) $
has  bounded second derivatives, each $ b_i(\cdot)$ has bounded first derivatives,
and $(a_{ij}(x))_{1\le i,j\le d}$ is uniformly elliptic, they showed
 $$
   \left\{ X_t^\e -\e^{-1} \bar b t:  \ {t\ge0} \right\}
 $$
 converges weakly as $\e \to 0$ to a
driftless Brownian motion with  covariance matrix
$$  \bigg(\int_{\T^d} \sum_{k,l=1}^d\Big(\delta_{ki}-\frac{\partial \psi_i(x)}{\partial
x_k}\Big) a_{kl}(x)\Big(\delta_{lj}-\frac{\partial
\psi_j(x)}{\partial x_l}\Big)\,\mu(dx)\bigg)_{1\le i,j\le d}.
$$
Here, $ X^\e_t:=\e X_{t/\e^{2}}$ for  $t\ge0$, $\mu(dx)$ is the unique invariant probability
measure
for the quotient process of $X$ on $\T^d$,
 $\bar b=\int_{\T^d} b(x)\,\mu(dx)$,  and $\psi
\in C^2(\R^d)$ is the unique periodic solution to the equation $$
\hat \sL
\psi(x)=b(x)-\bar b \quad  \hbox{on }  \T^d.
$$

The goal of this paper is to study
the periodic homogenization of jump diffusions
whose infinitesimal generators are of the following form when acting on $C_b^2(\R^d)$:
\begin{equation}\label{e4-1}
\sL f(x)=\int_{\R^d}\left(f(x+z)-f(x)-\langle \nabla f(x), z\rangle \I_{\{|z|\le 1\}} \right) k(x,z)\,
\Pi (dz)
+\langle b(x), \nabla f(x) \rangle.
\end{equation}
Here, $\Pi (dz)$ is a non-negative
 measure on $\R^d$
that does not charge at the origin $0$ and satisfies
$
\int_{\R^d}(1\wedge |z|^2)\,\Pi (dz) <\infty;
$
$b(x)$  is
 a  bounded  continuous  multivariate 1-periodic $\R^d$-valued function,
 and $k(x,z): \R^d\times \R^d  \to
   [0,\infty)$
is a function that is bounded so that $x\mapsto k(x,z)$ is multivariate 1-periodic
for each fixed $z\in \R^d$
and
\begin{equation}\label{e0}
\lim_{y\to x}\sup_{z\in \R^d}|k(y,z)-k(x,z)|=0.
\end{equation}

Since $k(\cdot,z)$ and $b(\cdot)$ are multivariate 1-periodic,
it is easy to verify that $\sL f$ is
pointwisely  well defined as a function on $\T^d$
for every $f\in C^2(\T^d)$.  By the maximum principle, $(\sL, C^2(\T^d))$ can be extended to a closed operator
$(\sL, \mathscr{D}(\sL))$ on $C(\T^d)$.
 {\it Throughout the paper,
the following two assumptions are in force.}

 \begin{itemize} \it
\item [\bf (A1)] There exists a strong Markov process $X:=\big((X_t)_{t\ge 0};(\Pp_x)_{x\in \R^d}\big)$ associated with $\sL$
in the sense that
for every $f\in C_b^{1,2}(\R_+\times\R^d)$
with $\R_+:=[0,\infty)$,
$$\left\{
f(t,X_t)-f(0,X_0)-\int_0^t \bigg(\frac{\partial f}{\partial
s}(s,X_s)+\sL f(s,\cdot)(X_s)\bigg)\,ds,\,\, t\ge0\right\}
$$
is a martingale under $\Pp_x$ for all $x\in \R^d$ with respect to
the natural filtration generated by $X$.

\item [\bf (A2)] Regarding $X$ as an $\T^d$-valued process, the process $X$ is exponentially ergodic in the sense that
 there exist a unique
invariant probability  measure $\mu(dx)$ and constants $\lambda_1$,
$C_0>0$ so that
\begin{equation}\label{a1-2-1}
\sup_{x\in \T^d}|\Ee_xf(X_t)-\mu(f)|\le C_0e^{-\lambda_1
t}\|f\|_\infty,\quad t>0,\ f\in C_b(\T^d).
\end{equation}
\end{itemize}

Assumption {\bf (A1)} is satisfied, when the martingale problem for the operator $(\sL, C_b^2(\R^d))$ is well posed.
The latter has been extensively investigated in the literature, see \cite{CCW, CZ,CZ1,CZ2, Peng}
and the
references therein. See also \cite[Theorem 3.1]{Ku} for
more recent study on the existence of a martingale solution associated with L\'evy type operators whose corresponding canonical
process has the strong Markov property. Since $\T^d$ is compact, Assumption {\bf (A2)} is a direct consequence of the irreducibility (that is, for any $t>0$,
$x\in \R^d$ and any non-empty open set $U\subset \R^d$, $\Pp_x(X_t\in U)>0$) and the strong Feller property (that is, for any $f\in B_b(\R^d)$ and $t>0$, $x\mapsto \Ee_xf(X_t)$ is bounded and continuous) of the process $X$; see \cite[Theorem 1.1]{Wa}. Assumption {\bf (A2)} also holds, if the process $X$ admits a transition density function $p(t,x,y)$ with respect to the Lebesgue
measure so that for any $t>0$, the function $(x,y)\mapsto p(t,x,y)$ is continuous on $\R^d\times \R^d$, and that there is a non-empty open set $U\subset \R^d$ such that $p(t,x,y)>0$  for all
$t>0$, $x\in \R^d$ and $y\in U$;
see \cite[p.\ 365, Theorem 3.1]{BLP} for a modification of Doeblin's celebrated result.
The reader is referred
to Subsection \ref{S:7.2} for concrete examples on Assumptions {\bf (A1)} and {\bf (A2)}.

In order to deal with
the scaling limit of
$X$
that requires recentering
(which includes cases considered in Subsection \ref{section3-2}, Section
\ref{section4} and Section \ref{section5}),
we need one more assumption.

\begin{itemize}
\item[\bf (A3)]\it
For every $f\in C(\T^d)$
with $\mu(f)=0$, there exists a unique multivariate $1$-periodic
 solution $\psi\in \mathscr{D}(\sL)$ to
\begin{equation}\label{a3-1-1}
\sL \psi =f   \  \hbox{ on }  \, \T^d \quad \hbox{ with} \quad \mu(\psi)=0
\end{equation}
and
\begin{equation}\label{a3-1-2}
\|\psi\|_\infty+\|\nabla \psi\|_\infty\le
C_1\|f\|_\infty,
\end{equation}
where $C_1>0$
is  independent of $f$.
\end{itemize}

There are a few literature on homogenization of non-local operators.
We refer readers to \cite{A1,A2, KPZ, PZ1, Sc1} for the periodic
homogenization results for stable-like operators or the operator
with convolution type kernels. The methods used in the these papers
are analytic and called the corrector method. The probabilistic
study of homogenization of periodic stable-like processes  can be
found in \cite{F1, HDS, San} via the characteristics of
semimartingales, in \cite{F2,  HIT, HDS0} by SDE driven by L\'evy
processes or by Poisson random measures, and in \cite{FT, Tom} via
the martingale problem method.
A closely related topic is homogenization of non-local operators or jump diffusions in random media,
which typically requires a different approach than the periodic media case,
see \cite{PZ2, Sc2} for example.
Recently we have studied homogenization of symmetric  stable-like processes in stationary ergodic random media
in \cite{CCKW1}.

 \medskip

The approach of this paper is  different from all the papers
mentioned above. We will use generator method combined with its
connection to martingales.
In particular, we summarize the novelties of our paper as follows.
\begin{itemize}

\item [(i)] Our results reveal that
  the crucial ingredients
  for the homogenization of L\'evy type operators
 are the  shape of large jumps for the jumping measures $\Pi (dz)$ and
its limiting spherical measure on the unit sphere $\bS^{d-1}$ when expressed in spherical coordinates,
 see e.g.\ conditions \eqref{e:3.3a}, \eqref{a2-2-1} and \eqref{a2-3-2}--\eqref{a2-3-1}.
Compared with the references mentioned above on the homogenization of non-local operators (which are mainly concerned on stable-like processes), our results work for
more general jump processes
with a large class of scaling factors, see the weighted function $\varphi$ in Section \ref{section3}, Example \ref{ex4-3} and Remark \ref{r6-6}.
Moreover, limiting processes of our homogenization results are also
quite general, including {\it all} stable L\'evy processes on $\R^d$ that can be
non-symmetric,
degenerate and have singular L\'evy measures,
  see Example \ref{ex4-1} below and examples in Subsection \ref{S:7.1} for more details.

\item [(ii)] We will establish the periodic homogenization results for L\'evy type operators
$\sL$ after suitable scalings that
depend only
on the tail of
$\Pi (dz) $ which gives the rate of large jumps.
Our results in particular cover the critical case; see Example \ref{ex4-1}(ii).
To the best of our knowledge, this is the first time in literature that critical cases have been studied.
Moreover,  the L\'evy measure $\Pi$ can be singular with respect to the Lebesgue measure on $\R^d$ and
its support can have zero Lebesgue measure.

 \item [(iii)] Among all the results mentioned above, the process under investigation
 is either
the unique strong solution  of a stochastic differential equation or a Feller process on $\R^d$.
In this paper,   the process is only assumed to
solve the martingale problem of  $(\sL, C_b^2(\R^d))$ and   have
strong Markov property.
\end{itemize}

\medskip

The main results of this paper are Theorems \ref{t3-1}, \ref{t3-2},
\ref{t3-3} and \ref{t3-4}
as well as Theorems \ref{T:6.5} and \ref{T:6.7}.
We use the following example, a special case of these much more general results,
for illustration.
We first introduce some notations which will be frequently used in the paper.
By assumptions on $k$ and $\Pi$ (see the line immediately after \eqref{e4-1}),
 \begin{equation}\label{e:1.2}
b_R(x):=\int_{\{1<|z|\le R\}}zk(x,z)\,\Pi (dz)
\end{equation}
is well defined for every $R>1$, and
 $b_R\in C_b (\R^d)$ is
 multivariate 1-periodic.
 Clearly,  if
$\int_{\{|z|>1\}}|z|\,\Pi (dz) <\infty$,
then
\begin{equation}\label{e:1.3}
b_\infty(x):=
\int_{\{|z|>1\}}zk(x,z)\,\Pi (dz)
\end{equation}
is also well defined
and is the limit of
$b_R(x)$ as $R\to \infty$.
Let $\R_+:=[0,\infty)$ and $\bS^{d-1}$ be the unit sphere in $\R^d$.
Denote by $\D([0, \infty); \R^d)$ the space of $\R^d$-valued right continuous functions having left limits
on $[0, \infty)$, equipped with the Skorohod topology.

\begin{example}\label{ex4-1}
Suppose that Assumptions {\bf(A1)}, {\bf(A2)} and {\bf(A3)} hold. Let
 $X:=(X_t)_{t\ge 0}$ be the
strong Markov process corresponding to the operator $\sL$ given by
\eqref{e4-1} with the jumping measure $\Pi (dz) $ such that
\begin{equation}\label{ex4-1-1}
\I_{\{|z|>1\}}\,\Pi (dz) =
\I_{\{r>1\}}\frac{1}{r^{1+\alpha}}dr\varrho_0(d\theta)
\end{equation}
 where $\alpha\in (0,\infty)$, $\varrho_0(d\theta)$ is a
 non-negative finite measure on $\bS^{d-1}$
 and $(r,\theta)$ denotes the spherical coordinates of $z\in \R^d$.

Denote by $\mu(dx)$ the stationary probability measure for the quotient process of
$X$ on $\T^d$. Define for any $R>1$,
\begin{equation}\label{e0a}
\bar
b:=\int_{\T^d} b(x)\,\mu(dx), \quad  \bar b_R:=\int_{\T^d}
b_R(x)\,\mu(dx),\quad \bar b_\infty:=\int_{\T^d} b_\infty(x)\,\mu(dx).
\end{equation}

\begin{itemize}
\item [(i)] Suppose that
$k(x, z)$ is a bounded continuous function on $\R^d\times \R^d$ so that
 $x\mapsto k(x,z)$ is
 multivariate 1-periodic for each fixed $z\in \R^d$
 and $z\mapsto k(x,z)$ is multivariate 1-periodic for each fixed $x\in \R^d$.
For any $\e\in (0,1]$, define
 $(Y_t^\e)_{t\ge 0}$ by
\begin{align*}
Y_t^\e=
\begin{cases}
\e X_{t/\e^{\alpha}},      & 0<\alpha<1,\\
\e X_{t/\e^{\alpha}} -(\bar b_{1/\e}+\bar b) t,&\alpha=1,\\
\e X_{t/\e^{\alpha}} -\e^{1-\alpha}(\bar b_\infty+\bar b)t,&1<\alpha<2.
\end{cases}
\end{align*}
Then the
process $(Y_t^\e)_{t\ge0}$ converges weakly
in  $\D ([0, \infty); \R^d)$,
as $\e \to 0$,
to a (possibly non-symmetric)
 $\alpha$-stable L\'evy process $(\bar X_t)_{t\ge 0}$
having jumping kernel
$\frac{\bar k (\theta)}{r^{1+\alpha}}\,dr\,\varrho_0(d\theta)$,
where $\bar k:\bS^{d-1}\to \R_+$ is defined by
$$
\bar k
(\theta):=\int_{\T^d}\bar k(x,\theta)\,\mu(dx),\quad  \theta\in \bS^{d-1},$$
and $\bar k: \R^d\times \bS^{d-1}\to \R_+$ satisfies that for all $x\in \R^d$ and $\theta\in \bS^{d-1}$,
\begin{equation}\label{e:notenote}
\bar k(x,\theta)=\lim_{T\to \infty}\frac{1}{T}\int_0^T k(x,(r,\theta))\,dr.
\end{equation}
The infinitesimal generator of the L\'evy process $(\bar X_t)_{t\ge 0}$ is given by
 \begin{align*}
\bar\sL f(x)=
\begin{cases}
 \displaystyle\int_{\R^d}\left(f(x+z)-f(x)\right)
 \bar k(z/|z|)
 \,\Pi_0(dz)
 &\quad  \alpha\in (0,1),\\
  \displaystyle\int_{\R^d}\left(f(x+z)-f(x)-\langle \nabla f(x),z\rangle\I_{\{|z|\le 1\}} \right)
    \bar k(z/|z|)
  \,\Pi_0(dz)
     &\quad\ \alpha=1,\\
  \displaystyle\int_{\R^d}\left(f(x+z)-f(x)-\langle \nabla f(x),z\rangle \right)
    \bar k(z/|z|)
  \,\Pi_0(dz)
  &\quad \alpha\in (1,2),
\end{cases}
\end{align*}
where
$\Pi_0(dz):=\I_{\{r>0\}}\frac{1}{r^{1+\alpha}}\,dr\,\varrho_0(d\theta)$.

Furthermore, if the finite measure  $\varrho_0$ on $\bS^{d-1}$ that does not charge on the set of rationally dependent $\theta \in \bS^{d-1}$,  then we can take
$\bar k (\theta)\equiv \int_{\T^d}\int_{\T^d}k(x,z)\,dz\,\mu(dx)$ for all $\theta\in \bS^{d-1}$,
which is a constant,
 in the statement above. Here we call $\theta =(\theta_1, \cdots, \theta_d) \in \bS^{d-1}$ is rationally dependent if there is some non-zero
$m=(m_1, \cdots, m_d) \in
\Z^d$ so that $\<m, \theta\>= \sum_{i=1}^d m_i\theta_i=0$.
Otherwise, we call $\theta$ rationally independent. When $d=1$, $\bS^{0}=\{1, -1\}$ so every its member is rationally independent.
 In particular,  if  $\varrho_0$ does not charge on singletons
 when $d=2$ and does not charge on
 subsets of $\bS^{d-1}$ that are of Hausdorff dimension $d-2$ when $d\geq 3$
 (for example, $\varrho_0$ is $\gamma$-dimensional Hausdorff measure with    $\gamma\in (d-2,d-1]$),
 then  $\varrho_0$ does not charge on  the set of rationally dependent $\theta \in \bS^{d-1}$
and so the result above holds
with
$\bar k (\theta)\equiv \int_{\T^d}\int_{\T^d}k(x,z)\,dz\,\mu(dx)$.
Moreover, if $\varrho_0$ is absolutely continuous with respect to the Lebesgue surface measure $\sigma$  on $\bS^{d-1}$ with
a bounded
 Radon-Nikodym derivative, then
we can replace the joint continuity assumption on $k(x, z)$ by the continuity of  the function $x \mapsto k(x,z)$ and condition \eqref{e0}.

\item [(ii)] When $\alpha=2$, define
$$
Y_t^\e:=\e X_{\e^{-2}|\log {\e}|^{-1}t} - \e^{-1}|\log {\e}|^{-1}(\bar b_\infty +\bar b) t,\quad t\ge0.
$$ Suppose that
\begin{equation}\label{ex4-1-3}
k_0:=\lim_{|z|\rightarrow \infty}\int_{\T^d}k(x,z)\,\mu(dx)>0.
\end{equation}
Then $(Y_t^\e)_{t\ge 0}$ converges weakly in $\D([0, \infty); \R^d)$, as $\e \to 0$, to
Brownian motion
$(B_t)_{t\ge 0}$ with the covariance matrix
$A=\{a_{ij}\}_{1\le i,j\le d}$ such that
\begin{equation}\label{ex4-1-4}
a_{ij}:=k_0\int_{\bS^{d-1}}\theta_i\theta_j\,\varrho_0(d\theta).
\end{equation}

\item [(iii)] When $\alpha>2$,    define
$$
Y_t^\e:=\e X_{t/\e^{2}} - \e^{-1}(\bar b_\infty +\bar b)t,\quad t\ge0.
$$
Then $(Y_t^\e)_{t\ge 0}$ converges weakly
in $\D([0, \infty); \R^d)$, as $\e \to 0$,
to a
$d$-dimensional
Brownian motion $(B_t)_{t\ge 0}$ with the covariance matrix
\begin{equation}\label{ex4-1-2}
A:=\int_{\T^d}\int_{\R^d}\big(z+\psi(x+z)-\psi(x)\big)\otimes
\big(z+\psi(x+z)-\psi(x)\big)k(x,z)\,\Pi (dz) \,\mu(dx).
\end{equation}
Here $\psi\in \mathscr{D}(\sL)$ is
the unique periodic solution
on $\R^d$ to the following equation
$$
\sL \psi(x)=-b_\infty(x)-b(x)+\bar b_\infty+\bar b,\quad x\in \T^d
$$
with $\mu(\psi)=0$.
\end{itemize}
\end{example}
 One sufficient condition for  {\bf(A1)}, {\bf(A2)} and {\bf(A3)} to hold in this
example is that
$k(x,z)$ is bounded between two positive constants, and that
there
is a constant $\beta\in(0,1)$ so that
$b(x)=(b_i(x))_{1\le i\le d} \in C_b^\beta(\R^d)$,
$$
\sup_{z\in \R^d}|k(x,z)-k(y,z)|\le c_0|x-y|^{\beta},\quad x,y\in \R^d
$$ for some $c_0>0$, and
\[
\I_{\{|z|\le 1\}}\,\Pi (dz) =\I_{\{|z|\le 1\}}\frac{1}{|z|^{d+\alpha_0}}\,dz
\]
for some $\alpha_0\in (1,2)$ -- see Propositions \ref{l4-1} and \ref{l4-2}.

\medskip

Motivated by the
classical central limit theorem for stable laws (see e.g. \cite[p.\
161, Theorem 3.7.2; and p.\ 164, Exercise 3.7.2]{Du}), in order to
study the limit behavior of the scaled process
$X^\e:= (\e X_{t/\e^\alpha})_{t\ge0}$,
we do not need to recenter it
when $\alpha\in(0,1)$, but do need to
recenter it when $\alpha\in [1,2)$. Moreover, the
normalizing factors of the centered terms  are different  in
the critical case $\alpha=1$ and in the case $\alpha\in (1,2)$.
 See \cite[Theorem 2.4]{JKO} for  related discussions on
 limit theorems for additive functions of a Markov chain to
stable laws. A recent paper \cite{HDS} studied the periodic
homogenization for stable-like Feller processes  under the centering
condition on the drift term $b(x)$, see \cite[Assumption (H4)]{HDS}.
We
 emphasize that such centering condition
is commonly assumed in all the quoted papers above
except \cite{F2}
 which only considers symmetric $\alpha$-stable L\'evy noises with $\alpha\in (1,2)$.
 (For the periodic
homogenization for diffusion processes  under the centering
condition on the drift term $b(x)$, the reader is referred to \cite{BLP}.)
  In some
sense, studying  homogenization problem with general drift,  as
done in  \cite{Bh, Frank1} for diffusions, requires
normalizing
the center first, which
typically needs
much more effort. Concerning the assertion (iii) in Example \ref{ex4-1}, the closely
related works
 in \cite{PZ1, San} deal with
non-local operators with convolution-type kernels, and L\'evy type
operators without
drift terms,   respectively. The
critical case corresponds to the assertion (ii). Note that in the critical case,
in contrast to the diffusive case (iii),
the  scaling factor is $\e^{-2}|\log {\e}|^{-1}$ rather than the
standard diffusive scaling $\e^{-2}$
and the corrector solution does not contribute to the diffusion coefficient of the limiting Brownian motion.
 Moreover, in this case,
  because of \eqref{ex4-1-4},
 the limiting Brownian motion may be degenerate even
under the non-degeneracy
assumption
on $\I_{\{|z|\le 1\}}\Pi(dz)$, which is different from the
diffusive scaling case (iii) (see Remark \ref{r4-1} below).
 We should mention that all the limit processes have the scaling property; however, different from the cases (ii) and (iii), the limit processes considered in (i) is
 an $\alpha$-stable L\'evy process
which can be non-symmetric and singular, as
the spherical measure $\varrho_0(d\theta) $ can be
any finite measure on $\bS^{d-1}$.

\medskip

 Finally, we emphasize that the results of our paper can be regarded as the counterpart
in the jump-diffusion setting  of the work by \cite{Bh}, which studied periodic homogenization for diffusion processes without assuming the zero averaging condition on the drift term (that is, $\bar b=0$) which was imposed in \cite{BLP}.
However, since we will treat general jump processes
with a large class of scaling factors, there are
essential differences which require new ideas. For example, it always takes the diffusion scaling in \cite{Bh},
while as mentioned above in the present paper the
scaling and the limit process
in the non-diffusive cases are determined by the asymptotic behavior of the jumping measure $\Pi (dz)$ at infinity.
The  $\alpha$-stable scaling $\eps^{-\alpha}$ with $\alpha\in (0,2)$
in Example \ref{ex4-1} is merely a special case
 (see Example \ref{ex4-3}).
Furthermore, we do not assume  the  uniform ellipticity condition on the non-local operator $\sL$ of \eqref{e4-1}
in the sense that the support of the jumping   measure
  $\I_{\{|z|>1\}}\Pi (dz)$
can have zero Lebesgue measure and whose linear span
can be a proper  linear subspace of $\R^d$.

\medskip

The remainder of this paper is organized as follows.
In the next section,
we present an elementary lemma and some properties on the
scaled processes under Assumptions {\bf(A1)} and {\bf(A2)}. Sections \ref{section3} and \ref{section4} are devoted to the
study of the
limiting
behaviors of the scaled processes
under the jump scaling and the diffusive scaling, respectively. In Section \ref{section5}, we
consider  homogenization in the critical cases.
Roughly speaking, the
limiting  process is still Brownian motion,
but the scaling factor is
different from the standard diffusive scaling $\e^{-2}$.
In Section \ref{S:6}, sufficient conditions are given for the key averaging assumption \eqref{e:3.6} of our main results to hold, which are also of independent interest.
With all the results above at hand, in Subsection \ref{S:7.1}
we give the proof of the assertions made in Example \ref{ex4-1}.
Two more examples are given to illustrate the power of our main results.
Sufficient conditions for Assumptions {\bf (A1)}-{\bf (A3)}  to hold are presented in Subsection \ref{S:7.2}.

 \medskip

In this paper, we use := as a way of definition. For two positive
functions $f$ and $g$, $f\asymp g$ means that $f/g$ is bounded between two positive constants, and $f\preceq g$ means that $f/g$ is bounded by a
positive constant.
We use $[a]$ to denote the largest integer
not exceeding $a$.
For $a, b\in \R$, $a\wedge b:= \min \{a, b\}$.
For any vector $x,y\in \R^d$, $x\otimes y$ denotes its tensor
product, which is equivalent to an $d\times d$-matrix defined by $(x\otimes
y)_{ij}=x_iy_j$ for  $1\le i,j\le d$.

\section{Preliminaries}
\subsection{Elementary lemma}
\begin{lemma}\label{l2-1}
For $\psi\in C_b^1(\R^d;\R^d)$ and
$\varepsilon>0$, set $$\Phi_{\e}(x):=x+\e\psi(x/\e ),\,\,
\Theta_{\e}(x,z):=\psi( x/\e +z )-\psi(x/\e ),\quad x\in \R^d.$$ Then, for any $f\in C_b^3(\R^d)$ and $M\in (0,\infty]$,
\begin{align*}
&f(\Phi_{\e}(x+\e z))-
f(\Phi_{\e}(x))-
\langle\nabla (f(\Phi_{\e}(\cdot)))(x), \e z\rangle\I_{\{|z|\le M\}}\\
&=f(\Phi_{\e}(x)+\e z)-f(\Phi_{\e}(x))-\langle \nabla f(\Phi_{\e}(x)),
\e z\rangle\I_{\{|z|\le M\}}\\
&\quad+\e \langle \nabla f(\Phi_{\e}(x)), \psi (x/\e +z )
-\psi(x/\e )-\nabla\psi(x/\e )\cdot z\I_{\{|z|\le M\}}\rangle\\
&\quad+\frac{\e^2}{2} \langle \nabla^2 f(\Phi_{\e}(x)), (2\Theta_{\e}(x,z)\otimes  z)\I_{\{|z|\le M\}} +
\Theta_{\e}(x,z)\otimes \Theta_{\e}(x,z)\rangle+G_{\e}(x,z),
\end{align*}
where $G_{\e}(x,z)$ satisfies
\begin{align*}
|G_{\e}(x,z)| \le &C_1\e^3\|\nabla^3 f\|_\infty
(1+\|\psi\|_\infty+\|\nabla \psi\|_\infty)^3
(|z|^2\I_{\{|z|\le R\}}+|z|\I_{\{R<|z|\le M\}}+\I_{\{|z|> M\}})\\
&+C_1\e^2\|\nabla^2 f\|_\infty \|\psi\|_\infty|z|\I_{\{R<|z|\le M\}}+C_1 \|f\|_\infty\I_{\{|z|>M\}}
\end{align*}
for all $0<R\le M\le \infty$ and some constant $C_1>0$  independent of $\e$, $M$, $R$, $x$, $z$, $f$ and $\psi$.
\end{lemma}

\begin{proof}
For any $f\in C_b^3(\R^d)$ and $M\in(0,\infty]$, we write
\begin{align*}
&f(\Phi_{\e}(x+\e z))-
f(\Phi_{\e}(x))-
\langle\nabla (f(\Phi_{\e}(\cdot)))(x), \e z\rangle\I_{\{|z|\le M\}}\\
&=f(\Phi_{\e}(x+\e z))-
f(\Phi_{\e}(x))-
\langle\nabla f(\Phi_{\e}(x)), (\e z+\e \nabla\psi(x/\e )\cdot z)\rangle\I_{\{|z|\le M\}}\\
&=\Big[f(x+\e z+\e \psi (x/\e +z ))-f(x+\e \psi (x/\e +z ))-
\langle\nabla f(x+\e \psi (x/\e +z )), \e z\rangle\I_{\{|z|\le M\}}\Big]\\
&\quad+\Big[\langle \nabla f(x+\e \psi(x/\e +z))-
 \nabla f(x+\e \psi(x/\e )), \e z\rangle\I_{\{|z|\le M\}}\Big]\\
&\quad +\Big[f(x+\e \psi(x/\e +z))-f(x+\e \psi(x/\e ))-\langle \nabla f(x+\e \psi(x/\e )), \e\nabla \psi(x/\e )\cdot z\rangle \I_{\{|z|\le M\}}\Big] \\
&=:\sum_{i=1}^3 I_i^{\e}.
\end{align*}

First, for fixed $x,z\in \R^d$ and $\e\in (0,1)$, let
$$
H_1(y):=f(x+y+\e z)-f(x+y)-\langle \nabla f(x+y), \e z\rangle\I_{\{|z|\le M\}}.
$$
Then, $$\|H_1\|_\infty\preceq \e^2 \|\nabla^2 f\|_\infty|z|^2\I_{\{|z|\le M\}}+\|f\|_\infty\I_{\{|z|> M\}},$$ and, for any $0<R\le M$,
$$\|\nabla H_1\|_\infty \preceq \e^2\|\nabla^3 f\|_\infty  |z|^2\I_{\{|z|\le R\}}+
\e\|\nabla^2 f\|_\infty|z|\I_{\{|z|> R\}}.$$
Thus,
\begin{align*}
\quad I_1^\e &=H_1(\e\psi (x/\e +z))=H_1(\e\psi(x/\e ))+\big[H_1(\e\psi(x/\e +z))-
H_1(\e\psi(x/\e ))\big]\\
&=:f(\Phi_{\e}(x)+\e z)-f(\Phi_{\e}(x))-\langle \nabla f(\Phi_{\e}(x)),\e z\rangle\I_{\{|z|\le M\}}
+G_{1,\e}(x,z).
\end{align*}
Furthermore, according to the mean value theorem,  $G_{1,\e}(x,z)$ satisfies that
\begin{align*}
 |G_{1,\e}(x,z)|
 &\preceq \e
\|\nabla H_1\|_\infty\big|\psi(x/\e +z)-\psi(x/\e )\big| \I_{\{|z|\le M\}}+ \|H_1\|_\infty\I_{\{|z|> M\}}\\
&\preceq\e^3 \|\nabla^3 f\|_\infty\|\psi\|_\infty |z|^2\I_{\{|z|\le R\}}+ \e^2\|\nabla^2 f\|_\infty \|\psi\|_\infty |z|\I_{\{R<|z|\le M\}}+\|f\|_\infty \I_{\{|z|> M\}}.
\end{align*}

Second, by the Taylor expansion,
it holds that
\begin{align*}
I_2^\e&=\e^2 \big\langle \nabla^2 f (\Phi_{\e}(x) ), \Theta_{\e}(x,z)\otimes  z\big\rangle\I_{\{|z|\le M\}}\\
&\quad+\frac{\e^3}{2} \big\langle \nabla^3 f (\Phi_{\e}(x)+\theta_0\Theta_\e(x,z) ),
\Theta_{\e}(x,z)\otimes \Theta_\e(x,z)\otimes  z\big\rangle\I_{\{|z|\le M\}}\\
&=:\e^2 \big\langle \nabla^2 f (\Phi_{\e}(x) ), \Theta_{\e}(x,z)\otimes  z\big\rangle\I_{\{|z|\le M\}}
+G_{2,\e}(x,z),
\end{align*}
where $\theta_0\in(0,1)$ and
\begin{align*}
|G_{2,\e}(x,z)|\preceq \e^3\|\nabla^3 f\|_\infty
\big(\|\psi\|_\infty\|\nabla \psi\|_\infty|z|^2\I_{\{|z|\le R\}}+
\|\psi\|_\infty^2 |z| \I_{\{R<|z|\le M\}}\big)
\end{align*} for any $0<R\le M.$

Third,
applying the Taylor expansion again and using
the mean value theorem, we obtain
\begin{align*}
I_3^\e&=\e \langle \nabla f(\Phi_{\e}(x)), \psi(x/\e +z)
-\psi(x/\e )-\nabla\psi(x/\e )\cdot z\I_{\{|z|\le M\}}\rangle\\
&\quad +\frac{\e^2}{2} \big\langle \nabla^2 f (\Phi_{\e}(x)+\e \theta_1  \Theta_\e(x,z) ), \Theta_{\e}(x,z)\otimes \Theta_{\e}(x,z)\big\rangle\\
&=\e \langle \nabla f(\Phi_{\e}(x)), \psi(x/\e +z)
-\psi(x/\e )-\nabla\psi(x/\e )\cdot z\I_{\{|z|\le M\}}\rangle\\
&\quad + \frac{\e^2}{2} \big\langle \nabla^2 f (\Phi_{\e}(x) ), \Theta_{\e}(x,z)\otimes \Theta_{\e}(x,z)\big\rangle\\
&\quad + \frac{\e^3}{2} \theta_1 \big\langle \nabla^3 f (\Phi_{\e}(x)
+\e \theta_1\theta_2\Theta_\e(x,z) ), \Theta_{\e}(x,z)\otimes \Theta_{\e}(x,z)\otimes
\Theta_\e (x,z)\big\rangle\\
&=: \e \big\langle \nabla f (\Phi_{\e}(x) ), \psi(x/\e +z)
-\psi(x/\e )-\nabla\psi(x/\e )\cdot z\I_{\{|z|\le M\}} \big\rangle\\
&\quad+\frac{\e^2}{2} \big\langle \nabla^2 f (\Phi_{\e}(x) ), \Theta_{\e}(x,z)\otimes \Theta_{\e}(x,z)\big\rangle+
G_{3,\e}(x,z),
\end{align*}
where $\theta_1,\theta_2\in (0,1)$ and
\begin{align*}
|G_{3,\e}(x,z)| \preceq &\e^3\|\nabla^3 f\|_\infty
(\|\psi\|_\infty \|\nabla \psi\|^2_\infty|z|^2\I_{\{|z|\le R\}}+
\|\psi\|_\infty^2\|\nabla \psi\|_\infty|z| \I_{\{R<|z|\le M\}}+ \|\psi\|^3_\infty\I_{\{|z|>M\}})
\end{align*}
for any $0<R\le M$.

Therefore, putting all the estimates above together, we prove the required assertion.
\end{proof}

 \subsection{Consequences of {\bf (A1)} and {\bf(A2)}}
Let $\sL$ be the operator given by \eqref{e4-1}
with $\Pi (dz) $, $k(x, z)$ and $b(x)$ satisfying all the conditions below \eqref{e4-1}.
We assume that Assumptions {\bf (A1)} and {\bf (A2)} in Section \ref{section1}
also hold true. Denote by $X:=(X_t)_{t\ge0}$ the strong Markov
process associated with the generator $\sL $ as in {\bf(A1)}.

Let $\rho: \R_+\to \R_+$ be a strictly
increasing function such that $\lim_{r\to \infty }\rho(r)=\infty$.
For  $\e\in(0,1),$ consider the
scaled process
\begin{equation}\label{Sca}
X^{\e}:=\{ \e X_{\rho(1/\e)t}: \,  t\ge 0\}.
\end{equation}
Clearly, $X^\e$ is  a
strong Markov process on $\R^d$, and the associated generator is
given by
\begin{equation}\label{e:2.2}\begin{split}
\sL^\e
f(x)&= \rho(1/\e)
\int_{\R^d}\big(f(x+\e
z)-f(x)-\e\langle \nabla f(x), z\rangle \I_{\{|z|\le 1\}}\big)k(x/\e , z)\, \Pi (dz)
 \\
&\quad  +\e\rho(1/\e)  \langle  b (x/\e ) , \nabla f(x) \rangle.\end{split}
\end{equation}
See e.g.\ \cite[Lemma 2.1]{CCKW1}. Since the coefficients of the
generator $\sL^\e$ are
  multivariate $\e$-periodic,
the process $X^\e$
can be also
viewed as an
$\T^d$-valued process if
 $1/\e$ is an integer.

\begin{lemma}\label{l2-4}
Under Assumption ${\bf (A2)}$, we have the following two statements.
\begin{itemize}
\item[(i)]  For every
$f\in C_b(\T^d)$ with $\mu(f)=0$, any $0<s<t$ and $x\in \R^d$,
$$
\lim_{\e \to 0}\Ee_x \left[\left|\int_s^t
f\left( X_r^\e/\e\right)\,dr \right|^2\right]=0.
$$

\item[(ii)] Suppose that for some $x\in \R^d$, \begin{equation}\label{e:cond}\lim_{\e\to0}\sup_{|s-t|\le 1/[\sqrt{\rho(1/\e) }]}\Ee_x
(|X_s^\e-X_{t}^\e|\wedge1)=0.\end{equation} Then for any
bounded continuous function $F:\T^d\times \R^d \rightarrow \R$ and any $0<s<t$,
\begin{equation}\label{l2-4-1}
\lim_{\e \to 0}\Ee_x \left[\left|\int_s^t
F\left(X_r^\e/\e ,X_r^\e\right)\,dr-\int_s^t
\bar F(X_r^\e)\,dr\right|^2\right]=0,
\end{equation}
where $\bar F:\R^d\to \R$ is defined by $$\bar F(y):=\int_{\T^d}F(x,y)\,\mu(dx).$$
\end{itemize}
\end{lemma}

\begin{proof} (i)
Let $f\in C_b(\T^d)$ be such that $\mu(f)=0$. Then, for any $0<s<t$ and $x\in \R^d$,
by the Markov property and {\bf (A2)},
\begin{equation} \label{e:2.4}\begin{split}
\Ee_x \left[\left|\int_s^t f\left(\e^{-1}
 X_r^\e\right)\,dr\right|^2\right]
&= 2\Ee_x \left[\int_s^t\int_s^r
f (X_{\rho(1/\e)r} )f(X_{\rho(1/\e )u} )\,du\,dr\right] \\
&= 2\int_s^t\int_s^r \Ee_x \big[f
(X_{\rho(1/\e)u})\Ee_{X_{\rho(1/\e)u}}
 f (X_{\rho(1/\e)(r-u)} )\big]\,du\,dr \\
&\preceq  \frac{\|f\|^2_\infty }{\rho(1/\e)}\int_s^t\big(1-
e^{-\lambda_1\rho(1/\e)(r-s)}\big)\,dr \preceq \frac{\|f\|^2_\infty (t-s)}{\rho(1/\e)}.
\end{split}\end{equation}
This along with the fact that $\lim_{\e\to 0}\rho(1/\e)=\infty$ yields the first desired assertion.

(ii) By the standard approximation, it suffices to prove \eqref{l2-4-1} for
$F(x,y)=f(x)g(y)$ with $f\in C_b(\T^d)$ and $g\in C_b^1(\R^d)$.
Without loss of generality, we assume that $f\in C_b(\T^d)$ with $\mu(f)=0$.
For $\e >0$, define $s_i=s+ \frac{i(t-s)}{[\sqrt{\rho(1/\e) }]}$. Let $$
I_\e:=\Ee_x \left[\left|\int_s^t
f\left( X_r^\e/\e  \right)g(X_r^\e)\,dr\right|^2\right]
$$
and
$$
J_\e:=\Ee_x \left[\left|\sum_{i=0}^{[\sqrt{\rho(1/\e) }]-1}\int_{s_i}^{s_{i+1}} f
\left( X_r^\e/\e  \right)\,dr\cdot g(X_{s_i}^\e)\right|^2\right].
$$
 We can write
$$I_\e=\sum_{0\le i,j\le
[\sqrt{\rho(1/\e) }]-1}\int_{s_i}^{s_{i+1}}\int_{s_j}^{s_{j+1}}
\Ee_x\big[
f\left(X_r^\e/\e  \right)f\left( X_u^\e/\e  \right)g(X_r^\e)g(X_u^\e)\big]\,dr\,du$$
and
$$J_\e=\sum_{0\le i,j\le [\sqrt{\rho(1/\e) }]-1} \int_{s_i}^{s_{i+1}}\int_{s_j}^{s_{j+1}}
\Ee_x\big[
f\left( X_r^\e/\e  \right)f\left( X_u^\e/\e  \right)g(X_{s_i}^\e)g(X_{s_j}^\e)\big]\,dr\,du.$$
Note that, for every $0\le i,j\le [\sqrt{\rho(1/\e) }]-1$, $s_i\le r\le
s_{i+1}$ and $s_j\le u\le s_{j+1}$,
\begin{align*}
&\bigg|\Ee_x\big[
f\left( X_r^\e/\e  \right)f\left(X_u^\e/\e \right)g(X_r^\e)g(X_u^\e)]-\Ee_x\big[
f\left( X_r^\e/\e \right)f\left( X_u^\e/\e  \right)g(X_{s_i}^\e)g(X_{s_j}^\e)\big]\bigg|\\
& \le \|f\|_\infty^2\|g\|_\infty\left(\Ee_x
|g(X_r^\e)-g(X_{s_i}^\e)|+\Ee_x
|g(X_u^\e)-g(X_{s_j}^\e)|\right)\\
&\le 2 \|f\|_\infty^2\|g\|_\infty (\|g\|_\infty+\|\nabla g\|_\infty)
\left(\Ee_x (|X_r^\e-X_{s_i}^\e|\wedge1)+\Ee_x
(|X_u^\e-X_{s_j}^\e|\wedge1)\right)\\
&\le c_1 \eta(\e),
\end{align*} where $c_1 = 4 (\|g\|_\infty+\|\nabla g\|_\infty) \|f\|_\infty^2\|g\|_\infty$ and
$$
\eta(\e):=\sup_{|r_1-r_2|\le (t-s)/[\sqrt{\rho(1/\e) }]}\Ee_x
(|X_{r_1}^\e-X_{r_2}^\e|\wedge1).
$$
Thus,
$$|I_\e-J_\e|\le c_1 \eta(\e) \sum_{0\le i,j\le [\sqrt{\rho(1/\e) }]-1}
(s_{j+1}-s_j)(s_{i+1}-s_i)\le
c_1 \eta (\e) (t-s)^2.
$$
Hence $\lim_{\e\to0} | I_\e-J_\e|=0$. So it remains to show that $\lim_{\e\to 0} J_\e =0$.

By the Cauchy-Schwarz inequality and \eqref{e:2.4},
\begin{align*}J_\e
&\le [\sqrt{\rho(1/\e) }] \|g\|_\infty^2
\sum_{i=0}^{[\sqrt{\rho(1/\e) }]-1}\Ee_x
\left[\left|\int_{s_i}^{s_{i+1}}f\left(X_r^\e/\e  \right)\,dr\right|^2
\right]\\
&\le  \frac{c_2 [\sqrt{\rho(1/\e) }]}{\rho(1/\e) }
\|g\|_\infty^2\|f\|_\infty^2 \sum_{i=0}^{[\sqrt{\rho(1/\e) }]-1}
(s_{i+1}-s_i)\le c_2\|f\|_\infty^2\|g\|_\infty^2\frac{(t-s)}{\sqrt{\rho(1/\e) }}.
\end{align*}
Since $\lim_{\e\to 0}\rho(1/\e) =\infty$, we get $\lim_{\e\to0} J_\e=0$.
\end{proof}

\section{Homogenization:
jump scalings}\label{section3}
Let $\bS^{d-1}:=\{x\in \R^d: |x|=1\}$ be the unit sphere on $\R^d$, and
$z:=(r,\theta)\in \R_+\times \bS^{d-1}$ be the spherical coordinate of $z\in \R^d\backslash\{0\}$.
Throughout this section, we assume that the jumping measure $\Pi (dz) $ in \eqref{e4-1} has the following form on $\{|z|>1\}$:
\begin{equation}\label{e:3.3a}
\I_{\{|z|>1\}}\Pi (dz) :=\I_{\{r>1\}}\varrho(r,d\theta)\,dr=\I_{\{r>1\}}\frac{\varrho_0\,(d\theta)+\kappa(r,d\theta)}{r\varphi(r)}\,dr,
\end{equation} where
\begin{itemize}
\item[(i)]  $\varrho_0(d\theta)$ is a  e
non-negative   finite
 measure on $\bS^{d-1}$ such that $\varrho_0
(\bS^{d-1})>0$;  \item[(ii)] for every $r>1$, $\kappa(r,d\theta)$ is a finite signed measure on  $\bS^{d-1}$
so that for any $r_0>1$,
\begin{equation}\label{e:3.4}
\sup_{r\in [r_0,\infty)}
|\kappa|(r,\bS^{d-1})<\infty,\quad \lim_{r\to \infty} |\kappa|(r,\bS^{d-1}) =0,
\end{equation}
where, for each $r>1$,
 $|\kappa|(r,d\theta)$ denotes the total variational measure of $\kappa(r,d\theta)$;
\item[(iii)] $\varphi:(1,\infty)\rightarrow \R_+$ is a strictly
increasing function such that there are constants $\alpha\in (0, 2)$, $c_0\in (0, 1]$,
and $\eta_0 \in (0, \alpha \wedge |\alpha -1|\wedge (2-\alpha))$ when $\alpha \not=1$ and $\eta_0 \in (0, 1/6)$ when
$\alpha =1$,
so that for any
$1< r\le R$,
\begin{equation}\label{e:3.1}
\lim_{\lambda \to \infty}
\Big|\frac{\varphi(\lambda r)}{\varphi (\lambda)}-r^\alpha\Big|=0, \quad c_0 (R/r)^{\alpha -\eta_0}\leq
\frac{\varphi (R)}{\varphi (r)} \leq c_0^{-1} (R/r)^{\alpha +\eta_0}.
\end{equation}
\end{itemize}

Define
\begin{equation}\label{e:3.2a}
\Pi_0 (dz)
:=\I_{\{r>0\}}
\frac{\varrho_0(d\theta)\,dr}{r^{1+\alpha}},
\end{equation}
where $\alpha\in (0,2)$ is the constant in \eqref{e:3.1}.
It is obvious that $\Pi_0 (dz) $ satisfies the scaling property that
$\Pi_0(sA)=s^{-\alpha}a_0(A)$ for all $s>0$ and $A\subset \R^d\backslash\{0\}$; however, since $\varrho_0(d\theta)$ may be non-symmetric on $\bS^{d-1}$,
$\Pi_0(dz)$ can be non-symmetric on
$\R^d$.

We also suppose that for the function $k(x,z)$ in \eqref{e4-1}, there exists a bounded function $\bar k:\R^d\times\R^d \to \R_+$ such that
for any $0<r<R$
and $f:\R^d\times \R^d\times\R^d \rightarrow \R$ satisfying
\begin{equation}\label{e:3.4b}
\lim_{\e \to 0}\sup_{x\in \R^d, |z_1-z_2|\le \e}|f(x,z_1)-f(x,z_2)|=0,
\end{equation}
it holds that
\begin{equation}\label{e:3.6}
 \lim_{\e \to 0}\sup_{x\in \R^d}\Big| \int_{\{r\le |z|\le R\}}f(x,z)
k(x/\e,z/\e)\,\Pi_0 (dz) -
\int_{\{r\le |z|\le R\}}f(x,z)\bar k(x/\e,z)\,\Pi_0 (dz) \Big|=0.
\end{equation}

\begin{remark}\label{R:3.1} \rm We make some comments on the assumptions above.
In the following, $\varphi$ is a function given in \eqref{e:3.1}.

\begin{itemize}
\item[(1)] Examples of functions satisfying \eqref{e:3.1} include $\varphi (r) =r^\alpha + r^\beta$ on $(1, \infty)$ for $0<\beta \le\alpha<2$
and $\varphi (r)=r^\alpha \log (1+r)$ on $(1, \infty)$ for $0<\alpha<2$.
In fact any strictly increasing function $\varphi(r)$ on $(1,\infty)$ of the form  $\int_{\alpha_1}^{\alpha_2} r^\beta \,\nu (d\beta)$ satisfies
condition \eqref{e:3.1}, where $0<\alpha_1\leq \alpha_2<2$, and $\nu$ is a finite measure
on $[\alpha_1, \alpha_2]$ so that $\alpha_2$ is in the support of $\nu$; see Example \ref{ex4-3}
for the proof of this fact.
Observe that the second condition in \eqref{e:3.1} is equivalent to
that there is some $R_0\geq 1$ so that for all $R\ge r> R_0$,
$$
c_0 (R/r)^{\alpha -\eta_0}\leq
\frac{\varphi (R)}{\varphi (r)} \leq c_0^{-1} (R/r)^{\alpha +\eta_0}.
$$ Indeed, the statements in this section still holds if we replace
$r>1$ in \eqref{e:3.3a} by $r>R_0$ for some $R_0\ge1$, and restrict $\varphi$ defined on $(R_0,\infty)$.

\item[(2)] Condition \eqref{e:3.4} means that the term
$\kappa(r, \cdot)$ is a lower order perturbation as $r\to \infty$,
and thus, by \eqref{e:3.3a}, the jumping measure $\Pi (dz) $ is comparable
to $\frac{\varrho_0(d\theta)\,dr}{r \varphi (r)}$ for large  $|z|$.

\item[(3)] Suppose that $\varphi_1(r)$ is a strictly increasing function on $(1,\infty)$ so that
\begin{equation}\label{e:3.5}
\lim_{r\to \infty} \frac{\varphi_1 (r)}{\varphi (r)}=1.
\end{equation}
Then $\varphi_1(r)$ clearly satisfies \eqref{e:3.1}, and we can rewrite $a(z)$ as
$$
\I_{\{|z|>1\}}\Pi (dz) =\I_{\{r>1\}}\frac{\varrho_0(d\theta)+\wt \kappa(r,d\theta)}{r\varphi_1(r)}\,dr
$$
with
$$
\wt \kappa (r,d\theta)= \left( \frac{\varphi_1 (r)}{\varphi (r)}-1\right) \varrho_0 (d\theta)
+ \frac{\varphi_1 (r)}{\varphi (r)} \kappa (r,d\theta).
$$
Evidently,
$\lim_{r\to \infty}|\wt \kappa|(r,\bS^{d-1})=0$ and $\sup_{r\in [r_0,\infty)}|\wt \kappa|(r,\bS^{d-1})<\infty$ for all $r_0>1$.
In other words, the representation of $\Pi (dz) $ in
the form of \eqref{e:3.3a} is invariant among the family of strictly increasing functions $\varphi$ on $(1,\infty)$
that mutually satisfy the relation \eqref{e:3.5}.

\item[(4)] Let $k:\R^d\times \R^d\to \R_+$ be the function satisfying the assumptions below \eqref{e4-1}.
 In view of \eqref{e:3.2a}, it is easy to see that  if
for every $x\in \R^d$ and $\varrho_0$-a.e.\ $\theta \in \bS^{d-1}$, there is a constant  $\bar k(x, \theta)$ so that
\begin{equation}\label{e:3.8a}
    \lim_{T\to \infty} \frac1T \int_0^{T}  k(x, (r, \theta))  \,dr =  \bar k(x, \theta),
\end{equation}
then   \eqref{e:3.6} holds with $\bar k(x,z) :=\bar k (x,z/|z|)$.  Clearly,   condition \eqref{e:3.8a} holds, if for $\varrho_0$-a.e.\ $\theta \in \bS^{d-1}$, there exists a bounded measurable function $\bar k (\cdot, \theta):\R^d \to \R_+$ such that
 $$
 \lim_{\e \to 0}\sup_{x\in \R^d}\left|k\left(x,(r/\e,\theta)\right)-\bar k (x,\theta)\right|=0.
 $$
 It is shown in Theorem  \ref{T:6.5}  below that condition \eqref{e:3.8a} holds for every $\theta \in \bS^{d-1}$  and so
 condition \eqref{e:3.6} automatically holds
 for any finite measure $\varrho_0$ on $\bS^{d-1}$,  when
 $k(x,z)$ is bounded, continuous and multivariate 1-periodic
 in both $x$ and $z$.
See  Section \ref{S:6}
 for more sufficient conditions for \eqref{e:3.6} including the information on
$\bar k(x, \theta)$,  when $z\mapsto k(x,z)$ is multivariate 1-periodic for any fixed $x\in \R^d$.
\end{itemize}
\end{remark}

The purpose of this section is to consider the
limiting
 behavior of the
scaled process
\begin{equation}\label{e:scaled}
X^\e=(X_t^\e)_{t\ge 0}:=(\e X_{\varphi(1/\e)t})_{t\ge 0}.
\end{equation}
Note that by \eqref{e4-1} and \eqref{e:2.2}, the
generator of $X^\e$ is given by
\begin{equation}\label{e:3.7}\begin{split} \sL^\e
f(x)&= \varphi(1/\e)
\int_{\R^d}\big(f(x+\e
z)-f(x)-\e\langle \nabla f(x), z\rangle \I_{\{|z|\le 1\}}\big)k(x/\e , z)\, \Pi (dz)
 \\
&\quad  +\e\varphi(1/\e)  \langle  b (x/\e ) , \nabla f(x) \rangle.
\end{split}\end{equation}

It turns out that the
limiting behavior of $X^\e$ as $\e\to0$ depends
on
the value of $\alpha$ associated with the scaling function $\varphi$ in \eqref{e:3.1}.
We will divide this
section into two parts. One is to consider
the invariance principle
for $X^\e$ that needs no recentering, and the other that requires recentering. In some literature,
invariance principle
that requires recentering is called non-central limit theorem; see for instance \cite{F2}.

\subsection{Invariance principle
without recentering: $\alpha\in (0,1)$}
Recall that $\alpha \in (0, 2)$ is the constant in \eqref{e:3.1}. In this subsection, we will restrict ourselves to the case $\alpha\in (0,1)$. Then,
by \eqref{e:3.3a}, we have
\begin{equation}\label{e:NO}
\begin{split}
&\lim_{\e\to0} \e\varphi(\e^{-1})=0,\quad \lim_{\delta\to0}\limsup_{\e\to0}
\left( \e\varphi(\e^{-1})\int_1^{\delta/\e}\frac{1}{\varphi(r)}\,dr \right) =0, \\
&\lim_{\delta\to0}\sup_{\e\in (0,1)}
\left( \delta\e\varphi(\e^{-1})\int_1^{1/(\delta\e)}\frac{1}{\varphi(r)}\,dr \right)=0,\quad
 \lim_{\delta\to0}\sup_{\e\in(0,1)}\left(\varphi(\e^{-1})\int_{1/(\delta\e)}^\infty\frac{1}{r\varphi(r)}\,dr\right)=0,\\
&\sup_{\e\in (0,1)} \left(\e \varphi(\e^{-1})\int_1^{1/\e}\frac{1}{\varphi(r)}\,dr\right)<\infty,
\quad \sup_{\e\in (0,1)}\left(\varphi(1/\e)\int_{1/\e}^\infty\frac{1}{r\varphi(r)}\,dr\right)<\infty.
\end{split}
\end{equation}
In particular,
$$
\lim_{\delta\to0}\sup_{\e\in (0,1)}
\left(\delta^2\e^2\varphi(\e^{-1})\int_1^{1/(\delta\e)}\frac{r}{\varphi(r)}\,dr\right)
\le\lim_{\delta\to0}\sup_{\e\in (0,1)}\left(\delta\e\varphi(\e^{-1})\int_1^{1/(\delta\e)}\frac{1}{\varphi(r)}\,dr\right)=0.
$$
In fact, for any $\e,\delta\in (0,1)$,
\begin{align*}
\delta\e\varphi(\e^{-1})\int_1^{1/(\delta\e)}\frac{1}{\varphi(r)}\,dr&=
\delta \int_{\e}^{1/\delta}\frac{\varphi({1}/{\e})}{\varphi({r}/{\e})}\,dr\\
&\preceq \delta\left(\int_{\e}^1 \frac{1}{r^{\alpha+\eta_0}}\,dr+\int_1^{1/\delta}\frac{1}{r^{\alpha-\eta_0}}\,dr\right)
\preceq \delta\left(1+{\delta}^{-1-\eta_0+\alpha}\right),
\end{align*}
where we have used the second condition of \eqref{e:3.1} in the first inequality. So, by
the fact $\eta_0\in (0,\alpha\wedge (1-\alpha))$, we
obtain
$$
\lim_{\delta\to0}\sup_{\e\in (0,1)}\left(\delta\e\varphi(\e^{-1})\int_1^{1/(\delta\e)}\frac{1}{\varphi(r)}\,dr\right)=0.
$$
Other estimates in \eqref{e:NO} can be proved similarly
and we omit  the
details.

\medskip

The following is the main result of this section.

\begin{theorem}\label{t3-1}
Suppose that \eqref{e:3.3a} and \eqref{e:3.6} hold.
If \eqref{e:3.1} holds with $\alpha \in (0, 1)$,
then the scaled process $(X_t^\e)_{t\ge 0}$ of \eqref{e:scaled}
converges weakly
in $\D ([0, \infty); \R^d)$,
as $\e \to 0$, to an $\alpha$-stable L\'evy process
$\bar X:=(\bar
X_t)_{t\ge0}$ with L\'evy measure
$\bar k_0(z) \,\Pi_0 (dz) $;
that is, the generator of the
$\alpha$-stable L\'evy process $\bar X$ is given by
$$\bar \sL f(x)=
\int_{\R^d} (f(x+z)-f(x))\bar k_0(z)\,\Pi_0 (dz) .
$$
Here  $\Pi_0 (dz) $ is the  measure  defined
in
\eqref{e:3.2a}  and  $\bar k_0(z):=\int_{\T^d}\bar k(x,z)\,\mu(dx)$,
 where  $\mu$ is the unique invariant probability  measure of $X$ on $\T^d$, and
 $\bar k(x,z)$ is the function in \eqref{e:3.6}.
\end{theorem}

To prove this theorem,  we need the following property for the
generator of the scaled process $X^\e$.

\begin{lemma}\label{l2-2}
Suppose that \eqref{e:3.3a} and \eqref{e:3.6} hold, and that $0<\alpha <1$. For
   every $f\in C_b^2(\R^d)$,
$$
\lim_{\e \to 0}\sup_{x\in \R^d}| \sL^\e f(x)-\bar \sL^\e f(x)|=0,
$$where
\begin{equation}\label{l2-2-1a}
\bar \sL^\e f(x):=\int_{\R^d}\big(f(x+z)-f(x)\big)\bar k (x/\e,z) \,\Pi_0 (dz)
\end{equation}
with $\bar k(x,z)$ being the function in \eqref{e:3.6} and $\Pi_0 (dz) $ defined by \eqref{e:3.2a}.
\end{lemma}

\begin{proof} By \eqref{e:3.7}, for every $\e, \delta\in (0,1)$ and $f\in C_b^2(\R^d)$,
\begin{align*}
 \sL^\e f(x)=&\varphi(1/\e)\int_{\{|z|\le {\delta}/{\e}\}}
\big(f(x+\e z)-f(x)-\langle \nabla f(x), \e z\rangle\big)k(x/\e , z)\,\Pi (dz) \\
&+\varphi(1/\e)\int_{\{{\delta}/{\e}<|z|< {1}/({\delta\e})\}}
\big(f(x+\e z)-f(x)\big)k(x/\e , z)\, \Pi (dz) \\
& +\varphi(1/\e)\int_{\{|z|\ge {1}/({\delta\e})\}}
\big(f(x+\e z)-f(x)\big)k(x/\e , z)\, \Pi (dz) +\e\varphi(1/\e)\left\langle \nabla f(x), b_{ \delta/\e }(x/\e )
+b(x/\e )\right\rangle\\
=&:\sum_{i=1}^4 \sL_{i}^{\e,\delta}f(x),
\end{align*}
where $b_{\delta/\e}(x)$ is defined by \eqref{e:1.2}.
We can write
\begin{align*}
\bar  \sL^\e f(x)=&\int_{\{|z|\le \delta\}}
\big(f(x+z)-f(x)\big)\bar k(x/\e ,z)\,\Pi_0 (dz) \\
&+\int_{\{\delta<|z|< 1/\delta \}}
\big(f(x+z)-f(x)\big)\bar k(x/\e ,z)\,\Pi_0 (dz) +\int_{\{|z|\ge 1/\delta \}}
\big(f(x+z)-f(x)\big)\bar k(x/\e ,z)\,\Pi_0 (dz) \\
=&:\sum_{i=1}^3 \bar \sL_{i}^{\e,\delta}f(x).
\end{align*}

Since $\bar k(x,z)$ is bounded and $\alpha\in (0,1)$, by \eqref{e:3.2a} we have
\begin{align*}
|\bar \sL_{1}^{\e,\delta}f(x)|&\preceq \|\nabla f\|_\infty\int_{\{|z|\le \delta\}}|z|\,\Pi_0 (dz)
\preceq \|\nabla f\|_\infty\int_0^\delta \frac{\varrho_0(\bS^{d-1})}{r^\alpha}\,dr\preceq \|\nabla f\|_\infty\delta^{1-\alpha}.
\end{align*}
Applying the same argument to $|\bar \sL_{3}^{\e,\delta}f(x)|$, we see that
$$
\lim_{\delta \to 0} \sup_{\e\in (0,1], x\in \R^d}\big(|\bar \sL_{1}^{\e,\delta}f(x)|
+|\bar \sL_{3}^{\e,\delta}f(x)|\big)=0.$$

On the other hand, according to \eqref{e:3.3a},
\begin{equation}\label{e:3.8}
\I_{\{|z|\ge 1\}}\,\Pi (dz) \le\frac{\varrho_0(d\theta)+|\kappa|(r,d\theta)}{r\varphi(r)}\I_{\{r\ge 1\}}\,dr,
\end{equation}
and so by \eqref{e:3.4}  we obtain that for $0<\e <\delta<1$,
\begin{align*}
\sup_{x\in \R^d}|
\sL_1^{\e,\delta}f(x)|
&\preceq \|\nabla^2 f\|_\infty\e^2\varphi(1/\e)\int_{\{|z|\le {\delta}/{\e}\}}|z|^2\,\Pi (dz) \\
&\preceq \|\nabla^2 f\|_\infty  \e^2\varphi(1/\e)
\left(
\int_{\{|z|\le  1\}}|z|^2\,\Pi (dz)
+\int_{1}^{{\delta}/{\e}}\frac{r}{\varphi(r)}\,dr\right)\\
&\preceq \|\nabla^2 f\|_\infty
\left( \e^2\varphi(1/\e)
\int_{\{|z|\le  1\}}|z|^2\,\Pi (dz)
+ \delta\e\varphi(1/\e)\int_{1}^{{\delta}/{\e}}\frac{1}{\varphi(r)}\,dr\right),
\end{align*}
\begin{align*}
\sup_{x\in \R^d}|\sL_3^{\e,\delta}f(x)|
& \preceq
\|f\|_\infty \varphi(1/\e)\int_{\{|z|\ge1/({\delta \e})\}}\,\Pi (dz)
\preceq \|f\|_\infty \varphi(1/\e)\int_{1/({\delta \e})}^\infty \frac{1}{r\varphi(r)}\,dr,
\end{align*}
and
$$
\sup_{x\in \R^d}\left|b_{ \delta/\e }(x/\e )\right|\preceq
 \int_{1}^{{\delta}/{\e}}\frac{1}{\varphi(r)}\,dr.
$$
Thus, by \eqref{e:NO} and
the fact that
$b(x)$ is bounded,
 we have
$$
\lim_{\delta \to 0}\limsup_{\e\to0}\sup_{x\in
\R^d}\big(|\sL_{1}^{\e,\delta}f(x)|+|\sL_{3}^{\e,\delta}f(x)|
+|\sL_4^{\e,\delta} f(x)|\big)=0.
$$

Furthermore, due to  \eqref{e:3.3a},
\begin{align*}
\sL_{2}^{\e,\delta}f(x) =&
\varphi(1/\e)
\int_{\{\delta<|z|<1/\delta\}}
\big(f(x+z)-f(x) \big)k( x/\e,{z/\e} )\,
a(d(z/\e))\\
=& \varphi(1/\e)
\int_{\delta}^{1/\delta}\int_{\bS^{d-1}}
\big(f(x+z)-f(x) \big)k( x/\e,{z/\e} )
\frac{\varrho_0(d\theta)+\kappa(r/\e,d\theta)}{r/\e \varphi(r/\e)}\,d(r/\e)\\
 =&\int_{\delta}^{1/\delta}\int_{\bS^{d-1}}
\big(f(x+z)-f(x) \big)k( x/\e ,{z/\e} )
\frac{\varphi(1/\e)}{r\varphi(r/\e)}\,{\varrho_0(d\theta)}\,dr\\
 &+\int_{\delta}^{1/\delta}\int_{\bS^{d-1}}
\big(f(x+z)-f(x) \big)k( x/\e ,{z/\e} )
 \frac{\varphi(1/\e)}{r\varphi(r/\e)}\,{\kappa(r/\e,d\theta)}\,dr\\
=&:\sL_{2,1}^{\e,\delta}f(x)+\sL_{2,2}^{\e,\delta}f(x).
\end{align*}
By \eqref{e:3.4} and \eqref{e:3.1}
as well as the dominated convergence theorem, we know that for every fixed $\delta\in (0,1)$,
$$
\lim_{\e \to 0}\sup_{x\in \R^d}
|\sL_{2,2}^{\e,\delta}f(x)|=0.
$$
Again by (the first condition in) \eqref{e:3.1}, \eqref{e:3.2a} and \eqref{e:3.6} as well as the dominated convergence theorem,
one can verify that for every fixed $\delta\in (0,1)$,
\begin{equation}\label{referen-1}
\lim_{\e \to 0}\sup_{x\in \R^d}
|\sL_{2,1}^{\e,\delta}f(x)-\bar \sL_{2}^{\e,\delta}f(x)|=0.
\end{equation}

Putting all the estimates together, and letting $\e \to
0$ and then $\delta \to 0$, we get the desired assertion.
\end{proof}

\begin{proof} [Proof of Theorem $\ref{t3-1}$]
(1) Recall that $\sL^\e$ is the infinitesimal generator for the Markov process
$X^\e:=\big((X_t^\e)_{t\ge 0}; (\Pp_x)_{x\in \R^d}\big)$.
For every $x\in \R^d$, $t>0$, $f \in C_b^2(\R^d)$ and
stopping time $\tau$,
\begin{equation}\label{t3-1-1a}
\Ee_x f(X_{t\wedge \tau}^\e) =f(x)+\Ee_x\left[\int_0^{t\wedge \tau}
\sL^\e f(X_s^\e)\,ds\right].
\end{equation}
For any $R>1$, we write
\begin{align*}
\sL^\e f(x)=&\varphi(1/\e)\int_{\{|z|\le {R}/{\e}\}}
\big(f(x+\e z)-f(x)-\langle\nabla f(x), \e z \rangle\big)k(x/\e , z)\,\Pi (dz) \\
&+\varphi(1/\e)\int_{\{|z|> {R}/{\e}\}}
\big(f(x+\e z)-f(x)\Big)k(x/\e , z)\,\Pi (dz) \\
&+\e\varphi(1/\e)\Big\langle \nabla f(x),
b_{R/\e} (x/\e )+b(x/\e )
\Big\rangle\\
=&:\sum_{i=1}^3 I_i^{\e,R}(x),
\end{align*}
where $b_{R/\e}(x)$ is defined by \eqref{e:1.2}. Using \eqref{e:3.8} and following the proof of Lemma \ref{l2-2}, we have
\begin{align*}
&\sup_{x\in \R^d}|I_1^{\e,R}(x)|\preceq \|\nabla^2 f\|_\infty\e^2 \varphi(1/\e)
\left(\int_{\{|z|\le 1\}}|z|^2 \,\Pi (dz)
+\int_1^{{R}/{\e}}\frac{r}{\varphi(r)}\,dr\right),\\
 &\sup_{x\in \R^d}|I_2^{\e,R}(x)|\preceq \|f\|_\infty\varphi(1/\e)
\int_{{R}/{\e}}^\infty\frac{1}{r\varphi(r)}\,dr,\\
& \sup_{x\in \R^d}|I_3^{\e,R}(x)|\preceq \|\nabla f\|_\infty\e
\varphi(1/\e) \left(
\|b\|_\infty
+\int_1^{{R}/{\e}}\frac{1}{\varphi(r)}\,dr\right).
\end{align*}
Hence, for every $R>1$,
\begin{equation}\label{t3-1-2}\begin{split}
 \sup_{x\in \R^d}|\sL^\e f(x)|
\preceq
& \|\nabla^2 f\|_\infty\e^2 \varphi(1/\e)
\left(1+\int_1^{{R}/{\e}}
\frac{r}{\varphi(r)}\,dr\right)\\
& +\|f\|_\infty\varphi(1/\e)
\int_{{R}/{\e}}^\infty\frac{1}{r\varphi(r)}\,dr +\|\nabla f\|_\infty\e
\varphi(1/\e)\left(1+\int_{1}^{{R}/{\e}}
\frac{1}{\varphi(r)}\,dr\right).\end{split}
\end{equation}

(2) In the following, for every $l>0$, let $f_l\in C_b^3(\R^d)$ be
such that
\begin{equation}\label{e:funcOO}f_l(x)=\begin{cases} 0&\quad |x|\le {l}/{2},\\
1&\quad |x|\ge l,\end{cases}\end{equation} and
$\|\nabla^{i} f_l\|_\infty\preceq
 l^{-i}$ for $0\le i\le 3$.
For any fixed $y\in \R^d$, we set $f_l^y(x):=f_l(x-y)$. Then, according to \eqref{t3-1-2}, we have
\begin{align*}
\sup_{x,y\in \R^d}|\sL^\e f_R^y(x)|\preceq&
 R^{-2}\e^2 \varphi(1/\e)\left(1+\int_1^{{R}/{\e}}
 \frac{r}{\varphi(r)}\,dr\right)
+\varphi(1/\e)
\int_{{R}/{\e}}^\infty\frac{1}{r\varphi(r)}\,dr\\
&+R^{-1}\e \varphi(1/\e)\left(1+\int_{1}^{{R}/{\e}}
\frac{1}{\varphi(r)}\,dr\right).
\end{align*}
This along with \eqref{t3-1-1a} yields that for any $T>0$,
\begin{align*}
 \Pp_0 \left(\sup_{t\in [0,T]}|X_t^\e|>R \right) &\preceq
\Ee [f_R(X^\e_{T\wedge \tau_R^\e})]=\Ee \left[\int_{0}^{T\wedge \tau_R^\e}\sL^\e f_R(X_s^\e)\,ds\right]\\
&\preceq T \Bigg[R^{-2}\e^2
\varphi(1/\e)\left(1+\int_1^{{R}/{\e}}\frac{r}{\varphi(r)}\,dr\right)
 +\varphi(1/\e)
\int_{{R}/{\e}}^\infty\frac{1}{r\varphi(r)}\,dr\\
&\qquad +R^{-1}\e
\varphi(1/\e)\left(1+\int_{1}^ {{R}/{\e}}
\frac{1}{\varphi(r)}\,dr\right)\Bigg].
\end{align*}
Here and in what follows,  $\tau_l^\e :=\inf\{t>0:
|X_t^\e-X_0^\e|>l\}$.
Hence, according to \eqref{e:NO}, we have
\begin{equation}\label{t3-1-3}
\lim_{R \to \infty}\sup_{\e\in (0,1)}\Pp_0 \Big(\sup_{t\in
[0,T]}|X_t^\e|>R \Big) =0.
\end{equation}
On the other hand, following the argument
in \eqref{t3-1-2},
we can obtain that for every $\theta\in (0,1)$ and $y\in \R^d$,
\begin{align*}
\sup_{x\in \R^d}|\sL^\e f_\theta^y(x)| \preceq &\|\nabla^2
f_\theta^y\|_\infty
\e^2\varphi(1/\e)\left(1+\int_{1}^{{\theta}/{\e}}\frac{r}{\varphi(r)}\,dr\right)+
\|f_\theta^y\|_\infty\varphi(1/\e)\int_{{\theta}/{\e}}^\infty\frac{1}{r\varphi(r)}\,dr\\
&+\|\nabla
f_\theta^y\|_\infty\e\varphi(1/\e)\left(1+\int_{1}^{{\theta}/{\e}}\frac{1}{\varphi(r)}\,dr\right)\\
\preceq & \theta^{-2}\e \varphi(1/\e)+\varphi(1/\e)\int_{{\theta}/{\e}}^\infty\frac{1}{r\varphi(r)}\,dr
+ \theta^{-1}\e\varphi(1/\e)\int_{1}^{{\theta}/{\e}}\frac{1}{\varphi(r)}\,dr.
\end{align*}
It follows from \eqref{e:NO} that for
every $\theta\in (0,1)$,
$$
\sup_{\e\in (0,1)}\sup_{x,y\in \R^d}|\sL^\e f_\theta^y(x)|\le
C(\theta)<\infty.
$$
Therefore, for any stopping time $\tau$ with $\tau\le T$ and any
positive constant $\delta(\e)$ with $\lim_{\e \to 0}\delta(\e)=0$,
\begin{align*}
\Pp_0\big(|X_{\tau+\delta(\e)}^\e-X_{\tau}^\e|>\theta\big)&=
\Ee_0 \big[\Pp_{X_\tau^\e}\big(|X_{\delta(\e)}^\e-X_0^\e|>\theta\big)\big]\le \Ee_0 \big(\Ee_{X_{\tau}^\e}\
f_\theta\big(X_{\tau_\theta^\e\wedge
\delta(\e)}^\e\big) \big)\\
&=\Ee_0 \left[\Ee_{X_{\tau}^\e}\left(\int_0^{\tau_\theta^\e\wedge
\delta(\e)}\sL^\e f_\theta(X_s^\e)\,ds\right)\right] \le
C(\theta)\delta(\e),
\end{align*}
which implies
\begin{equation}\label{t3-1-4}
\lim_{\e \to 0}\Pp_0\big(|X_{\tau+\delta(\e)}^\e-X_{\tau}^\e|>\theta\big)=0.
\end{equation}
Due to \eqref{t3-1-3} and \eqref{t3-1-4} (see e.g.\ \cite[Theorem
1]{A}), we conclude that $\{X^\e\}_{\{\e\in(0,1]\}}$ is
tight
in $\D([0, \infty); \R^d)$.

(3) By (2), for any sequence $\{X^{\e_n}\}_{n\ge1}$
with $\lim_{n \rightarrow \infty}\e_n=0$, there is a subsequence
$\{X^{\e_{n_k}}\}_{k\ge1}$ (which we still
denote by
$\{X^{\e_n}\}_{n\ge1}$ below for the notional simplicity)
such that the distribution of $\{X^{\e_n}\}_{n\ge1}$ on
$\mathscr{D}([0,\infty);\R^d)$ equipped with the Skorohod topology converges weakly to a probability measure
$\bar \Pp$ on $\mathscr{D}([0,\infty);\R^d)$. Let
\begin{equation}\label{limit-operator}
\bar \sL f(x)=
\int_{\R^d}\big(f(x+z)-f(x)\big)\bar k_0(z)\,\Pi_0 (dz) ,
\end{equation} which  is
the infinitesimal generator of the L\'evy process as in the
statement. In particular, the associated martingale problem for $(\bar \sL, C^2_c (\R^d))$
is unique. Thus, it suffices to verify that for
any subsequence of $\{\e_n\}_{n\ge 1}$, the limit distribution $\bar
\Pp$ is the same as that of the solution to the martingale problem
for the operator $(\bar \sL, C^2_c(\R^d))$.

Due to the fact that the distribution of $\{X^{\e_n}\}_{n\ge1}$ converges weakly to $\bar \Pp$
in $\D([0, \infty); \R^d)$,
there exist a
probability space $(\wt  \Omega, \wt  \F, \wt  \Pp)$, and a
series of stochastic processes $\{\wt  X^n\}_{n\ge1}$
and $\wt  X$ defined on it, such that
the distribution of $\wt  X^n$ under $\wt
\Pp$ is the same as that of $X^{\e_n}$  under
$\Pp_0$ for any $n\ge1$, the distribution of $ \wt  X $ is the same
as $\bar \Pp$, and $\wt  X^n$
converges to $\wt
X$ almost surely
in $\D([0, \infty); \R^d)$.

Note again that $X^\e:=((X_{t}^\e)_{t\ge 0}; (\Pp^x)_{x\in \R^d})$
is a solution to the martingale problem for the operator
$(\sL^\e, C^2_c(\R^d))$. Then, for every $0<s_1<s_2,\cdots<s_k<s\le t$, $f\in C_c^2(\R^d)$
and $G\in C_b(\R^{dk})$,
$$
\wt \Ee\left[\left(f(\wt  X_t^n)-f(\wt  X_s^n)-\int_s^t
\sL^{\e_n}f(\wt  X_r^n)\,dr\right) G\big(\wt  X_{s_1}^n,\cdots,
\wt  X_{s_k}^n\big)\right]=0.
$$
According to \eqref{t3-1-2}, Lemma \ref{l2-2} and the dominated
convergence theorem,
\begin{equation}\label{t3-1-5}
\lim_{n \rightarrow \infty}\wt \Ee\left[\left(f(\wt
X_t^n)-f(\wt  X_s^n)- \int_s^t  \bar \sL^{\e_n} f(\wt
X_r^n)\,dr\right) G\big(\wt  X_{s_1}^n,\cdots, \wt
X_{s_k}^n\big)\right]=0,
\end{equation}
where $\bar \sL^\e f(x)$ is defined by \eqref{l2-2-1a}.
Set $F:\T^d\times \R^d \to \R$ and $\bar F:\R^d\to \R$ by
\begin{align*}
F(x,y):=\int_{\R^d}\big(f(y+z)-f(y)\big)\bar k(x,z)\,\Pi_0 (dz) , \quad \bar F(y):=\int_{\T^d} F(x,y)\,\mu(dx),
\end{align*}
where $\bar k(x,z)$ is given by \eqref{e:3.6}.
Then, $\bar \sL^{\e}f(x)=F(x/\e,x)$ and $\bar \sL f(x)=\bar F(x)$. Therefore,
\begin{align*}
&\wt  \Ee\left[\left|\int_s^t \bar \sL^{\e_n}f(\wt  X_r^n)\,dr-\int_s^t \bar \sL f(\wt  X_r)\,dr\right|\right]\\
&\preceq  \wt  \Ee\left[\left|\int_s^t \big(F(\wt X_r^n/\e_n, \wt X_r^n)-\bar F
(\wt X_r^n)\big)\,dr\right|\right]+
\wt  \Ee\left[\left|\int_s^t \big(\bar F(\wt X_r^n)-\bar F(\wt  X_r)\big)\,dr\right|\right]\\
&=:I_1^n+I_2^n.
\end{align*}
Following the proof of \eqref{t3-1-4}, we have that for any $\theta>0$ and $\delta(\e)>0$ with $\lim_{\e\to0}\delta(\e)=0$,
$$
\lim_{\e\to 0} \Pp_0\Big(\sup_{0\le s\le t\le s+ \delta(\e)}|X_{t}^\e-X_s^\e|>\theta\Big)=0.
$$
Clearly,
$$
\Ee_0  \left[\sup_{0\le s\le t\le s+ \delta(\e)}|X_t^\e-X_s^\e| \wedge1\right]
\le  \Pp_0\Big(\sup_{0\le s\le t\le s+ \delta(\e)}|X_t^\e-X_s^\e|>\theta\Big)+\theta.
$$
By letting $\e\to0$ first and then $\theta\to0$ in the inequality above, we get that \eqref{e:cond} holds
as $\lim_{\e\to 0} \rho(1/\e) = \lim_{\e\to 0}\varphi(1/\e)=\infty$. Thus, by
Lemma \ref{l2-4} and the fact that $F$ is uniformly continuous,
\begin{equation}\label{e:I1}
\lim_{n\to\infty}\wt  \Ee\left[\left|\int_s^t \big(F ( \wt X_r^n/\e_n, \wt  X_r^n)-\bar F
(\wt
X_r^n)\big)\,dr\right|^2\right]=0,
\end{equation}
and so
 $\lim_{n \rightarrow \infty}I_1^n=0$. On the other hand,  by
the facts that $\wt  X^n$
converges almost surely
in $\D([0, \infty); \R^d)$ to
 $\wt  X$
and $\bar F\in C_b(\R^d)$,
as well as the dominated
convergence theorem, it holds that $\lim_{n \rightarrow
\infty}I_2^n=0$. Thus, we
obtain
$$
\lim_{n \rightarrow \infty} \wt  \Ee\left[\left|\int_s^t \bar
\sL^{\e_n}f(\wt  X_r^n)\,dr-\int_s^t \bar \sL f(\wt
X_r)\,dr\right|\right]=0.
$$

Putting the estimate above into \eqref{t3-1-5} and letting $n
\rightarrow \infty$, we get
$$
\wt \Ee\left[\left(f(\wt  X_t)-f(\wt  X_s)-\int_s^t \bar \sL
f(\wt  X_r)\,dr\right) G\big(\wt  X_{s_1},\cdots, \wt
X_{s_k}\big)\right]=0.
$$
Thus,
$(\wt X, \wt \Pp)$ is a solution for the martingale problem $(\bar
\sL, C^2_c(\R^d))$.  This shows that $\wt X$ is a pure jump L\'evy
process with L\'evy measure $\bar k_0(z)\, \Pi_0 (dz) $.
\end{proof}

\subsection{Invariance principle
with recentering: $\alpha\in [1,2)$}\label{section3-2}

In this subsection, we are concerned with the case that $\alpha\in [1,2)$ and will present
scaling limit theorems that require recentering for
$X^\e$ of \eqref{e:scaled}.
Recall that $\Pi_0 (dz) $ is defined by
\eqref{e:3.2a} and $\bar k_0(z):=\int_{\T^d}\int_{\T^d}\bar k(x,z)\,\mu(dx)$, where
 $\mu$ is the unique invariant probability  measure of $X$ on $\T^d$, and  $\bar k(x,z)$ is
the function in \eqref{e:3.6}.
The following is the main result of this subsection.

\begin{theorem}\label{t3-2}
Suppose that  \eqref{e:3.3a} and \eqref{e:3.6} hold,
and that Assumption ${\bf (A3)}$ is satisfied.
\begin{itemize}
\item[(i)] Assume that \eqref{e:3.1} holds with $\alpha=1$.
Let $$Y_t^\e:=X_t^\e-\e\varphi(1/\e)(\bar b_{1/\e}+\bar b) t=
\e\big(X_{\varphi(1/\e)t}-\varphi(1/\e)(\bar b_{1/\e}+\bar b)
t\big),\quad t\ge 0,$$ where $\bar b_{1/\e}:=\int_{\T^d}
b_{1/\e}(x)\,\mu(dx)$ and  $\bar
b:=\int_{\T^d} b(x)\,\mu(dx)$. Then,
as $\e \to 0$, $(Y_t^\e)_{t\ge 0}$ converges weakly in $\D ([0, \infty); \R^d)$
to a Cauchy $($i.e.\ $1$-stable$)$
L\'evy process whose generator $\bar \sL$ is
\begin{equation}\label{t3-2-1}
\bar \sL f(x)=\int_{\R^d}\big(f(x+z)-f(x)-\langle \nabla f(x), z\I_{\{|z|\le 1\}}\rangle \big)\bar k_0(z) \,\Pi_0 (dz) .
\end{equation}

\item [(ii)] Assume that \eqref{e:3.1} holds with $\alpha \in (1, 2)$.
Let $$Y_t^\e:=X_t^\e-\e\varphi(1/\e)(\bar b_\infty+\bar b) t=
\e\big(X_{\varphi(1/\e)t}-\varphi(1/\e)(\bar b_\infty+\bar b)
t\big),\quad t\ge 0,$$ where $\bar b_\infty:=\int_{\T^d} b_\infty(x)\,\mu(dx)$.
Then,
as $\e \to 0$, $(Y_t^\e)_{t\ge 0}$ converges weakly in $\D ([0, \infty); \R^d)$
to  an $\alpha$-stable  L\'evy process
whose generator $\bar \sL$ is
$$
\bar \sL f(x)=\int_{\R^d}\big(f(x+z)-f(x)-\langle \nabla f(x), z\rangle \big)\bar k_0(z)\, \Pi_0 (dz) .
$$
\end{itemize}
\end{theorem}

Note that when \eqref{e:3.1} holds with $\alpha \in (1, 2)$ (resp.\ $\alpha=1$), $\lim_{\e \to 0} \e \varphi (1/\e) = \infty$ (resp.\ $\lim_{\e \to 0} \e \varphi (1/\e)>0$).
So in assumptions of Theorem \ref{t3-2}, one really needs to recenter $X^\e$ in order to have a limit.

To prove Theorem \ref{t3-2}, we need two lemmas. The first one is
analogous to Lemma \ref{l2-2}. Recall that the infinitesimal generator $\sL^\e$, given by \eqref{e:3.7}, of the process $X^\e$ can be written as
$$\sL^\e f(x)=\sL_0^\e f(x)+\e\varphi(1/\e)\left\langle \nabla f(x), b_{1/\e}(x/\e )+b(x/\e )\right\rangle$$ and
$$
\sL_0^\e f(x):=\varphi(1/\e)\int_{\R^d}\big(f(x+\e
z)-f(x)-\langle \nabla f(x), \e z\rangle\I_{\{|z|\le 1/
\e\}}\big)k(x/\e , z) \, \Pi (dz) .
$$
Note that, according to \eqref{e:3.8},
$$
\int_{\{|z|>1\}} | z| k(x,z)\,\Pi (dz) \preceq \int_{1}^{\infty}\frac{1}{\varphi(r)}\,dr,
$$
and so we can define
\begin{equation}\label{e3}
b_\infty(x)=\int_{\{|z|>1\}}zk(x,z)\,\Pi (dz)
\end{equation}
provided
\begin{equation}\label{e2-1}
\int_{1}^{\infty}\frac{1}{\varphi(r)}\,dr<\infty.
\end{equation}
In this case,
$$
\sL^\e f(x)=\sL_1^\e f(x)+\e\varphi(1/\e)
\left\langle \nabla f(x),
b_\infty(x/\e )+b(x/\e )\right\rangle,
$$
where
\begin{equation}\label{e:3.25}
  \sL_1^\e f(x)= \varphi(1/\e)\int_{\R^d}
\big(f(x+\e z)-f(x)-\langle \nabla f(x), \e z\rangle \big)k(x/\e , z)\,\Pi (dz) .
\end{equation}

\begin{lemma}\label{l2-2-111111}
\begin{itemize}
\item[(i)] For any $\e\in (0,1)$ and $x\in \R^d$, define
\begin{equation}\label{e:lllsssppp-}
\sL_{0,x}^\e f(y):=\varphi(1/\e)\int_{\R^d}\big(f(y+\e
z)-f(y)-\langle \nabla f(y), \e z\rangle\I_{\{|z|\le 1/
\e\}}\big)k(x/\e , z) \, \Pi (dz) .
\end{equation}
Suppose that \eqref{e:3.1} holds with $\alpha=1$. Then,
for every $f\in C_b^2(\R^d)$,
\begin{equation}\label{l2-2-1}
\lim_{\e \to 0}\sup_{x,y\in \R^d}|\sL_{0,x}^\e f(y)- \bar \sL^\e_{0,x}
f(y)|=0,
\end{equation}
where $$\bar \sL_{0,x}^\e f(y):=
\int_{\R^d}\big(f(y+z)-f(y)-\langle \nabla f(y), z\I_{\{|z|\le 1\}}\rangle \big)\bar k(x/\e ,z)\,\Pi_0 (dz) .$$

\item[(ii)] For any $\e\in (0,1)$ and $x\in \R^d$, define
$$  \sL_{1,x}^\e f(y)= \varphi(1/\e)\int_{\R^d}
\big(f(y+\e z)-f(y)-\langle \nabla f(y), \e z\rangle \big)k(x/\e , z)\,\Pi (dz) .$$
Suppose that \eqref{e:3.1} holds with $1<\alpha <2$. Then, for every $f\in C_b^2(\R^d)$,
\begin{equation}\label{l2-2-2}
\lim_{\e \to 0}\sup_{x,y\in \R^d}|  \sL_{1,x}^\e f(y)-\bar \sL_{1,x}^\e f(y)|=0,
\end{equation}where $$\bar \sL_{1,x}^\e f(y):=
\int_{\R^d}\big(f(y+z)-f(y)-\langle \nabla f(y), z\rangle \big)\bar k(x/\e ,z)\,\Pi_0 (dz) .$$ \end{itemize}
 \end{lemma}
\begin{proof} We only prove (ii), since the proof of (i) is similar and simpler.

Suppose that $1<\alpha<2$. Then, by  \eqref{e:3.1}, we have
\eqref{e2-1},
and so $  \sL_{1,x}^\e f$ is well defined for any $\e\in (0,1)$ and $x\in \R^d$.
Moreover, according to \eqref{e:3.1} and $1<\alpha<2$,
\begin{equation}\label{a2-1-3a}
\begin{split}
&\lim_{\e\to0} \e^2\varphi(1/\e)=0,\quad \lim_{\delta \to 0}\left(\int_0^\delta\frac{r}{\varphi(r)}\,dr+\int_{1/\delta}^\infty \frac{1}{\varphi(r)}\,dr\right)=0,\\
&\lim_{\delta\to0}\limsup_{\varepsilon\to0} \left(\e^2\varphi(1/\e)\int_1^{\delta/\varepsilon}\frac{r}{\varphi(r)}\,dr\right)=0,\quad \lim_{\delta\to0}\limsup_{\varepsilon\to0} \left(\e\varphi(1/\e)\int_{1/(\delta\varepsilon)}^\infty\frac{1}{\varphi(r)}\,dr\right)=0.\end{split}
\end{equation}
The proof for \eqref{a2-1-3a} is similar to that of \eqref{e:NO}, and we omit
the details.

For every $\delta\in (0,1)$ and $x\in\R^d$, we write
\begin{align*}
 \sL_{1,x}^\e f(y)=&\varphi(1/\e)
\left(\int_{\{|z|\le {\delta}/{\e}\}}+\int_{\{{\delta}/{\e}<|z|< {1}/({\delta\e})\}}+
\int_{\{|z|\ge {1}/({\delta\e})\}}\right)\\
&~~~~~\qquad\qquad\qquad\qquad\qquad\qquad~~~~~~~~~
\big(f(y+\e z)-f(y)-\langle \nabla f(y), \e z\rangle\big)k(x/\e , z)\,\Pi (dz)\\
=&:\sum_{i=1}^3 \sL_{1,x,i}^{\e,\delta}f(y)
\end{align*} and
\begin{align*}
\bar \sL^\e_{1,x} f(y)=&
\left(\int_{\{|z|\le \delta\}}+\int_{\{\delta<|z|< 1/ \delta \}}+\int_{\{|z|\ge  1/ \delta \}}\right)
\big(f(y+z)-f(y)-\langle \nabla f(y), z\rangle\big)\bar k(x/\e,z)\,\Pi_0 (dz)\\
=&:\sum_{i=1}^3 \bar \sL_{1,x,i}^{\e,\delta}f(y).
\end{align*}

By \eqref{e:3.2a} and \eqref{a2-1-3a},
it is obvious that
\begin{align*}
&\lim_{\delta \to 0}\sup_{\e\in(0,1)}\sup_{x,y\in \R^d}\Big(|\bar \sL_{1,x,1}^{\e,\delta}f(y)|
+|\bar \sL_{1,x,3}^{\e,\delta}f(y)|\Big)\\
&\preceq
\lim_{\delta \to 0}
\bigg(\|\nabla^2 f\|_\infty \int_{0}^\delta \frac{r}{\varphi(r)}\,dr+
\|f\|_\infty \int_{{1}/{\delta}}^\infty\frac{1}{r\varphi(r)}\,dr+
\|\nabla f\|_\infty \int_{{1}/{\delta}}^\infty\frac{1}{\varphi(r)}\,dr\bigg)=0.
\end{align*}
On the other hand,  according to \eqref{e:3.8},
we have
\begin{align*}
\sup_{x,y\in \R^d}|\sL_{1,x,1}^{\e,\delta}f(y)|
&\preceq \|\nabla^2 f\|_\infty \e^2\varphi(1/\e)
\int_{\{|z|\le {\delta}/{\e}\}}|z|^2\,\Pi (dz) \\
&\preceq  \|\nabla^2 f\|_\infty
\e^2\varphi(1/\e)\left(
\int_{\{|z|\le  1\}}|z|^2\,\Pi (dz)
+\int_1^{{\delta}/{\e}}\frac{r}{\varphi(r)}\,dr\right)
\end{align*} and
\begin{align*}
\sup_{x,y\in \R^d}
|\sL_{1,x,3}^{\e,\delta}f(y)|&
\preceq \varphi(1/\e)\int_{\{|z|\ge {1}/({\delta\e})\}}\big(\|f\|_\infty+\e \|\nabla f\|_\infty|z|\big)\,\Pi (dz) \\
&\preceq \big(\|f\|_\infty+\|\nabla f\|_\infty\big)
\left(\varphi(1/\e)\int_{{1}/({\delta\e})}^{\infty} \frac{1}{r\varphi(r)}\,dr
+\varphi(1/\e)\e\int_{{1}/({\delta\e})}^{\infty} \frac{1}{\varphi(r)}\,dr\right)\\
& \preceq \big(\|f\|_\infty+\|\nabla f\|_\infty\big) \,
\varphi(1/\e)\e\int_{{1}/({\delta\e})}^{\infty} \frac{1}{\varphi(r)}\,dr.
\end{align*}
These estimates along with \eqref{a2-1-3a} yields that
$$
\lim_{\delta \to 0}\limsup_{\e \to 0}\sup_{x,y\in \R^d}\big(|\sL_{1,x,1}^{\e,\delta}f(y)|+|\sL_{1,x,3}^{\e,\delta}f(y)|\big)=0.
$$

Following the argument for \eqref{referen-1},
we can also obtain  that for every fixed $\delta\in (0,1)$,
 $$
\lim_{\e \to 0}\sup_{x,y\in \R^d}
|\sL_{1,x,2}^{\e,\delta}f(y)-\bar \sL_{1,x,2}^{\e,\delta}f(y)|=0.
$$

Combining all the estimates above, by first letting $\e \to 0$ and then $\delta \to 0$,
we get the assertion \eqref{l2-2-2}.
\end{proof}

In the next lemma, we use the convention $1/0=\infty$.

\begin{lemma}\label{t3} Suppose that Assumption ${\bf (A3)}$ holds. For any $\e\in [0,1]$, let
$\psi^\e\in \mathscr{D}(\sL)$
be the solution to
\begin{equation}\label{t3-2-2}
\sL \psi^\e(x)=-b_{1/\e}(x)-b(x)+\bar b_{1/\e}+\bar b,\quad x\in \T^d
\end{equation}
with
 $\mu(\psi^\e)=0$.
Then,
\begin{equation}\label{t3-2-3}
\|\psi^\e\|_\infty+\|\nabla \psi^\e\|_\infty\preceq 1+
\int_1^{{1}/{\e}}\frac{1}{\varphi(r)}\,dr.
\end{equation}
 \end{lemma}

 \begin{proof} According to \eqref{e:3.8},
$$
  \sup_{x\in \R^d}|b_{1/\e}(x)|
\preceq \int_{\{1<|z|\le 1/{\e}\}}|z|\,\Pi (dz) \preceq
\int_1^{{1}/{\e}}\frac{1}{\varphi(r)}\,dr.
$$
This along with the fact that $b(x)\in C_b(\R^d)$ and {\bf(A3)} yields the desired assertion.
\end{proof}

Now, we are in a position  to present the

\begin{proof}[Proof of Theorem $\ref{t3-2}$]
(1) Suppose that Assumption {\bf (A3)} holds. We first assume that the solution $\psi^\e$ of
\eqref{t3-2-2} satisfies that $\mu(\psi^\e)=0$ and also $\psi^\e \in C^2(\T^d)$.
Set $\Phi_\e(x):=x+\e\psi^\e(x/\e )$.
Define
$$
Z_t^\e:=Y_t^\e+\e\psi^\e (X_t^\e/\e  )
=\Phi_\e(X_t^\e)
-\e\varphi(1/\e)(\bar b_{1/\e}+\bar b) t,\quad t\ge0.
$$
For $f\in C_b^2(\R^d)$, define
$$
f_{\e,s}(x):=f\big(x-\e\varphi(1/\e)(\bar b_{1/\e}+\bar b) s\big)
\quad \hbox{and} \quad
F_\e(s,x)=f_{\e,s}(\Phi_\e(x)).
$$
Clearly
$f(Z_t^\e)=F_{\e}(t,X_t^\e)$.
Since $X^\e:=((X_t^\e)_{t\ge 0};
(\Pp_x)_{x\in \R^d})$ is a solution to the martingale problem for
the operator $\sL^\e$, it holds that
for any $x\in \R^d$, $t>0$, $f\in
C_b^3(\R^d)$ and
any stopping time $\tau$,
$$
\Ee_x\big[f(Z_{t\wedge \tau}^\e)\big]=f\big(x+\e\psi^\e(\e^{-1}
x)\big)+\Ee_x\left[\int_0^{t\wedge \tau} \left(\frac{\partial
F_\e}{\partial s}(s,X_s^\e)+\sL^\e
F_\e(s,\cdot)(X_s^\e)\right)\,ds\right].
$$

Note that
$$
\frac{\partial F_\e}{\partial
s}(s,x)=-\e\varphi(1/\e)\left\langle \nabla
f_{\e,s}(\Phi_\e(x)), \bar b_{1/\e}+\bar b\right\rangle.
$$
Applying Lemma \ref{l2-1} with $R=1$ and $M={1}/{\e}$,
we
find that
\begin{align*}
&\sL^\e F_\e (s,\cdot)(x)\\
&= \varphi(1/\e)\,\,\int_{\R^d}\big(f_{\e,s}(\Phi_\e(x+\e
z))-
f_{\e,s}(\Phi_\e(x))-\langle \nabla
(f_{\e,s}(\Phi_\e(\cdot)))(x), \e z\rangle\I_{\{|z|\le 1\}}\big)k(x/\e , z) \,\Pi (dz) \\
&\quad+
\e\varphi(1/\e)\langle \nabla (f_{\e,s}(\Phi_\e(\cdot)))(x),b(x/\e )\rangle \\
&=\varphi(1/\e)\int_{\R^d}\big(f_{\e,s}(\Phi_\e(x+\e
z))- f_{\e,s}(\Phi_\e(x))-\langle\nabla
(f_{\e,s}(\Phi_\e(\cdot)))(x)
,\e z\rangle\I_{\{|z|\le {1}/{\e}\}})k(x/\e , z)\,\Pi (dz) \\
&\quad +\e\varphi(1/\e)\langle \nabla
(f_{\e,s}(\Phi_\e(\cdot)))(x),
b_{1/\e}(x/\e )+
b(x/\e )\big\rangle\\
&=\varphi(1/\e)\int_{\R^d}\big(f_{\e,s}(\Phi_\e(x)+\e
z)-
f_{\e,s}(\Phi_\e(x))-\langle
\nabla f_{\e,s}(\Phi_\e(x)),\e z \rangle\I_{\{|z|\le {1}/{\e}\}} )k(x/\e , z)\,\Pi (dz) \\
&\quad+\e\varphi(1/\e)\bigg\langle \nabla
f_{\e,s}(\Phi_\e(x)),
\int_{\R^d}\big(\psi^\e( x/\e +z )-
\psi^\e(x/\e )-\langle
\nabla \psi^\e(x/\e ), z\rangle\I_{\{|z|\le {1}/{\e}\}} \big)k ( x/\e ,z )\,\Pi (dz) \bigg\rangle\\
&\quad+\e\varphi(1/\e)\big\langle \nabla
f_{\e,s}(\Phi_\e(x\big)),
b_{1/\e}(x/\e )
+b(x/\e )\big\rangle\\
&\quad +\e\varphi(1/\e)\big\langle \nabla
f_{\e,s}(\Phi_\e(x\big)), \nabla \psi^\e(x/\e )\cdot
(
 b_{1/\e}(x/\e )+b(x/\e ))\big\rangle\\
&\quad
+\varphi(1/\e)\int_{\R^d}K_{\e}(x,z)k ( x/\e ,z )\,\Pi (dz) \\
&=\sL^{\e}_{0,x} f_{\e,s} (\Phi_\e(x))+
\e\varphi(1/\e) \langle \nabla
f_{\e,s}(\Phi_\e(x)),(\sL \psi^\e+b_{1/\e}+b)(x/\e )\rangle  +\varphi(1/\e)\int_{\R^d}K_{\e}(x,z)k ( x/\e ,z )\,\Pi (dz) ,
 \end{align*}
where $\sL^\e$
and $\sL_{0,x}^\e$
are defined by \eqref{e:3.7} and \eqref{e:lllsssppp-} respectively,
and
$K_{\e}(x,z)$
satisfies  that
\begin{align*}
&|K_{\e}(x,z)|\\
&\preceq (\e^3\|\nabla^3
f\|_\infty+\e^2\|\nabla^2 f\|_\infty)\big(1+\|\psi^\e\|_\infty+
\|\nabla\psi^\e\|_\infty\big)^3\big(|z|^2\I_{\{|z|\le 1\}}+|z|\I_{\{1<|z|\le
{1}/{\e}\}}+\I_{\{|z|>{1}/{\e}\}}\big) \\
&\quad+\|f\|_\infty \I_{\{|z|>{1}/{\e}\}}.
\end{align*}

According to
all estimates above,  \eqref{t3-2-2}, \eqref{e:3.8} and \eqref{t3-2-3}, we have that
\begin{equation}\label{t3-2-6}
\begin{split}
\frac{\partial F_\e}{\partial s}(s,x)+\sL^\e F_\e(s,\cdot)(x)
&=\sL^{\e}_{0,x}f_{\e,s}(\Phi_\e(x))\\
&\quad +\e\varphi(1/\e) \big\langle \nabla
f_{\e,s}(\Phi_\e(x)), \big(\sL \psi^\e+b_\e+b-\bar
b_\e-\bar b\big)(x/\e )
\big\rangle\\
&\quad +\varphi(1/\e)\int_{\R^d}K_{\e}(x,z)k ( x/\e ,z )\,\Pi (dz) \\
&=\sL^{\e}_{0,x}f_{\e,s}(\Phi_\e(x))
+H_\e(x),
\end{split}
\end{equation}
where
\begin{equation}\begin{split}\label{t3-2-6-1}
\sup_{x\in \R^d}|H_{\e} (x)| &\preceq \e^2\varphi(1/\e) \Big(\sum_{i=2}^3\|\nabla^i
f\|_\infty\Big)\!\!
\left(1\!+\int_1^{{1}/{\e}}\frac{1}{\varphi(r)}\,dr\right)^3\!\!\left(1\!+\int_{1}^{1/ \e}\frac{1}{\varphi(r)}\,dr
\!+\int_{{1}/{\e}}^\infty
\frac{1}{r\varphi(r)}\,dr\right)\\
&\quad
+\varphi(1/\e) \|f\|_\infty \int_{{1}/{\e}}^\infty\frac{1}{r\varphi(r)}\,dr .
\end{split}
\end{equation}

\medskip

(2) For any $l\ge1$, let $f_l$ be the function defined by \eqref{e:funcOO}. Then, for any $x,y\in \R^d$,
\begin{equation}\label{t3-2-7a}
\begin{split}
 |\sL_{0,x}^{\e}f_l(y)|
&\le \varphi(1/\e) \left|\int_{\{|z|\le
{1}/{\e}\}}\big(f_l(y+\e z)-f_l(y)-\langle \nabla f_l(y), \e z
\rangle \big)
k( x/\e ,z )\,\Pi (dz)
\right|\\
&\quad +\varphi(1/\e)\left|\int_{\{{1}/{\e}<|z|\le {l}/{\e}\}}\big(f_l(y+\e z)-f_l(y)\big)k( x/\e ,z )\,\Pi (dz)\right|\\
&\quad  +\varphi(1/\e) \left|\int_{\{|z|>
{l}/{\e}\}}\big(f_l(x+\e z)-f_l(x)\big) k( x/\e ,z )\,\Pi (dz)
\right| \\
 &\preceq
\|\nabla^2 f_l\|_\infty\varphi(1/\e)\e^2 \int_{\{|z|\le
{1}/{\e}\}}|z|^2 \,\Pi (dz)
 +\|\nabla f_l\|_\infty\e\varphi(1/\e)\int_{\{{1}/{\e}<|z|\le {l}/{\e}\}}|z|\,\Pi (dz)\\
&\quad+\| f_l\|_\infty\varphi(1/\e)
\int_{\{|z|> {l}/{\e}\}} \, \Pi (dz) \\
&\preceq l^{-2}\e^2\varphi(1/\e)\left(1+\int_1^{{1}/{\e}}
\frac{r}{\varphi(r)}\,dr\right)+l^{-1}\e\varphi(1/\e)\int_{{1}/{\e}}^{{l}/{\e}}\frac{1}{\varphi(r)}\,dr\\
&\quad
+\varphi(1/\e)\int_{{l}/{\e}}^\infty\frac{1}{r\varphi(r)}\,dr.
\end{split}
\end{equation}

Let $F_{\e,l}(t,x):=f_l\big(\Phi_\e(x)- \e\varphi(1/\e)(\bar
b_{1/\e}+\bar b) t \big)$. Then, combining \eqref{t3-2-6}, \eqref{t3-2-6-1} with
\eqref{t3-2-7a}, we find that
\begin{equation}\label{t3-2-7}\begin{split}
&\sup_{x\in \R^d}\sup_{s>0}\Big|\frac{\partial F_{\e,l}}{\partial s}(s,x)+\sL^\e F_{\e,l}(s,\cdot)(x)\Big|\\
&\preceq
l^{-2}\e^2\varphi(1/\e)\left(1+\int_1^{{1}/{\e}}\frac{r}{\varphi(r)}\,dr+
\left(\int_1^{{1}/{\e}} \frac{1}{\varphi(r)}\,dr\right)^4+
\left(\int_{1/\e}^\infty\frac{1}{r\varphi(r)}dr\right)^4\right)\\
&\quad +
l^{-1}\e\varphi(1/\e)\left(\e+ \int_{{1}/{\e}}^{{l}/{\e}}\frac{1}{\varphi(r)}\,dr +\e\left(\int_1^{{1}/{\e}}\frac{1}{\varphi(r)}\,dr\right)^4+ \e\left(\int_{1/\e}^\infty\frac{1}{r\varphi(r)}dr\right)^4\right)\\
&\quad+\varphi(1/\e)\int_{1/\e}^\infty\frac{1}{r\varphi(r)}\,dr.
\end{split}\end{equation}
Note that in the arguments above for \eqref{t3-2-7} we need that $\psi^\e \in C^2(\T^d)$. For general
$\psi^\e\in \mathscr{D}(\sL)$ satisfying \eqref{t3-2-2}, there exists a sequence of function
$\{\psi^\e_k\}_{k\ge1}\subset C^2(\T^d)$ such that
$\mu(\psi^\e_k)=0$ for all $k\ge1$ and
$$
\lim_{k \to \infty}\sup_{x\in \T^d} (\big|\psi^\e_k(x)-\psi^\e(x)\big|+\big|\sL \psi^\e_k(x)-\sL \psi^\e(x)\big|)=0.
$$
This along with \eqref{a3-1-2} yields that $$
\sup_{k\ge 1}\Big(\|\psi^\e_k\|_\infty+\|\nabla \psi^\e_k\|_\infty\Big)\preceq 1+\int_1^{{1}/{\e}}\frac{1}{\varphi(r)}\,dr.
$$
For the arguments above, we have used the facts that $\sL\psi_k^\e\in C(\T^d)$ with $\mu(\sL\psi^\e_k)=0$ for all $k\ge1$ and the solution to
\eqref{a3-1-1}
with $f=\sL \psi^\e_k$
is unique.
By a standard approximation procedure, it is not difficult to verify that \eqref{t3-2-7}
still holds true for every $\psi^\e\in \mathscr{D}(\sL)$ satisfying \eqref{t3-2-2}.

We now assume that  \eqref{e:3.1} holds with $\alpha=1$.
It follows from \eqref{e:3.1} that
$$\varphi(1/\e)\int_1^{{1}/{\e}}\frac{r}{\varphi(r)}\,dr\preceq \int_{1}^{{1}/{\e}} r(\e r )^{-1-\eta_0}\,dr\preceq \e^{-2}.$$
Estimating other terms in the right hand side of \eqref{t3-2-7}
by  the same way as above, we   get
$$
\lim_{R \to \infty}\sup_{\e\in (0,1)}\sup_{x\in \R^d} \sup_{s>0}\Big|\frac{\partial F_{\e,R}}{\partial s}(s,x)+\sL^\e
F_{\e,R}(s,\cdot)(x)\Big| =0.
$$
Since for every $R>1$ and $T>0$
\begin{align*}
 \Pp_0\Big(\sup_{t\in [0,T]}|Z_t^\e|>R\Big)
&\preceq \Ee f_R(Z_{T\wedge \tau_R^\e}^\e) =\Ee \left[\int_0^{T\wedge \tau_R^\e}\Big(\frac{\partial
F_{\e,R}}{\partial s}
(s,X_s^\e)+\sL^\e F_{\e,R}(s,\cdot)(X_s^\e)\Big)\,ds\right]\\
&\preceq T\sup_{x\in \R^d}\sup_{s>0}\left|\frac{\partial
F_{\e,R}}{\partial s}(s,x)+\sL^\e F_{\e,R}(s,\cdot)(x)\right|,
\end{align*}
we  conclude that
\begin{equation}\label{t3-2-9}
\lim_{R \to \infty}\sup_{\e\in (0,1)}\Pp_0
\left(\sup_{t\in [0,T]}|Z_t^\e|>R \right) =0.
\end{equation}

According to the argument for \eqref{t3-2-7a},
we can also obtain that for
every $\theta\in (0,1)$,
\begin{align*}
\sup_{x\in \R^d}\sup_{s>0}\Big|\frac{\partial F_{\e,\theta}}{\partial s}(s,x)+\sL^\e F_{\e,\theta}(s,\cdot)(x)\Big|  \preceq & \theta^{-3}
\e^2\varphi(1/\e)\bigg(1+\int_1^{{1}/{\e}}
\frac{r}{\varphi(r)}\,dr+
\Big(\int_1^{{1}/{\e}}\frac{1}{\varphi(r)}\,dr\Big)^4\bigg)\\
&+ \varphi(1/\e)\int_{1/\e}^{\infty} \frac{1}{r\varphi(r)}\,dr.
\end{align*}
By this estimate and the fact that \eqref{e:3.1} holds with $\alpha=1$, we
have
$$
\sup_{\e\in (0,1)}\sup_{x\in \R^d}\sup_{s>0}\Big|\frac{\partial F_{\e,\theta}}{\partial s}(s,x)+\sL^\e F_{\e,\theta}(s,\cdot)(x)\Big|
\preceq C(\theta)
$$
for some constant $C(\theta)>0$.
This together with the proof of \eqref{t3-2-9} gives us that
for any increasing function $\delta (\e)$ with $\lim_{\e \to 0} \delta (\e)=0$ and stopping
time $\tau$ with $\tau\le T$
\begin{equation}\label{t3-2-10}
\lim_{\e \to 0}\Pp_0\big(|Z_{\tau+\delta(\e)}^\e-Z_{\tau}^\e|>\theta\big)=0.
\end{equation}
Therefore, it follows from  \eqref{t3-2-9} and \eqref{t3-2-10} as
well as \cite[Theorem 1]{A} that the distribution of
$\{Z^\e\}_{\e\in (0,1)}$ is tight
in $\D([0, \infty); \R^d)$.

\medskip

(3)   Let $\{Z^{\e_n}\}_{n\ge1}$ be a sequence of
processes with $\lim_{n \rightarrow \infty}\e_n=0$. There is a
subsequence $\{ Z^{\e_{n_k}}\}_{k\ge1}$ (which will be
still denoted by $\{ Z^{\e_n}\}_{n\ge1}$ below for
simplicity) such that the distribution of $ Z^{\e_n}$ on
$\mathscr{D}([0,\infty;\R^d)$ converges weakly
under the Skorohod topology to a probability measure
$\bar \Pp$ on $\mathscr{D}([0,\infty);\R^d)$. Note that
$$
Y^{\e}_t=Z_t^{\e }+\e\psi^\e (X_t^\e  /\e ),\quad t\ge0.
$$
By \eqref{t3-2-3},  $\lim_{\e \to 0}\e\|\psi^\e\|_\infty=0$. This
implies that the distribution of
$(Y^{\e_n}_t)_{t\ge 0}$ converges weakly
in $\D([0, \infty); \R^d)$
to $\bar \Pp$.
Similar to the part (3) of
the proof for Theorem \ref{t3-1}, it suffices to
verify that  for any subsequence $\{\e_n\}_{n\ge 1}$, the limit
distribution $\bar \Pp$ is the same as that of the solution to the
martingale problem for $\bar \sL$ defined by \eqref{t3-2-1}.

For every $0<s_1<s_2,\cdots<s_k<s\le t$, $f\in C_b^2(\R^d)$ and
$G\in C_b(\R^{dk})$, by
\eqref{t3-2-6},
\begin{align*}
&\Ee \Big[\Big(f(Z_t^{\e})-f(Z_s^{\e})-\int_s^t \Big(\frac{\partial
F_{\e}}{\partial
r}(r,X_r^\e)+\sL^{\e}F_{\e}(r,\cdot)(X_r^\e)\Big)\,dr\Big)
G(Z_{s_1}^\e,\cdots, Z_{s_k}^\e)\Big]\\
&=\Ee_0\Big[\Big(f(Z_t^\e)-f(Z_s^\e)-\int_s^t (\sL_{0,X_r^\e}^{\e}f_{\e,r}(
\Phi_\e(X_r^\e))+H_\e (X_r^\e))\,dr\Big) G(Z_{s_1}^\e,\cdots,
Z_{s_k}^\e)\Big]\\
&=0.
\end{align*}
According to condition \eqref{e:3.1} with $\alpha =1$
and
\eqref{t3-2-6-1},
$\lim_{\e \to
0}\sup_{x\in \R^d}|H_\e(x)|=0.$ Combining this
with \eqref{l2-2-1} further yields
\begin{align*}
\lim_{\e \to 0}\Ee_0\Big[\Big(f(Z_t^\e)-f(Z_s^\e)-\int_s^t
\bar \sL_{0,X_r^\e}^\e f(Z_r^\e)
\,dr\Big)
G\big(Z_{s_1}^\e,\cdots, Z_{s_k}^\e\big)\Big]=0.
\end{align*}
Note that, by tracking the proof of Lemma \ref{l2-4}, we can verify that if \eqref{e:cond} holds with
$(X_t^\e)_{t\ge 0}$ replaced by $(Z_t^\e)_{t\ge 0}$, then for every bounded continuous function
$F:\T^d\times \R^d \to \R$ and $0<s<t$,
\begin{align*}
\lim_{\e \to 0}\Ee_x\left[\left|\int_s^t F\left(X_r^\e/\e,Z_r^\e\right)dr-\int_s^t \bar F\left(Z_r^\e\right)dr\right|\right]=0,
\end{align*}
where $\bar F$ is the same function as in Lemma \ref{l2-4}.
Using this property and \eqref{t3-2-10}, we can follow the argument for \eqref{e:I1} to
obtain
\begin{align*}
\lim_{\e \to 0} \Ee_0\left[\left|\int_s^t \bar \sL_{0,X_r^\e}^\e f(Z_r^\e)\,dr- \int_s^t \bar\sL
f(Z_r^\e)\,dr\right|^2\right]=0.
\end{align*}
Hence,
$$
\lim_{\e \to 0}\Ee_0\left[\left(f(Z_t^\e)-f(Z_s^\e)-\int_s^t \bar  \sL
f(Z_r^\e)\,dr\right) G(Z_{s_1}^\e,\cdots, Z_{s_k}^\e)\right]=0,
$$
where $\bar \sL$ is defined in \eqref{t3-2-1}.
Notice further that the distribution of $\{Z^{\e_n}\}_{n\ge1}$ converges
weakly to $\bar \Pp$.
According to the proof of Theorem \ref{t3-1} (in particular, by applying the Skorohod representation theorem),
letting $\e=\e_n$ and taking $\e_n \to 0$ in the equation above give
us
$$
\bar \Ee\Big[\Big(f(Z_t)-f(Z_s)-\int_s^t \bar  \sL  f(Z_r) \,dr\Big) G
(Z_{s_1},\cdots, Z_{s_k} )\Big]=0,
$$
where $(Z_t)_{t\ge 0}$ denotes the coordinate process on
$\mathscr{D}([0,\infty);\R^d)$, and $\bar \Ee$ denotes the expectation
with respect to $\bar \Pp$. This implies that the distribution of
$\bar\Pp$ is a solution to the martingale problem for the L\'evy
operator $\bar \sL$.
By now we have finished the proof for the
assertion (i) of Theorem \ref{t3-2}.

\medskip

(4)   Next,  we assume that condition \eqref{e:3.1} holds with $\alpha \in (1, 2)$.
In this case,
$b_\infty(x)$ is well defined.  According to Assumption
{\bf (A3)}, let $\psi\in \mathscr{D}(\sL)$
be the unique solution to the following
equation
$$
\sL \psi(x)=-b_\infty(x)-b(x)+\bar b_\infty+\bar b,\quad x\in \T^d
$$
with $\mu(\psi)=0$.
By the approximation argument as in Step (2),
without loss of generality we can suppose that $\psi\in C^2(\T^d)$. For every $f\in C_b^3(\R^d)$,  define
$$F_\e(s,x):=f (x+\e\psi(x/\e )-\e\varphi(1/\e)(\bar
b_\infty+\bar b) s ).$$ Applying Lemma \ref{l2-1}
with $R=1$ and $M=\infty$, and following the same arguments for
\eqref{t3-2-6} and \eqref{t3-2-7}, we obtain
$$
\frac{\partial F_\e}{\partial s}(s,x)+\sL^\e F_\e(s,\cdot)(x)= \sL_{1,x}^\e
f_{\e,s}(\Phi_\e(x))+H_\e(x),
$$
where $\Phi_\e(x)=x+\e\psi(x/\e )$,
$f_{\e,s}(x)=f\big(x-\e\varphi(1/\e)(\bar b_\infty+\bar b)s\big)$,
and $H_\e$ satisfies
that
\begin{align*}
\sup_{x\in \R^d}|H_\e(x)| &\preceq \Big(\sum_{i=2}^3\|\nabla^i
f\|_\infty\Big)\varphi(1/\e)\e^2
\left(\int_{\{|z|\le 1\}}|z|^2 \,\Pi (dz) +\int_{1}^{\infty}\frac{r}{\varphi(r)}\,dr\right)\\
&\preceq \left(\sum_{i=2}^3\|\nabla^i
f\|_\infty\right)\varphi(1/\e)\e^2.
\end{align*}
In particular, $$\lim_{\e \to 0}\sup_{x\in \R^d}|H_\e(x)|=0.$$ Using
these estimates and repeating the proof for the assertion (i) (in
particular, applying \eqref{l2-2-2} instead of
\eqref{l2-2-1}),
we obtain the assertion (ii) of Theorem \ref{t3-2}.
\end{proof}

\section{Homogenization: diffusive scaling}\label{section4}

In this section, we treat the case that the jumping measure for the non-local operator $\sL$
of \eqref{e4-1}  has
a finite second moment, i.e.,
\begin{equation}\label{a2-2-1}
\int_{\R^d}|z|^2\,\Pi (dz) <\infty.
\end{equation}
Under this condition, it is natural to conjecture that, after
appropriate scaling, $X$ would converge to Brownian motion. Thus
we will take the scaling function $\rho (r)=r^2$ in \eqref{Sca}, and
consider the limit of the scaled process $X^\e=(X_t^\e)_{t\ge
0}:=(\e X_{t/\e^{2}})_{t\ge 0}$. Here is the main result of this
section.

\begin{theorem}\label{t3-3}
Suppose that  Assumption ${\bf (A3)}$ and \eqref{a2-2-1} hold. Let
$$Y_t^\e:=X_t^\e-(\bar b_\infty +\bar b) t/\e= \e\big(X_{t/\e^2}-(\bar
b_\infty+\bar b) t/\e^2),\quad t\ge0.$$ Then $(Y_t^\e)_{t\ge 0}$ converges weakly
in $\D([0, \infty); \R^d)$,
as $\e \to 0$, to
Brownian motion
with the covariance matrix $A$ given by
$$
A:=\int_{\T^d}\int_{\R^d}\big(z+\psi(x+z)-\psi(x)\big)\otimes
\big(z+\psi(x+z)-\psi(x)\big)k(x,z)\,\Pi (dz)
\,\mu(dx),
$$
where $\psi\in \mathscr{D}(\sL)$
is
the unique
solution to the following
equation
\begin{equation}\label{function-1}
\sL \psi(x)=-b_\infty(x)-b(x)+\bar b_\infty+\bar b,\quad x\in \T^d
\end{equation}
such that $\mu(\psi)=0$.
\end{theorem}

\begin{remark}\label{r4-1}
Let
$\{e_i: 1\leq i\leq d \} $
be the standard orthonormal basis of $\R^d$. We claim that if the process $X$ is irreducible,
and for each $e_i$, $1\le i \le d$, there exists a sequence $\{z_k^i\}_{k\ge1} \subset
{\rm supp}[\Pi]$
such that $z_k^i\neq 0$ for all $k\ge1$ and
 \begin{equation}\label{r4-1-1}
\lim_{k\to \infty}z_k^i=0, \qquad \lim_{k \to \infty}z_k^i/|z_k^i|=e_i,\quad 1\le i\le d,
\end{equation}
then the covariance matrix $A$ in Theorem \ref{t3-3} above is non-degenerate.
Indeed, for any $\xi\in \R^d$,
$$\langle A\xi, \xi\rangle= \int_{\T^d}\int_{\R^d}\langle z+\psi(x+z)-\psi(x),\xi\rangle^2\,\Pi (dz) \,\mu(dx).$$ Since the process $X$ is irreducible, for any $t>0$, $x\in \R^d$ and open set $U\subset \R^d$, $\Pp_x(X_t\in U)>0$. Then, $\mu(U)=\int_{\R^d}\Pp_x(X_t\in U)\,\mu(dx)>0$; that is, ${\rm supp}[\mu]=\T^d$.

Now, assume  that for some $0\neq \xi\in\R^d$,
$\langle A\xi, \xi\rangle=0$. Note that, under {\bf (A3)}, by the dominated convergence theorem, $x\mapsto \int_{\R^d}\langle z+\psi(x+z)-\psi(x),\xi\rangle^2\,\Pi (dz) $ is a continuous function. This along with the fact ${\rm supp}[\mu]=\T^d$ yields that  for every $x\in \T^d$, $$\int_{\R^d}\langle z+\psi(x+z)-\psi(x),\xi\rangle^2\,\Pi (dz) =0.$$
Without loss of generality, we assume that $\xi=(\xi_{(1)},\cdots,\xi_{(d)})$ with $\xi_{(1)}>0$ (since $\xi \neq 0$).
Let $\{z_k^1\}_{k\ge1}$ be the sequence in the assumptions above. Then, we have
$$\langle\psi(x+z_k^1)-\psi(x),\xi\rangle=\langle -z_k^1,\xi\rangle,\quad  k\ge 1, x\in \T^d.$$ By the mean value theorem and the fact that $\|\nabla \psi\|_\infty<\infty$ (see Lemma \ref{t3} with $\e=0$), we
have
$$\langle  \nabla \langle \psi(x) , \xi\rangle,z_k^1/|z_k^1|\rangle + o(|z_k^1|)/|z_k^1|=\langle z_k^1/|z_k^1|,-\xi\rangle,\quad  k\ge 1, x\in \T^d.$$
Letting $k \to \infty$ and using \eqref{r4-1-1}, we
obtain
\begin{equation*}
\partial_{x_1}\langle \psi(x),\xi\rangle=-\xi_{(1)}<0,\quad  x\in \T^d,
\end{equation*}
which obviously contradicts with the fact that $x\mapsto \langle \psi(x),\xi\rangle$ is continuous and
multivariate periodic.
Therefore, we have $\langle A\xi,\xi\rangle>0$ for every $\xi\neq 0$, and so
$A$ is non-degenerate.

We further note that the assumptions above, which guarantee that the covariance matrix $A$ in Theorem \ref{t3-3} is non-degenerate, are weak in some sense. For example, let $d=2$, $\Pi (dz) =\delta_{e_1}(dz)$, and $\xi=e_2$. Then, for any $x\in \T^d$, since $\psi$ is  multivariate periodic,
$$\int_{\R^2}\langle z+\psi(x+z)-\psi(x),\xi\rangle^2\,\Pi (dz) =(\psi_{(2)}(x+e_1)-\psi_{(2)}(x))^2= 0$$ and so $\langle A\xi, \xi\rangle=0$, where we write $\psi(x)=(\psi_{(1)}(x),\psi_{(2)}(x))$. Hence, the associated covariance matrix $A$ in Theorem \ref{t3-3} is degenerate.
\end{remark}

 Note that the scaled process $X^\e$ is  a strong Markov process, whose generator is given by
\begin{align*}
\sL^\e f(x)&= \e^{-2}
\int_{\R^d}\big(f(x+\e z)-f(x)-\e\langle \nabla f(x),z\rangle\I_{\{|z|\le 1\}}\big)k( x/\e ,z )\,\Pi (dz) + \e^{-1}
\langle \nabla f(x), b(x/\e )\rangle\\
&= :\sL_1^\e f(x)+ \e^{-1} \langle \nabla f(x), b_\infty(x/\e)+b(x/\e )\rangle.
\end{align*}
Here, $b_\infty(x)$ is defined by \eqref{e3}, and
\begin{equation}\label{e6-}
 \sL_1^\e f(x)= \e^{-2}\int_{\R^d}\big(f(x+\e z)-f(x)-\langle \nabla f(x), \e z\rangle \big)k( x/\e ,z )\,
\Pi (dz) .
\end{equation}

\begin{lemma}\label{l2-3}
For every $x\in \R^d$ and $f\in C_b^3(\R^d)$, define
\begin{equation}\label{e6}
 \sL_{1,x}^\e f(y)= \e^{-2}\int_{\R^d}\big(f(y+\e z)-f(y)-\langle \nabla f(y), \e z\rangle \big)k( x/\e ,z )\,
\Pi (dz) .
\end{equation} Then,
\begin{equation}\label{9no1}
 \sL_{1,x}^\e f(y)=\Big\langle \nabla^2 f(y), \frac{1}{2}\int_{\R^d}(z\otimes z) k( x/\e ,z )\,\Pi (dz)
 \Big\rangle
+G_{1,\e}(x,y),
\end{equation}
where $G_{1,\e}(x,y)$ satisfies
\begin{equation}\label{e:l2-3-1}\lim_{\e \to 0}\sup_{x,y\in \R^d}|G_{1,\e}(x,y)|=0.\end{equation}
\end{lemma}

\begin{proof}
According to the Taylor expansion, for any $f\in C_b^3(\R^d)$ and $R>1$,
\begin{align*}
&f(y+\e z)-f(y)-\langle \nabla f(y), \e z\rangle\\
&=\begin{cases}{\e^2}\langle \nabla^2 f(y), z\otimes z\rangle/{2}
+{\e^3} \langle \nabla^3 f(y+\theta_1\e z), z\otimes z\otimes z\rangle/{6},&\quad |z|\le R,\\
{\e^2}\langle \nabla^2 f(y+\theta_2 \e z), z\otimes z\rangle/{2}, &\quad |z|>R\end{cases}\\
&= {\e^2}\langle \nabla^2 f(y), z\otimes z\rangle/{2} +H_{\e}(y,z),
\end{align*}
where $\theta_1, \theta_2\in (0,1)$ and
\begin{align}\label{eq:enoo3nx}
|H_{\e}(y,z)|\preceq \big(\|\nabla^2 f\|_\infty+\|\nabla^3
f\|_\infty\big) \big(\e^3|z|^3\I_{\{|z|\le
R\}}+\e^2|z|^2\I_{\{|z|>R\}}\big).
\end{align}
In particular,  \eqref{9no1} holds with
$$
G_{1,\e}(x,y)= \varepsilon^{-2}  \int_{\R^d}H_{\e}(y,z)k( x/\e ,z )\,\Pi (dz) .
$$

By \eqref{eq:enoo3nx}, it holds that
$$|G_{1,\e}(x,y)|\le C_1\left(\sum_{i=2}^3\|\nabla^i f\|_\infty\right)\left(\e\int_{\{|z|\le R\}}|z|^3\,\Pi (dz) +
\int_{\{|z|>R\}}|z|^2\,\Pi (dz) \right)$$
with any $R>0$ and some $C_1>0$ $($which is independent of $f$, $\e$ and $R)$.
Since $\int_{\R^d}|z|^2\,\Pi (dz) <\infty$, first letting $\e \to 0$ and then
$R \to \infty$ in the estimate above,
we obtain \eqref{e:l2-3-1}. \end{proof}

\begin{proof}[Proof of Theorem $\ref{t3-3}$] We first assume that
$\psi\in C^2(\T^d)$.
Let
$$
Z_t^\e=Y_t^\e+\e\psi (  X_t^\e/\e )=X_t^\e+\e\psi ( X_t^\e/\e  )- \e^{-1} (\bar
b_\infty+\bar b) t,\quad t\ge0.
$$
Recall that the generator of the process $X^\e$ is
\begin{align*}\sL^\e
f(x)=
& \sL_1^\e f(x)+  \e^{-1}\langle \nabla f(x), b_\infty(
x/\e)+b(x/\e )\rangle,
\end{align*}
where $\sL_1^\e f$ is defined by \eqref{e6-}. Then, for every $f\in
C_b^3(\R^d)$, $x\in \R^d$ and $t>0$,
$$
\Ee_x[f(Z_t^\e)]=\Ee_x[F_\e(t,X_t^\e)]=f(\Phi_\e(x))+\Ee_x\left[\int_0^t \Big(\frac{\partial
F_\e}{\partial s}(s,X_s^\e)+\sL^\e
F_\e(s,\cdot)(X_s^\e)\Big)\,ds\right],$$
where $F_{\e}(s,x)=f(\Phi_\e(x)-\e^{-1}(\bar b_\infty +\bar
b)s)$ and  $\Phi_\e(x)=x+\e\psi(x/\e )$.

Let
$f_{\e,s}(x)=f(x-\e^{-1}(\bar b_\infty +\bar b) s)$. Then, $F_\e(s,x)= f_{\e,s}(\Phi_\e(x)).$ Set $\Theta_\e(x,z):=\psi( x/\e +z )-\psi(x/\e )$ and
\begin{align*}
A(x):=\frac{1}{2}\int_{\R^d}\big(z+\psi(x+z)-\psi(x)\big)\otimes\big(z+\psi(x+z)-\psi(x)\big)k(x,z)\,\Pi (dz) .
\end{align*}
Applying Lemma \ref{l2-1} with $M=\infty$,
$R\to 0$, and using
\eqref{function-1} and Lemma \ref{l2-3}, we can verify that
\begin{align*}
&  \frac{\partial F_\e}{\partial s}(s,x)+\sL^\e F_\e(s,\cdot)(x)\\
&= -\e^{-1}\langle \nabla f_{\e,s} (\Phi_\e(x)), \bar b_\infty+\bar b\rangle
 + \sL_{1,x}^\e f_{\e,s}(\Phi_\e(x))\\
 &\quad + \e^{-1} \Big\langle \nabla f_{\e,s} (\Phi_\e(x)),\!\int\big( \psi(x/\e  +z )\!-\! \psi(x/\e  )\!-\!\nabla \psi(x/\e ) \cdot z\big)k( x/\e ,z )\,\Pi (dz) \Big\rangle\\
&\quad + \frac{1}{2}\Big\langle
\nabla^2 f_{\e,s}(\Phi_\e(x)), \int
\big(2\Theta_\e(x,z)\otimes z+\Theta_\e(x,z)\otimes \Theta_\e(x,z)\big)k( x/\e ,z )\,\Pi (dz) \Big\rangle \\
&\quad +H_1^\e(x)+ \e^{-1} \langle \nabla f_{\e,s} (\Phi_\e(x)), b_\infty (x/\e )+b(x/\e )\rangle +\e^{-1} \langle \nabla f_{\e,s} (\Phi_\e(x)), \nabla\psi(x/\e  )\cdot(b_\infty(x/\e )+b(x/\e ))\rangle  \\
&= \sL_{1,x}^\e f_{\e,s}(\Phi_\e(x))+ \frac{1}{2}\Big\langle \nabla^2 f_{\e,s}(\Phi_\e(x)),
\int\!
\big(2\Theta_\e(x,z)\otimes z\!+\!\Theta_\e(x,z)\otimes \Theta_\e(x,z)\big)k( x/\e ,z )\,\Pi (dz) \Big\rangle + H_1^\e(x)\\
&=\frac{1}{2}\Big\langle \nabla^2 f_{\e,s}\big(\Phi_\e(x)\big), \int
\!\big(z+\Theta_\e(x,z)\big)\!\otimes\! \big(z+\Theta_\e(x,z)\big)
k( x/\e ,z )\,\Pi (dz) \Big\rangle\!
+\!H_2^\e(x)\\
&=\langle \nabla^2 f_{\e,s} (\Phi_\e(x) ), A(x/\e)\rangle+H_2^\e(x),
\end{align*} where in the first and the second equalities $\sL_{1,x}^\e$ is defined by \eqref{e6}.
Here $H_1^\e$ satisfies
$$
|H_1^\e(x)|\preceq  \e \left(\sum_{i=2}^3\|\nabla^i f\|_\infty \right)\bigg(\int_{\{|z|\le R\}}|z|^2\,\Pi (dz) + \int_{\{|z|>R\}}|z|\,\Pi (dz) \bigg),
$$ thanks to the fact that $\|\psi\|_\infty+\|\nabla \psi\|_\infty \preceq 1$ under Assumption {\bf (A3)} (see Lemma \ref{t3} with $\e=0$), and $H_2^\e(x)=H_1^\e(x)+G_{1,\e}(x)$ with $G_{1,\e}$ as in Lemma \ref{l2-3}.
As explained in the proof of Theorem \ref{t3-2}
the above estimate still holds true when $\psi\in \mathscr{D}(\sL)$.

Given these estimates, the rest of the proof is very similar to that of Theorem \ref{t3-2}, so we omit it
(see also the proof of Theorem \ref{t3-4} below).
\end{proof}

\section{Homogenization: critical cases}\label{section5}
Throughout this section, $\phi$ is a strictly positive and strictly decreasing
function on $\R_+$ so that $\lim_{r \to0}\phi(r)=\infty$,
$\lim_{r\to 0} r^{-2}\phi(r)^{-1}=\infty$,
\begin{equation}\label{a2-3-2} \lim_{\e \to
0}\left(\frac{\e\int_{\{|z|\le 1/\e\}}|z|^3 \,\Pi (dz) }
{\phi(\e)}+
\frac{\int_{\{|z|>1/\e\}}|z|\, \Pi (dz) }{\e\phi(\e)}\right)=0,
\end{equation} and
 \begin{equation}\label{a2-3-1-00}
\limsup_{\e \to 0}\frac{ \int_{\{|z|\le 1/\e\}}|z|^2 \,\Pi (dz) }{\phi(\e)}<\infty
\end{equation} with
 \begin{equation}\label{a2-3-1}
A:=\lim_{\e \to 0}\frac{\int_{\T^d}\int_{\{|z|\le 1/\e\}} (z\otimes z) k(x,z)\,\Pi (dz) \,\mu(dx)}{\phi(\e)}
\end{equation}
being a
non-zero $d\times d $-matrix.

We make four remarks on the assumptions above.
\begin{itemize}
\item[(i)] Since the matrix $A$ is
non-zero and $\lim_{\e \to0} \phi(\e)=\infty$, $\int_{\R^d}  |z|^2\,\Pi (dz) $ has to be infinite.
\item[(ii)] In \eqref{a2-3-2}, \eqref{a2-3-1-00} and \eqref{a2-3-1}, the domain $\{|z|\le 1/\e\}$ can be replaced by $\{r_0\le |z|\le 1/\e\}$ for any fixed $r_0\ge 1$.
\item[(iii)] For $\I_{\{|z|\ge1\}} \Pi (dz) =\I_{\{|z|\ge1\}}|z|^{-(d+\alpha)}\,dz$ with $\alpha\in (0,2)$, condition
\eqref{a2-3-1-00} holds with $\phi(\e)=\e^{\alpha-2}$ but condition
\eqref{a2-3-2} fails.
\item[(iv)] For $\I_{\{|z|\ge1\}}\Pi (dz) =\I_{\{|z|\ge1\}}|z|^{-d-2}\,dz$, conditions
\eqref{a2-3-2} and \eqref{a2-3-1-00} are satisfied with
$\phi(\e)=\log (1+1/\e).$\end{itemize}

Under \eqref{a2-3-2} and \eqref{a2-3-1-00}, we will take
$\rho(\e)=\e^2/\phi(1/\e)$
in \eqref{Sca}, which corresponds
to the scaling function for critical cases in the setting of
infinite second moments. The purpose of this section is to study the
limit behavior of the scaled process $X^\e:=(X^\e_t)_{t\ge 0}$
 defined by $X^\e_t=\e X_{\e^{-2}\phi(\e)^{-1}t}$ for any $t>0$.

\begin{theorem}\label{t3-4}
Suppose that  Assumption ${\bf (A3)}$, \eqref{a2-3-2},
\eqref{a2-3-1-00} and \eqref{a2-3-1} hold. Let
$$Y_t^\e:=X_t^\e-\e^{-1}\phi(\e)^{-1}(\bar b_\infty+\bar b) t=
\e\big(X_{\e^{-2}\phi(\e)^{-1}t}-\e^{-2}\phi(\e)^{-1}(\bar b_\infty+\bar
b) t\big),\quad t\ge0.$$   Then,  $(Y_t^\e)_{t\ge 0}$  converges weakly
in $\D([0, \infty); \R^d)$,
 as $\e \to 0$, to
 Brownian motion with the non-zero covariance matrix $A$ defined by \eqref{a2-3-1}.
\end{theorem}

For the scaled process
$X^\e=(X_t^\e)_{t\ge 0}:=(\e X_{\e^{-2}\phi(\e)^{-1}t})_{t\ge 0}$ as above,
its infinitesimal generator is given by
\begin{align*}
\sL^\e f(x)&=\frac{1}{\e^2\phi(\e)}
\int_{\R^d}\big(f(x+\e
z)-f(x)-\langle \nabla f(x),\e z\rangle\I_{\{|z|\le 1\}}\big)k( x/\e ,z )\,\Pi (dz) +\frac{1}{\e\phi(\e)}
\left\langle  b(x/\e ), \nabla f(x)\right\rangle\\
&=: \sL_1^\e f(x)+\frac{1}{\e\phi(\e)}
\left\langle  b_\infty(x/\e )+b(x/\e ),  \nabla f(x) \right\rangle,
\end{align*}
where
$$
 \sL_1^\e f(x)=\frac{1}{\e^2\phi(\e)}\int_{\R^d}\big(f(x+\e z)-f(x)-\langle \nabla f(x), \e z\rangle \big)
k( x/\e ,z ) \,\Pi (dz) .
$$

Similar to Lemma \ref{l2-3}, we have the following statement.
\begin{lemma}\label{l2-5} For any $x,y\in \R^d$ and $f\in C_b^3(\R^d)$, let
$$
 \sL_{1,x}^\e f(y)=\e^{-2}\phi(\e)^{-1}\int_{\R^d}\big(f(y+\e z)-f(y)-\langle \nabla f(y), \e z\rangle \big)k( x/\e ,z )\,\Pi (dz) .
$$ Suppose that  \eqref{a2-3-2}  holds. Then, for every $x,y\in \R^d$ and $f \in C_b^3(\R^d)$,
$$
 \sL_{1,x}^\e f(y)=\Big\langle \nabla^2 f(y), \frac{1}{2\phi(\e)}\int_{\{|z|\le 1/\e \}}(z\otimes z) k( x/\e ,z )\,\Pi (dz) \Big\rangle
+G_{2,\e}(x,y),
$$
where $G_{2,\e}(x,y)$ satisfies that $$\lim_{\e \to 0}
\sup_{x,y\in \R^d}|G_{2,\e}(x,y)|=0.$$ \end{lemma}
\begin{proof}
According to the Taylor expansion, for any $f\in C_b^3(\R^d)$ and
$\e\in (0,1)$,
\begin{align*}
&f(x+\e z)-f(x)-\langle \nabla f(x), \e z\rangle\\
&=\begin{cases}\frac{\e^2}{2}\langle \nabla^2 f(x), z\otimes z\rangle
+\frac{\e^3}{6} \langle \nabla^3 f(x+\theta_1\e z), z\otimes z\otimes z\rangle,&\quad |z|\le 1/\e ,\\
\langle \nabla f(x+\theta_2 \e z), \e z\rangle-\langle \nabla f(x), \e z\rangle, &\quad |z|>1/\e \end{cases}\\
&= \frac{\e^2}{2}\langle \nabla^2 f(x), z\otimes z\rangle \I_{\{|z|\le 1/\e \}}+H_{\e}(x,z),
\end{align*} where $\theta_1,\theta_2\in (0,1)$ and
$$
|H_{\e}(x,z)| \preceq \big(\|\nabla f\|_\infty+\|\nabla^3
f\|_\infty\big)\big( \e^3|z|^3\I_{\{|z|\le
1/\e \}}+\e|z|\I_{\{|z|>1/\e \}}\big).
$$
In particular, we have
$$
 \sL_{1,x}^\e f(y)=\Big\langle \nabla^2 f(y), \frac{1}{2\phi(\e)}\int_{\{|z|\le 1/\e\}} (z\otimes z ) k( x/\e ,z )\,\Pi (dz) \Big\rangle
+G_{2,\e}(x,y),
$$
where
$$ G_{2,\e}(x,y)=\e^{-2}\phi(\e)^{-1}\int H_\e(y,z)k( x/\e ,z )\,\Pi (dz) .$$ Furthermore,
\begin{equation}\label{l2-5-1}
 |G_{2,\e}(x,y)|\le C_1\left(\sum_{i=1}^3\|\nabla^i f\|_\infty\right)\left(\frac{\e\int_{\{|z|\le 1/ {\e}\}}|z|^3\,\Pi (dz) }{\phi(\e)}
+\frac{\int_{\{|z|>1/{\e}\}}|z|\,\Pi (dz) }{\e\phi(\e)}\right)
\end{equation}
with some $C_1>0$ independent of $f$ and $\e$.
By \eqref{a2-3-2},
we can further obtain that
$$\lim_{\e \to 0}\sup_{x\in \R^d}|G_{2,\e}(x)|\preceq \lim_{\e \to 0}\left(
\frac{\e \int_{\{|z|\le 1/ {\e} \}}|z|^3\, \Pi (dz) }{\phi(\e)}
+\frac{\int_{\{|z|>1/\e \}}|z|\,\Pi (dz) }{\e\phi(\e)}\right)=0.
$$  The proof is finished.
\end{proof}

\begin{proof}[Proof of Theorem $\ref{t3-4}$] (1) Let $\psi\in \mathscr{D}(\sL)$
be the unique solution to the following equation
$$
\sL \psi(x)=-b_\infty(x)-b(x)+\bar b_\infty+\bar b,\quad  x\in \T^d
$$
with $\mu(\psi)=0$.
Let $$Z_t^\e =Y_t^\e+\e\psi (  X_t^\e/\e )=X_t^\e+\e\psi (  X_t^\e/\e   )-\e^{-1}\phi(\e)^{-1}
(\bar b_\infty+\bar b) t,\quad t\ge0.$$
As explained in the proof of Theorem \ref{t3-2}, without of loss generality we can assume that $\psi\in C^2(\T^d)$.
Noticing that $((X_t^\e)_{t\ge 0};(\Pp_x)_{x\in \R^d})$ is a solution to the martingale problem for the operator $\sL^\e$, we obtain that for every $f\in C_b^3(\R^d)$, $x\in \R^d$ and $t>0$,
$$
\Ee_x[f(Z_t^\e)]=f(\Phi_\e(x))+\Ee_x\left[\int_0^t
\Big(\frac{\partial F_\e}{\partial s}(s,X_s^\e)+\sL^\e
F_\e(s,\cdot)(X_s^\e)\Big)\,ds\right],
$$
where $F_{\e}(s,x)=f\big(\Phi_\e(x)-\e^{-1}\phi(\e)^{-1}(\bar b_\infty+\bar b)s\big)$ and $\Phi_\e(x)=x+\e\psi(x/\e )$.

Let $f_{\e,s}(x)=f\big(x-\e^{-1}\phi(\e)^{-1}(\bar b_\infty+\bar b) s\big)$, $\Theta_\e(x,z)=\psi( x/\e +z )-\psi(x/\e )$, and
$$
A_\e(x) =\frac{1}{2}\int_{\{|z|\le 1/{\e} \}}(z\otimes z)k(x,z)\,\Pi (dz) .
$$Applying Lemma \ref{l2-1} with $R=1/{\e}$ and $M=\infty$, and using Lemma \ref{l2-5}
and \eqref{l2-5-1},
we can verify that
\begin{align*}
&\frac{\partial F_\e}{\partial s}(s,x)+\sL^\e F_\e(s,\cdot)(x)\\
&= \sL_{1,x}^\e f_{\e,s} (\Phi_\e(x) )\\
&\quad +\frac{1}{2\phi(\e)}\Big\langle \nabla^2 f_{\e,s} (\Phi_\e(x)
), \int\!
 (2\Theta_\e(x,z)\otimes  z \!+\Theta_\e(x,z)\otimes \Theta_\e(x,z) )k( x/\e ,z )\,\Pi (dz) \Big\rangle
+H_1^\e(x)\\
&=\frac{1}{2}\Big\langle
\nabla^2 f_{\e,s} (\Phi_\e(x) ), \frac{1}{\phi(\e)}\int_{\{|z|\le 1/{\e} \}}(z\otimes z)
k( x/\e ,z )\,\Pi (dz) \Big\rangle
+H_2^\e(x)\\
&= \Big\langle
\nabla^2 f_{\e,s} (\Phi_\e(x) ), \frac{1}{\phi(\e)}A_\e (x/\e )\Big\rangle+H_2^\e(x),
\end{align*}
where $H_1^\e$
and $H_2^\e$ satisfy
\begin{align*}
  |H_i^\e(x)|
 &\preceq \Big(\sum_{i=1}^3\|\nabla^i f\|_\infty \Big)  \Bigg[\frac{\e(\int_{\{0<|z|\le 1\}}|z|^2 \,\Pi (dz) +\int_{\{1<|z|\le  1/{\e}\}}|z|^3 \,\Pi (dz) )}{\phi(\e)}\\
&\qquad\qquad\qquad\qquad\,\,\,\, +\frac{\int_{\{|z|> 1/{\e} \}}|z| \,\Pi (dz) }{\e\phi(\e)}+\frac{1}{\phi(\e)}\bigg(\int_{\{|z|\le 1\}}|z|^2\,\Pi (dz) +\int_{\{|z|>1\}}|z|\,\Pi (dz) \bigg)\Bigg].
\end{align*}
In particular, by \eqref{a2-3-2}, $\lim_{\e \to 0}
 {\phi(\e)}=\infty$ and $\int_{\{|z|>1\}}|z|\,\Pi (dz) <\infty$,  it holds that $$\lim_{\e \to 0}\sup_{x\in \R^d} |H_2^\e(x)|=0.$$

(2)  For any $l\ge1$, let $f_l$ be the function defined by \eqref{e:funcOO}, and  $F_{\e,l}(s,x)=f_l\big(\Phi_\e(x)-\e^{-1}\phi(\e)^{-1}(\bar b_\infty+\bar b) s\big)$.  According to all the estimates above and  $\sup_{\e\in (0,1)}\sup_{x\in \R^d}\frac{|A_\e(x)|}{\phi(\e)}<\infty$
 (which is due to \eqref{a2-3-1-00} and the boundedness of $k(x,z)$), we can get
$$
\lim_{R \to \infty}\sup_{\e\in (0,1), x\in \R^d, s>0}
\left|\frac{\partial F_{\e,R}}{\partial s}(s,x)+\sL^\e F_{\e,R}(s,\cdot)(x)\right|=0
$$ and
$$
\sup_{\e\in (0,1), x\in \R^d, s>0}
\left|\frac{\partial F_{\e,\theta}}{\partial s}(s,x)+\sL^\e F_{\e,\theta}(s,\cdot)(x)\right|\le C(\theta),\quad \theta\in (0,1).
$$
Thus, following the proof of Theorem \ref{t3-2}, we can obtain that $\{Z^\e\}_{\e\in (0,1)}$ is tight
in $\D ([0, \infty); \R^d)$.

(3) Recall that the generator of  the process $X^\e$ is $\sL^\e$, and again that
$$F_{\e}(s,x)=f\big(x+\e\psi(x/\e)-\e^{-1}\phi(\e)^{-1}(\bar b_\infty+\bar b)s\big),\quad f_{\e,s}(x)=f\big(x-\e^{-1}\phi(\e)^{-1}(\bar b_\infty+\bar b) s\big).$$
 For every $0<s_1<s_2,\cdots<s_k<s\le t$,
$f\in C_b^3(\R^d)$ and $G\in C_b(\R^{dk})$,
\begin{align*}
&\Ee_0\Big[\Big(f(Z_t^{\e})-f(Z_s^{\e})-\int_s^t
\Big(\frac{\partial F_{\e}}{\partial r}(r,X_r^\e)+\sL^{\e}F_{\e}(r,\cdot)(X_r^\e)\Big)\,dr\Big)
G (Z_{s_1}^\e,\cdots, Z_{s_k}^\e )\Big]\\
&=\Ee_0\bigg[\Big(f(Z_t^\e)-f(Z_s^\e)-\int_s^t
\Big(\Big\langle\nabla^2 f_{\e,r} ( \Phi_\e(X_r^\e)  ),\frac{1}{\phi(\e)}A_\e( X_r^\e /\e)\Big\rangle+H_2^\e (X_r^\e)\Big)\,dr\Big)G\Big(Z_{s_1}^\e,\cdots, Z_{s_k}^\e\Big)\bigg]\\
&\to 0 \qquad \hbox{as } \e \to 0.
 \end{align*}
 Let
 $$
 A_\e=\int_{\T^d}A_\e(x)\,\mu(dx)=\int_{\T^d}\int_{\{|z|\le {\e}^{-1}\}}(z\otimes z)k(x,z)\,dz\,\mu(dx).
 $$
Then, following the argument
in \eqref{e:I1},
and using Lemma \ref{l2-4} and the fact that  $$\sup_{\e\in (0,1)}\sup_{x\in \R^d}\frac{|A_\e(x)|}{\phi(\e)}<\infty,$$ we get that
\begin{align*}
\lim_{\e \to 0}\Ee_0\Bigg[\bigg|\int_s^t \!\!\Big\langle\nabla^2
f_{\e,r} ( \Phi_\e (X_r^\e ) ), \, &\frac{1}{\phi(\e)}A_\e\left(
X_r^\e/\e \right)\Big\rangle \,dr- \int_s^t \Big\langle\nabla^2 f_{\e,r} ( \Phi_\e (X_r^\e )
),\frac{A_\e}{\phi(\e)}\Big\rangle\, dr\bigg|^2\Bigg]=0.
\end{align*}
Hence, putting all estimates together,  we
obtain
$$
\lim_{\e \to 0}
\Ee_0\Big[\Big(f(Z_t^\e)-f(Z_s^\e)-\int_s^t
\Big\langle\nabla^2 f(Z_r^\e),\frac{A_\e}{\phi(\e)}\Big\rangle \,dr\Big)
G (Z_{s_1}^\e,\cdots, Z_{s_k}^\e )\Big]=0.
$$
Given this, the fact that $\lim_{\e \to 0}\frac{A_\e}{\phi(\e)}=A$ and the tightness of $\{Z^\e\}_{\e\in (0,1)}$
in $\D([0, \infty); \R^d)$,
one can follow the proof of Theorem \ref{t3-2} to get the desired assertion.
\end{proof}

\section{Sufficient conditions for averaging assumption \eqref{e:3.6}}\label{S:6}

In this section,  we present some sufficient conditions for the key averaging assumption \eqref{e:3.6},
which is needed  for the proof of the assertions in  Example \ref{ex4-1}
and the two additional examples in Subsection \ref{S:7.1}.
The main results of this section are Theorems \ref{T:6.5} and \ref{T:6.7}.

\medskip

Let $\Pi_0(dz)$ be defined by \eqref{e:3.2a}. Let $k(x, z)$ be a
non-negative
bounded function on $\R^d\times \R^d$
so that $x\mapsto k(x,z)$ is
multivariate 1-periodic
for each fixed $z\in \R^d$ and condition \eqref{e0} holds. We will represent $z$ in the spherical coordinate $(r, \theta)$  with $r=|z|$ and $\theta = z/|z|$,
and will write $k(x, z)$ as $k(x, (r, \theta))$ as well.

\begin{proposition}\label{l6-1}
Suppose
that for every $x\in \R^d$ and $\varrho_0$-a.e.\ $\theta \in \bS^{d-1}$, there is a constant  $\bar k(x, \theta)$ so that   \begin{equation}\label{e:2a}
    \lim_{T\to \infty} \frac1T \int_0^{T}  k(x, (r, \theta))  \,dr =  \bar k(x, \theta).
  \end{equation}
  Then for any
bounded function $f:\R^d\times \R^d  \rightarrow \R$ satisfying
\eqref{e:3.4b}  and for $0<r\le R$,
\begin{equation}\label{e:8}
 \lim_{\e \to 0} \sup_{x\in \R^d} \Big| \int_{\{r\le |z|\le R\}}f(x,z)
 \left( k(x/\e  ,z/\e)  - \bar k (x/\e  ,z/|z|) \right)\,\Pi_0 (dz) \Big| =0.
\end{equation}
\end{proposition}
\begin{proof}
(1)  Let $\Lambda_0$ denote the collection of all $\theta\in \bS^{d-1}$ such
that for any $x\in \R^d$ there is a constant  $\bar k(x, \theta)$ so that \eqref{e:2a} holds. Now we are going to show that
the function
 $x\mapsto \bar k(x, \theta)$ is equi-continuous in $x$ for all $\theta \in \Lambda_0$. Moreover,
 for every $\theta \in \Lambda_0$,
 the convergence in \eqref{e:2a} is  uniform   in  $x\in \R^d $.

  For any $\eps >0$,
  by \eqref{e0},      there is a constant $\delta_0>0$ so that $|k(x, z)-k(y, z)|<\eps$ for any $x, y,z\in\R^d$ with
 $|x-y|\leq \delta_0$.   Thus  for $\theta \in \Lambda_0$,
  $$
 |\bar k(x, \theta) -\bar k(y, \theta)|
 \leq \limsup_{T\to \infty} \frac1T \int_0^{T}  |k(x, (r, \theta)) - k(y, (r, \theta))| \,dr \leq \eps .
 $$
 as long as $|x-y|\leq \delta_0$.  In other words,  $x\mapsto \bar k(x, \theta)$ is equi-continuous in $x$ for all $\theta \in \Lambda_0$.

 Let $\{y_k: 1\leq k\leq N\}\subset \T^d$
 be an $\delta_0$-net in $\T^d$ meaning that for every $x\in \T^d$ there is some $1\le j\le N$ so that $|x-y_j|< \delta_0$.
 By \eqref{e:2a}, for each $\theta \in \Lambda_0$, there is some $T_0\geq 1$ so that for every $1\leq k\leq N$,
 $$
 \Big| \frac1T \int_0^{T}  k(y_k, (r, \theta)) \, dr -  \bar k(y_k, \theta)  \Big| <\eps
 \quad \hbox{for every } T\geq T_0 .
 $$
 For  every $x\in \T^d$,  there is some $y_k$ so that $|x-y_k| <\delta_0$. Hence for  any
 $T\geq T_0$,
 \begin{align*}
 &  \Big| \frac1T \int_0^{T}  k(x, (r, \theta)) \, dr -  \bar k(x, \theta)  \Big| \\
 &\leq    \frac1T \int_0^{T} | k(x, (r, \theta)) - k(y_k, (r, \theta))|  \,dr +  \Big| \frac1T \int_0^{T}  k(y_k, (r, \theta)) \, dr -  \bar k(y_k, \theta)  \Big| + | \bar k(y_k, \theta) -\bar k (x, \theta)|  \\
 &\leq   3\eps,
 \end{align*} where in the last inequality we used \eqref{e0}   again.
 This proves that for each $\theta \in \Lambda_0$,
 \begin{equation}\label{e7}
\lim_{T \to \infty} \sup_{x\in \R^d} \Big|\frac1T \int_0^{T}  k(x, (r, \theta))  \,dr-\bar k(x, \theta)\Big|=0.
 \end{equation}

(2)   Let $f$ be a bounded function such that \eqref{e:3.4b} is satisfied.
For each
$\eps_0 >0$,
   there is $\delta_1  \in (0, 1/4]$ so that  $ |f(x, z_1)-f(x, z_2)|  < \eps_0$ whenever $|z_1-z_2|<\delta_1$.
   For $0<r<R$, we divide $[r, R]$ into $N= 1+[(R-r)/ \delta_1]$ equal subintervals with partition points $\{t_0, t_1, \cdots, t_N\}$
   with $\Delta = t_k-t_{k-1} = (R-r)/N \in
   (0, \delta_1)$.
   By taking $\delta_1$ smaller if needed (which may depend on $r$, $R$ and $\|f\|_\infty$),
   we can and do assume
   that
\begin{equation}\label{e8}
\sup_{x\in \R^d}\Big|\frac{f(x,(s,\theta))}{s^{1+\alpha}} -\frac{f(x,(t_k,\theta))}{t_k^{1+\alpha}}\Big| \leq \eps_0   \quad \hbox{for } s\in [t_{k-1}, t_k],\ \theta\in \bS^{d-1}.
\end{equation}
  Using spherical coordinates and Fatou's lemma,
     we have for any
     $0<r\le R$,
  \begin{align*}
 &  \limsup_{\e \to 0} \sup_{x\in \R^d} \Big| \int_{\{r\le |z|\le R\}}f(x,z)
( k(x/\e  ,z/\e) - \bar k (x/\e  ,  z/|z|) ) \,\Pi_0 (dz) \Big|\\
&\leq  \limsup_{\e \to 0}     \sup_{x\in \R^d}\left|
\int_{\bS^{d-1}} \int_r^R   f(x, (s, \theta)) \left(
k(x/\e ,  (s/\e, \theta) ) -\bar k (x/\e , \theta ) \right) \frac{1}{s^{1+\alpha}} \,ds  \,   \varrho_0 (d\theta)  \right|     \\
&\leq \limsup_{\e \to 0}   \sup_{x\in \R^d}\int_{\bS^{d-1}} \left| \sum_{k=1}^N \int_{t_{k-1}}^{t_k}   f(x, (s, \theta)) \left(
k(x/\e  , (s/\e, \theta) ) -\bar k (x/ \e , \theta) \right) \frac{1}{s^{1+\alpha}} \,ds  \right|  \,  \varrho_0 (d\theta)   \\
&\leq \limsup_{\e \to 0}  \sup_{x\in \R^d} \left|
 \int_{\bS^{d-1}}  \sum_{k=1}^N
 \frac{ f(x, (t_k, \theta))}{t_k^{1+\alpha}}
 \int_{t_{k-1}}^{t_k}   \left(
k(x/\e  ,  (s/\e, \theta) ) -\bar k (x/ \e , \theta) \right)   \,ds  \,  \varrho_0 (d\theta) \right|    \\
&\quad  +   2\eps_0R\|k\|_\infty\varrho_0(\bS^{d-1})\\
&\leq \limsup_{\e \to 0}  \sum_{k=1}^N  \frac{ \| f\|_\infty }{t_k^{1+\alpha}}\Bigg(
 \int_{\bS^{d-1}}  t_{k-1}\sup_{x\in \R^d} \left|
 \frac{\eps}{t_{k-1}}\int_0^{t_{k-1}/\eps}
   k(x/\e  ,  (s, \theta) )\,ds -\bar k (x/ \e , \theta)\right| \,\varrho_0 (d\theta)   \\
  & \qquad\qquad\qquad\qquad \quad+\int_{\bS^{d-1}}  t_{k}\sup_{x\in \R^d} \left|
 \frac{\eps}{t_{k}}\int_0^{t_{k}/\eps}
   k(x/\e  ,  (s, \theta ))\,ds -\bar k (x/ \e , \theta)\right| \,\varrho_0 (d\theta)  \Bigg)\\
& \quad +   2\eps_0R\|k\|_\infty\varrho_0(\bS^{d-1}) \\
&\leq     \sum_{k=1}^N  \frac{ \|f\|_\infty }{t_k^{1+\alpha}} \Bigg(
\int_{\bS^{d-1}}t_{k-1}\limsup_{\e \to 0} \sup_{x\in \R^d} \left|
 \frac{\eps}{t_{k-1}}\int_0^{t_{k-1}/\eps}
   k(x  , ( s, \theta ))\,ds -\bar k (x , \theta)\right| \,\varrho_0 (d\theta)   \\
 &\qquad\qquad\qquad + \int_{\bS^{d-1}}t_{k}\limsup_{\e \to 0} \sup_{x\in \R^d} \left|
 \frac{\eps}{t_{k}}\int_0^{t_{k}/\eps}
   k(x  ,  (s, \theta ))\,ds -\bar k (x , \theta)\right| \,\varrho_0 (d\theta) \Bigg) \\
& \quad + 2\eps_0 R\|k\|_\infty\varrho_0(\bS^{d-1}) \\
&=  2\eps_0 R\|k\|_\infty\varrho_0(\bS^{d-1}).
\end{align*}
Here in the third inequality we have used \eqref{e8} and the fact $\|\bar k\|_\infty\le \|k\|_\infty$,
while the last inequality follows from the fact that \eqref{e7}
holds for $\varrho_0$-a.e.\ $\theta\in \bS^{d-1}$. Since $\eps_0>0$ is arbitrary, we get  \eqref{e:8} immediately.
\end{proof}

 As in the proof of Proposition \ref{l6-1},
 in the following   denote by $\Lambda_0$
the collection of all $\theta\in \bS^{d-1}$ so
that  for every $x\in \R^d$  there is a constant  $\bar k(x, \theta)$ such that \eqref{e:2a} holds.

\begin{remark}\label{R:6.8}
Here are two simple cases so that $\theta\in \Lambda_0$.
\begin{itemize}
 \item[(i)]   If $\theta \in \bS^{d-1}$ has the property that  $r\mapsto k(x, (r, \theta))$ is
 multivariate $T(x)$-periodic
 for each $x\in \R^d$
(it can have different period for different $x\in \R^d$), then clearly \eqref{e:2a} holds for any function $f(x, z)$ that satisfies \eqref{e:3.4b} with $ \bar k(x, \theta)=\frac{1}{T(x)}\int_0^{T(x)} k(x,(r,\theta))\,dr$, and so $\theta \in \Lambda_0$.

\item[(ii)]  A function $\varphi (r)$ on $[0, \infty)$ is said to be almost periodic if it is the uniform limit of some periodic functions
(cf.\ \cite[p. 81]{B}). It follows from   (i) above that
if $\theta \in \bS^{d-1}$ has the property that  $r\mapsto k(x, (r, \theta))$ is almost periodic for each $x\in \R^d$,
then \eqref{e:2a} holds for any function $f(x, z)$ that satisfies \eqref{e:3.4b}, and so $\theta \in \Lambda_0$.  See
 Lemma \ref{P:5}
and its proof below for more information.
\end{itemize}
\end{remark}

Next, we will
present some sufficient conditions for $\theta\in \Lambda_0$ under the periodicity of $z\mapsto k(x,z)$.

\begin{corollary}\label{P:4}
Suppose in addition that for each $x\in \R^d$, there is some $T:=T(x)>0$ so that
$z\mapsto k(x, z)$ is  multivariate $T$-periodic.
 \begin{itemize}
\item[\rm (i)]
$\theta=(\theta_1, \cdots, \theta_d) \in \bS^{d-1}$ is in $  \Lambda_0$,
if $\theta$ is pairwise rational in the sense that each $\theta_i/\theta_j$ is a rational number whenever $\theta_j\not= 0$;

\item[\rm (ii)] If $\varrho_0$ does not charge on the set of those $\theta \in \bS^{d-1}$ that are not pairwise rational,
then
\eqref{e:8} holds for any function $f(x, z)$ that satisfies \eqref{e:3.4b}.
\end{itemize}
In particular, suppose that $\varrho_0(d\theta)=\delta_{\theta_0}(d\theta)$ for some rational point
 $\theta_0=(m_1/n,\cdots,m_d/n)\in \bS^{d-1}$, where $n,m_1,\cdots,m_d\in \mathbb{Z}$
 and
 $\delta_{\theta_0}(d\theta)$ denotes the Dirac measure on $\bS^{d-1}$.
 Then   \eqref{e:8} holds for any function $f(x, z)$ satisfying \eqref{e:3.4b} with
 $$\bar k(x,\theta)=\frac{1}{n}\int_0^n k\left(x,(r,\theta_0)\right)\,dr
 \quad \hbox{for all } \theta \in \bS^{d-1}.
 $$
\end{corollary}

\begin{proof} (i) If the measure $\theta \in \bS^{d-1}$ is pairwise rational,
then  there is some $r_0>0$ so that $r_0\theta$ has integer coordinates.
For each $x\in \R^d$,
$$
r\mapsto  k(x, (r, \theta)) = k(x, (r\theta_1, \cdots, r\theta_d))
$$
is a bounded $(r_0T)$-periodic function on $[0, \infty)$, and so \eqref{e:2a} holds with $ \bar k(x, \theta)=\frac{1}{r_0T}\int_0^{r_0T} k(x,(r,\theta))\,dr$ from Remark \ref{R:6.8}(i).

(ii) The assertion  follows immediately from (i) and  Proposition \ref{l6-1}.

Having (i) and (ii) at hand, we can easily see the validity of the last assertion.
\end{proof}

Recall that $\theta =(\theta_1, \cdots, \theta_d) \in \bS^{d-1}$ is said to be rationally dependent if there is some non-zero
$m=(m_1, \cdots, m_d) \in
\Z^d$ so that $\<m, \theta\>= \sum_{i=1}^d m_i\theta_i=0$.
Otherwise, we call $\theta$ rationally independent. When $d=1$, $\bS^{0}=\{1, -1\}$ so every its member is rationally independent.

\begin{lemma}\label{P:5}
Suppose that $f(x)=f(x_1, \cdots, x_d)$ is a continuous
multivariate $1$-periodic
function on $\R^d$.
Then for each $\theta \in \bS^{d-1}$, there is a constant $C(\theta)$ so that
 \begin{equation}\label{e:3}
 \lim_{T\to \infty} \frac1T \int_0^{T}  f(\theta t) \,dt = C(\theta).
  \end{equation}
Set
$$
\Gamma_f= \left\{ \theta \in \bS^{d-1}: C(\theta ) =\int_{\T^d} f(x) \,dx \right\}.
$$
Then every rationally independent $\theta \in \bS^{d-1}$ is in $\Gamma_f$. In particular,
 $\Gamma_f=\{1, -1\}$  if $d=1$, $\bS^1\setminus \Gamma_f$ is countable if  $d=2$,
and ${\rm dim}_H (\bS^{d-1}\setminus \Gamma_f) \leq d-2$ if $d\geq 3$.  Here  ${\rm dim}_H$ stands for the
Hausdorff dimension.
\end{lemma}

\begin{proof}
The result is trivial when $d=1$.  So we assume $d\geq 2$ in the rest of the proof. Let $\langle\cdot,\cdot \rangle$ denote the inner product in $\R^d$.
 Define $f(x)=\sum_{k\in \Z^d: |k|\leq N}   c_k e^{i 2\pi  \langle k,  x\rangle}$ with
 $$
  c_k=  \int_{\T^d}  e^{-i 2\pi \langle k, x\rangle }  f(x) \,dx.
 $$ Then, for   $ \theta   \in \bS^{d-1}$,
$$
  f (\theta t ) = \sum_{k\in \Z^d: |k|\leq N}   c_k e^{i 2\pi \langle k, \theta\rangle t }.
$$
 Clearly
 \begin{equation}\label{e:1}
 C(\theta ):=\lim_{T\to \infty} \frac1T \int_0^{T} f (\theta t)\, dt =\sum_{k\in \Z^d: |k|\leq N, \langle k, \theta \rangle=0} c_k.
 \end{equation}

 Note that for each non-zero $k\in \Z^d$ and $d\geq 2$, the set $\{ \theta \in \bS^{d-1}: \langle k, \theta \rangle=0\}$ is a two-point set when $d=2$, and has Hausdorff dimension $d-2$ when $d\geq 3$.
 Noting also that
 \begin{equation}\label{e:7a}
 \bS^{d-1}\setminus \Gamma_f\subset  \{ \theta \in \bS^{d-1}: \langle k, \theta \rangle =0  \hbox{ for some non-zero } k\in \Z^d\},
 \end{equation}
 we have  $\bS^{d-1}\setminus \Gamma_f$
 is a countable set when $d=2$,
and ${\rm dim}_H (\bS^{d-1}\setminus \Gamma_f) \leq d-2$ for $d\geq 3$.

Now, suppose that $f(x_1, \cdots, x_d)$ is a continuous multivariate $1$-period function on $\R^d$.
It can then be viewed as a continuous function on $\T^d$. By the Stone-Weierstrass theorem,
it can be uniformly approximated by functions of the form $\sum_{k\in \Z^d: |k|\leq N}   c_k e^{i 2\pi  \langle k,  x\rangle }$
on $\R^d$,
see e.g.\ \cite[p.\ 26]{B}.
Thus for any $\theta=(\theta_1, \cdots, \theta_d ) \in \bS^{d-1}$,
$$
  f(\theta t) = f(\theta_1t, \cdots, \theta_d t)
$$
 can be approximated uniformly by the functions of the form
$\sum_{k\in \Z^d: |k|\leq N}  c_k e^{i 2\pi \langle k, \theta \rangle t}$.
    It follows from \eqref{e:1}   that for any   $\theta \in \bS^{d-1}$,
  there is a constant $C(\theta)$ so that
  \begin{equation}\label{e:3}
 \lim_{T\to \infty} \frac1T \int_0^{T}  f (\theta t) \,dt = C(\theta),
 \end{equation}
and \eqref{e:7a} continues to hold for this $f$.  In particular, $C(\theta) = \int_{ \T^d} f(x)\,dx$ if $\theta\in \bS^{d-1}$ is rationally independent.
 The  assertion of the proposition  now follows.
\end{proof}

   \begin{theorem}\label{T:6.5}
  Suppose that $k(x, z)$ is jointly continuous on $\R^d\times \R^d$ and
 $k(x, z)$ is multivariate $1$-periodic both in $x$ and in $z$. Then
 \begin{itemize}
 \item[\rm (i)]
 $\Lambda_0=\bS^{d-1}$; that is,
 \eqref{e:2a} holds for every $\theta \in \bS^{d-1}$ and $x\in \R^d$ with some $\bar k(x,\theta)$.

 \item[\rm (ii)] Let
 $$
 \bar k(x) = \int_{\T^d} k(x, z) \,dz,\quad x\in \R^d.
 $$
 Then for each $x\in \R^d$, $\bar k (x, \theta) = \bar k (x)$ for every rationally independent $\theta \in \bS^{d-1}$.
 In particular we have for every $x\in \R^d$,
 $\bar k(x, 1) =  \bar k(x,  -1)=\bar k(x)$ when $d=1$, $\{\theta \in \bS^1:   \bar k(x, \theta)\not= \bar k(x)\}$
 is countable when $d=2$, and the Hausdorff dimension of $\{\theta \in \bS^{d-1}:   \bar k(x, \theta)\not= \bar k(x)\}$
 is no larger than $d-2$.

 \item[\rm (iii)]
  Property  \eqref{e:8}
  holds for any function $f(x, z)$ that satisfies \eqref{e:3.4b}.
 \end{itemize}
 \end{theorem}

 \begin{proof}
 This follows directly by applying Lemma \ref{P:5}  to function $z\mapsto k(x, z)$ and by Proposition \ref{l6-1}.
 \end{proof}

  \begin{remark}\label{R:7} We present two explicit cases that
   Theorem
  \ref{T:6.5} applies.

   \begin{itemize}
 \item[(i)]  Assume that $\Pi_0(dz)$ is absolutely continuous with respect to the Lebesgue measure on $\R^d$;
 or equivalently, $\varrho_0$ is absolutely continuous with respect to  the Lebesgue surface measure $\sigma$  on $\bS^{d-1}$. Then under the assumptions of
 Theorem
 \ref{T:6.5},
 \eqref{e:8} holds with $\bar k(x,\theta)=\bar k(x):= \int_{\T^d} k(x, z)\, dz$ for all $x\in \R^d$ and $\theta\in \bS^{d-1}$;
 that is,
 \begin{equation}\label{e:9a}
 \lim_{\e \to 0} \sup_{x\in \R^d} \Big| \int_{\{r\le |z|\le R\}}f(x,z)
 \left( k(x/\e  ,z/\e)  - \bar k  (x/\e ) \right)\,\Pi_0 (dz) \Big| =0.
\end{equation}
   We emphasize that for this result we do not assume the boundedness of the
 Radon-Nikodym derivative $\frac{ \varrho_0 (d\theta)}{\sigma (d\theta )}$.

 \item[(ii)]
In fact the conclusion \eqref{e:9a} holds for
 any  finite measure $\varrho_0$ on $\bS^{d-1}$ that does not charge on the set of rationally dependent $\theta \in \bS^{d-1}$.
 In particular,  if  $\varrho_0$ does not charge on singletons
 when $d=2$ and does not charge on
 subsets of $\bS^{d-1}$ that are of Hausdorff dimension $d-2$ when $d\geq 3$
 (for example, $\varrho_0$ is $\gamma$-dimensional Hausdorff measure with $\gamma   \in (d-2, d-1]$)
 then \eqref{e:9a} holds for any function $f(x, z)$ that satisfies \eqref{e:3.4b}.
\end{itemize}
 \end{remark}

 We can drop the continuous assumption on
 $z\mapsto k(x,z)$ in Theorem \ref{T:6.5}(iii)
 when
  the spherical measure $\varrho_0$ in $\Pi_0$ is absolutely continuous with respect to $\sigma$ with bounded  Radon-Nikodym derivative.

\begin{theorem} \label{T:6.7}
Suppose that   $\varrho_0$ is absolutely continuous with respect  to the Lebesgue surface measure $\sigma$  on $\bS^{d-1}$   with bounded
 Radon-Nikodym derivative.
 Let $k(x, z)$ be a bounded function on $\R^d\times \R^d$ such that $z\mapsto k(x,z)$ is $1$-periodic for each fixed $x\in \R^d$
and \eqref{e0} is true.
Then
 \eqref{e:8}  holds for any function $f(x, z)$ satisfying \eqref{e:3.4b} with $\bar k(x,\theta)=\bar k (x):= \int_{\T^d} k(x, z)\, dz$ for all $x\in \R^d$ and $\theta\in \bS^{d-1}$.
 \end{theorem}

\begin{proof}
 Let $\varphi \geq 0$ is a smooth function with compact support in $\R^d$ having $\int_{\R^d} \varphi (y)\, dy =1$.
  For $\delta >0$, let  $\varphi_\delta (y) := \delta^{-d} \varphi (y/\delta)$.   Define
  $$
  k_\delta (x, z) = \int_{\R^d} k(x, z-y) \varphi_\delta (y)\, dy.
  $$
  Clearly $k_\delta (x, z)$ is a bounded, multivariate 1-periodic and continuous function on $\R^d\times \R^d$.
   Thus by Remark \ref{R:7},
   \begin{equation}\label{e:11}
 \lim_{\e \to 0} \sup_{x\in \R^d} \Big| \int_{\{r\le |z|\le R\}}f(x,z)
 \left( k_\delta (x/\e  ,z/\e)  - \bar k_\delta (x/\e  ) \right)\,\Pi_0 (dz) \Big| =0
\end{equation}
   for any function $f(x, z)$ that has property \eqref{e:3.4b}, where
   $\bar k_\delta (x ) := \int_{[0,1]^d} k_\delta (x, z) \,dz$.
  Clearly,
   $$
   \lim_{\delta \to 0} \int_{\T^d} | k(x, z) - k_\delta (x, z) |  \,dz =0,
   $$
   Condition \eqref{e0} implies that the above convergence is uniform in $x\in \R^d$.
   Furthermore,
      $ \bar k (x):= \int_{\T^d} k(x, z)\, dz $ is uniformly continuous in $x$,
          and $ \bar k_\delta (x) $ converges to $\bar k(x)$ uniformly as $\delta \to 0$.
    Observe that by
    the fact that $\varrho_0$ is absolutely continuous with respect to $\sigma$  on $\bS^{d-1}$ with bounded
 Radon-Nikodym derivative and the multivariate 1-periodicity of
 $(x,z)\mapsto k(x, z)$,
    it holds that for any $0<r<R$
    and $\eps \in (0, 1)$,
    \begin{align*}
    &    \int_{\{r\le |z|\le R\}} | k (x/\e  ,z/\e)  -k_\delta (x/\e  ,z/\e)  | \,\Pi_0 (dz)   \\
    &\leq   c_1 r^{-(d+\alpha)} \int_{\{r\le |z|\le R\}} | k (x/\e  ,z/\e)  -k_\delta (x/\e  ,z/\e)  | \,dz    \\
    &\leq  c_2  r^{-(d+\alpha)}  |B(0, R)\setminus B(0, r)| \int_{\T^d} | k(x, z) - k_\delta (x, z) |\,  dz,
    \end{align*}
 where $c_1$ and $c_2$ are two positive constants  that are independent of $\eps\in (0, 1)$,  $\delta\in (0, 1)$
 and $0<r<R$.
 Thus we have
    \begin{align*}
 &  \lim_{\e \to 0} \sup_{x\in \R^d}  \Big| \int_{\{r\le |z|\le R\}}f(x,z)
 \left( k  (x/\e  ,z/\e)  - \bar k (x/\e  ) \right)\,\Pi_0 (dz) \Big| \\
 &\leq  \lim_{\e \to 0} \sup_{x\in \R^d}   \| f\|_\infty   \int_{\{r\le |z|\le R\}}
 | k (x/\e  ,z/\e)  -k_\delta (x/\e  ,z/\e)  | \,\Pi_0 (dz)   \\
&\quad  +  \lim_{\e \to 0} \sup_{x\in \R^d}  \Big| \int_{\{r\le |z|\le R\}}f(x,z)
 \left( k_\delta (x/\e  ,z/\e)  - \bar k_\delta (x/\e  ) \right)\,\Pi_0 (dz) \Big| \\
 &\quad + \lim_{\e \to 0} \sup_{x\in \R^d}
  \| f\|_\infty | \bar k_\delta (x/\eps) -  \bar k (x/\eps)| \,
 \Pi_0 (r\leq |z| \leq R)     \\
  &\le c_2 \|f\|_\infty  r^{-(d+\alpha)}  |B(0, R)\setminus B(0, r)| \sup_{x\in \R^d}\int_{\T^d} | k(x, z) - k_\delta (x, z) |\,  dz\\
  &\quad+
 \| f\|_\infty  \Pi_0 (r\leq |z| \leq R)
   \sup_{x\in \R^d}   | \bar k_\delta (x ) -  \bar k (x )|.   \end{align*}
Letting $\delta \to 0$ in the right hand side of the inequality above proves the result.
\end{proof}

\section{Examples and comments}\label{S:7}

\subsection{Examples}\label{S:7.1}
In this subsection, we first give the proof
of the assertions in Example \ref{ex4-1}, and then present  two additional examples
to further illustrate the applications
of our main results.  Example \ref{ex4-1} together with two examples below show that the
periodic homogenization of jump processes is very different from that of
diffusion processes. In the homogenization of jump processes, large
jumps play a key role on the homogenized process.
The scale function $\varphi$ is determined by
the tail of the jumping kernel.

\begin{proof}[Proof of Example $\ref{ex4-1}$]
(i) Suppose that $\alpha\in (0,2)$ and that
$k(x, z)$ is a bounded continuous function on $\R^d\times \R^d$ so that
 $x\mapsto k(x,z)$ is
 multivariate $1$-periodic
 for each fixed $z\in \R^d$, $z\mapsto k(x,z)$ is
 multivariate $1$-periodic
 for each fixed $x\in \R^d$
 and \eqref{e0} is true.
Clearly $\varphi (r):= r^{\alpha}$
satisfies \eqref{e:3.1}.
Then it is easy to see that $\Pi (dz) $ defined by \eqref{ex4-1-1} has the expression \eqref{e:3.3a} with
$\varrho_0(d\theta)$, $\varphi(r)$  given above and $\kappa(r,d\theta) \equiv 0 $.
Furthermore, we know by Theorem  \ref{T:6.5} that
 \eqref{e:3.6} holds with $\bar k(x,z)=\bar k(x,z/|z|)$ given by \eqref{e:notenote}.
Therefore, the claimed assertions in this example follow readily from Theorem \ref{t3-1},
Theorem \ref{t3-2}, Remark \ref{R:7}(ii) and Theorem
\ref{T:6.7}.

(ii) Suppose that $\alpha=2$ and
$\lim_{|z|\rightarrow \infty}\bar k(z)=k_0$,
where $ \bar k(z)=\int_{\T^d}k(x,z)\,\mu(dx)$. Then
$$
\lim_{\e \to 0} \frac{1}{|\log {\e}|}\left|\int_{\{|z|\le
 1/{\e} \}}(z\otimes z)\big(\bar
k(z)-k_0\big)\, \Pi (dz)\right|=0
$$
and
\begin{align*}
\lim_{\e \to 0}\frac{1}{|\log {\e}|}\int_{\{1\le |z|\le
1/{\e} \}}z_i z_j\,
\Pi (dz) =
\int_{\bS^{d-1}}\theta_i\theta_j\,\varrho_0(d\theta).
\end{align*}
Consequently,
$$
\lim_{\e \to 0}\frac{1}{|\log {\e}|}\int_{\{|z|\le
1/{\e} \}}(z\otimes z)\bar k(z)\,\Pi (dz) = A,
$$
 where $A=\{a_{ij}\}_{1\le i,j\le d}$ with $$a_{ij}:=k_0\int_{\bS^{d-1}}\theta_i\theta_j\,\varrho_0(d\theta).$$
 Thus, \eqref{a2-3-1-00} and \eqref{a2-3-1}  hold with $\phi(\e)=|\log {\e}|$.
Furthermore, it is easy to see that \eqref{a2-3-2} holds for $\phi(\e)$. Then, the
assertion follows from  Theorem \ref{t3-4}.

(iii) If $\alpha>2$, then $\int_{\R^d}|z|^2 \,\Pi (dz) <\infty$, so the
desired assertion immediately follows from Theorem \ref{t3-3}.
\end{proof}

\begin{remark}  We call a subset $\Gamma \subset \R^d$  an unbounded  generalized  cone, if $\lambda \Gamma \subset \Gamma$ for every  $\lambda \geq 0$.  Note that $\Gamma$ can have several branches starting from the origin,  and it can be non-symmetric.
Let $\sigma(d\theta)$ denote the Lebesgue surface measure on
$\bS^{d-1}$.
  If  $$\I_{\{|z|>1\}}\Pi(dz)=\frac{1}{|z|^{d+\alpha}}\I_{\{|z|>1: z\in \Gamma\}}\,dz=
\frac{1}{r^{1+\alpha}}\I_{\{r>1,\theta\in \Gamma\}}\,dr\,\sigma(d\theta)$$ for some
generalized cone $\Gamma$ with $\sigma(\Gamma\cap \bS^{d-1})>0$ in Example \ref{ex4-1}(i),
then the generator $\sL_0$ of the limit process
$(\bar X_t)_{t\ge 0}$ is given by
 \begin{align*}
\bar\sL_0 f(x)=
\begin{cases}
 \displaystyle\int_{\R^d}\left(f(x+z)-f(x)\right)\frac{1}{|z|^{d+\alpha}}\I_{\Gamma } (z)\,dz,
 &\quad  \alpha\in (0,1),\\
  \displaystyle\int_{\R^d}\left(f(x+z)-f(x)-\langle \nabla f(x),z\rangle\I_{\{|z|\le 1\}} \right)\frac{1}{|z|^{d+1}}
  \I_{\Gamma } (z)\,dz,
  &\quad\ \alpha=1,\\
  \displaystyle\int_{\R^d}\left(f(x+z)-f(x)-\langle \nabla f(x),z\rangle \right)\frac{1}{|z|^{d+\alpha}}
 \I_{\Gamma } (z)\,dz,
  &\quad \alpha\in (1,2).
\end{cases}
\end{align*}
This gives us
another concrete
example that the jumping kernel limiting process $(\bar X_t)_{t\ge 0}$ can be degenerate.
\end{remark}

In the following, we
always  suppose that Assumptions {\bf(A1)}, {\bf(A2)} and {\bf(A3)}  hold, and that
$k(x, z)$ is a  non-negative  bounded function on $\R^d\times \R^d$ such that $x\mapsto k(x,z)$ is  multivariate $1$-periodic
for each fixed $z\in \R^d$
and \eqref{e0} is true.
We refer the reader
 to Subsection \ref{S:7.2} for conditions on small jumps of the jumping kernel such that
all {\bf(A1)}, {\bf(A2)} and {\bf(A3)} are satisfied.  Let
 $(X_t)_{t\ge 0}$ be the strong Markov process corresponding to the operator $\sL$ given by
\eqref{e4-1}. Let $\mu(dx)$ be the stationary probability measure for the quotient process of
$X$ on $\T^d$. Let
$b_R(x)$, $b_\infty(x)$, $\bar b_{R}$
(with $R>1$) and $\bar b_\infty$ be defined by \eqref{e:1.2}, \eqref{e:1.3} and \eqref{e0a}, respectively. Let $\sigma(d\theta)$ denote the Lebesgue surface measure on
$\bS^{d-1}$.

\begin{example}\label{ex4-3}
Let $a_0(\theta)$ be a
non-negative bounded function defined on the unit sphere $\bS^{d-1}$.
Suppose that
$$
\I_{\{|z|>1\}}\,\Pi (dz) =
\frac{a_0 (z/|z|)}
{|z|^d \Phi(|z|)}\I_{\{|z|>1\}}\,dz,$$ where
\begin{equation}\label{e:6.1}
\Phi(r):=\int_{\alpha_1}^{\alpha_2} r^\alpha \,\nu(d\alpha)
\end{equation}
for constants $0<\alpha_1<\alpha_2<2$ and a  non-negative finite measure $\nu$ on
$[\alpha_1,\alpha_2]$
such that $\alpha_2 \in {\rm supp}[\nu]$ (that is, $\nu ((\alpha_2-\e, \alpha_2])>0$
for any $\e >0$).
Suppose also that for every fixed $x\in \R^d$, $k(x,\cdot):\R^d \rightarrow \R_+$ is multivariate 1-periodic and satisfies \eqref{e0}.

For any $\e\in (0,1]$, define $(X_t^\e)_{t\ge 0}=
(\e X_{\Phi (1/\e)t})_{t\ge 0}$,
and
$(Y_t^\e)_{t\ge 0}$ by
\begin{align*}Y_t^\e=\begin{cases}
X_t^\e ,& 0<\alpha_2<1,\\
X_t^\e-\e\Phi (1/\e)(\bar b_{1/\e}+\bar b) t,&\alpha_2=1,\\
X_t^\e-\e\Phi (1/\e)(\bar b_\infty +\bar
b)t,&1<\alpha_2<2.\end{cases}
\end{align*}
Then, the
process $(Y_t^\e)_{t\ge0}$ converges weakly in $\D ([0, \infty); \R^d)$,
as $\e \to 0$, to a (possibly non-symmetric) $\alpha_2$-stable
process with jumping measure $\bar k_0a_0(z/|z|)|z|^{-d-\alpha_2}\,dz,$ where $\bar k_0:=\int_{\T^d}\int_{\T^d} k(x,z)\,dz\,\mu(dx)$.
\end{example}

\begin{proof}
Let $\varphi (r)=\Phi (r)$. Clearly, $\varphi (r)$ is a strictly increasing function on $(1, \infty)$. We claim that it satisfies condition \eqref{e:3.1} with $\alpha=\alpha_2$. For any $\eta \in (0, \alpha_2)$,
since $\nu ((\alpha_2-\eta, \alpha_2])>0$,
$$
 \varphi (r) \asymp \int_{\alpha_2-\eta}^{\alpha_2} r^\alpha \,\nu (d\alpha)
\quad \hbox{for } r> 1.
$$
Thus for $r>1$ and $\lambda\ge1$,
$$
\lambda^{\alpha_2-\eta} \varphi (r) \preceq \lambda^{\alpha_2-\eta} \int_{\alpha_2 -\eta}^{\alpha_2}
r^\alpha \,\nu (d\alpha)
\preceq   \varphi (\lambda r) \preceq \lambda^{\alpha_2}
 \int_{\alpha_2 -\eta}^{\alpha_2} r^\alpha\, \nu (d\alpha) \leq
 \lambda^{\alpha_2} \varphi (r).
$$
Hence we have shown that for any $\eta \in (0, \alpha_2)$, there is a positive constant $c_0=c_0(\eta)\le1$
so that
\begin{equation}\label{e:6.2}
c_0 (R/r)^{\alpha_2 -\eta} \leq \frac{\varphi (R)}{\varphi (r)} \leq c_0^{-1} (R/r)^{\alpha_2}
\quad \hbox{for any } R\geq r>1.
\end{equation}
Furthermore, for any $\eta >0$ sufficiently small,
clearly we have for every $r> 1$,
$$
 \limsup_{\lambda \to \infty} \frac{\varphi (\lambda r)}{\varphi (\lambda)}=
 \limsup_{\lambda \to \infty}  \frac{\int_{\alpha_1}^{\alpha_2} \lambda^\alpha r^\alpha \,\nu (d\alpha)}
{\int_{\alpha_1}^{\alpha_2} \lambda^\alpha  \,\nu (d\alpha)} \leq r^{\alpha_2}.
$$
On the other hand, since $\nu ((\alpha_2-\eta, \alpha_2])>0$,
$$
\liminf_{\lambda \to \infty} \frac{\varphi (\lambda r)}{\varphi (\lambda)}
\succeq\liminf_{\lambda \to \infty}  \frac{\int_{\alpha_2-\eta}^{\alpha_2} \lambda^\alpha r^\alpha \,\nu (d\alpha)}
{\int_{\alpha_2-\eta}^{\alpha_2} \lambda^\alpha \, \nu (d\alpha)}
\geq r^{\alpha_2-\eta} \quad \hbox{for } r> 1.
$$
Since the above holds for every sufficiently small $\eta>0$, passing $\eta \to 0$ yields that
$\displaystyle
\liminf_{\lambda \to \infty} {\varphi (\lambda r)}/{\varphi (\lambda)}=r^{\alpha_2}$
for  $r>1$.
Hence we get
$$
\lim_{\lambda \to \infty} \frac{\varphi (\lambda r)}{\varphi (\lambda)}=r^{\alpha_2}
 \quad \hbox{for } r\geq 1.
$$
 This together with \eqref{e:6.2} proves
 the claim that \eqref{e:3.1} holds with $\alpha_2$
in place of $\alpha$ there.
On the other hand, it follows from Theorem \ref{T:6.7}
that \eqref{e:3.6} holds  with $\bar k(x,z)=\bar k(x):=\int_{\T^d}k(x,u)\,du$ for all $x,z\in \R^d$.
 The desired assertions now follows from Theorems
\ref{t3-1} and \ref{t3-2}, after noticing that $\Pi (dz) $ has the representation
\eqref{e:3.3a} with $\varrho_0(d\theta)=a_0(\theta)\,\sigma(d\theta)$ and $\kappa (r,d\theta)\equiv0$, where $\sigma(d\theta)$ denotes the Lebesgue surface measure on
$\bS^{d-1}$.
\end{proof}

\begin{remark}\label{r6-6} (1)
If $\nu(d\eta )=\delta_{\{\alpha\}} (d\eta) + \delta_{\{\beta\}} (d\eta) $ with
$0<\beta<\alpha<2$ in \eqref{e:6.1}, then
$\I_{\{|z|>1\}}\,\Pi (dz)$  in Example \ref{ex4-3}  is reduced to
$$
 \frac{a_0(z/|z|) }{|z|^{d+\alpha}+|z|^{d+\beta}} \,  \I_{\{|z|>1\}}\,dz.
$$
  In this case, $\I_{\{|z|>1\}} \,\Pi (dz)$ admits the expression
 \eqref{e:3.3a} with
 $\varphi (r)=r^\alpha + r^\beta$, $\varrho_0(d\theta)=a_0(\theta)\sigma (d\theta) $ and
 $\kappa(r, d\theta) \equiv 0$.
  If we take
 $\varphi_1 (r)=r^\alpha$, $\varrho_0(d\theta)=a_0(\theta)\sigma (d\theta) $ and
 $\kappa_1 (r, d\theta) := \frac{r^\beta }{ r^\alpha + r^\beta } \, a_0(\theta) \sigma (d\theta)$,
 then $\I_{\{|z|>1\}} \Pi (dz) $ can be also
 represented by \eqref{e:3.3a}
with $\varphi_1$  and $\kappa_1(r, d \theta)$ in place of $\varphi$ and $\kappa(r, d \theta)$.
Thus the homogenization result for $X$ holds with both $\varphi$ and $\varphi_1$ as its time scaling function.

\medskip

(2) If $\nu (d\eta)$ is the Lebesgue measure on  $[\alpha/2, \alpha]$
for some $\alpha\in(0,2)$,
then
$$
\I_{\{|z|>1\}}\Pi (dz) =
\frac{a_0(z/|z|)}{|z|^d (|z|^{\alpha} -|z|^{\alpha/2})\log |z| } \I_{\{|z|>1\}}\,dz.
$$
In this case,  $\I_{\{|z|>1\}} \Pi (dz) $ admits the expression \eqref{e:3.3a}
with $\varphi(r)=\Phi(r)=(r^\alpha -r^{\alpha/2}) \log r $,
$\varrho_0(d\theta)=a_0(\theta)\,\sigma(d\theta)$ on $\bS^{d-1}$ and
$\kappa (r,d\theta)\equiv 0$.
If we take
$$
\varphi_1 (r)=r^\alpha  \log r \I_{\{r>1\}},
$$
then clearly $\lim_{r\to \infty} \varphi_1 (r)/\varphi (r)=1$.
Thus by Remark \ref{R:3.1}, $\I_{\{|z|>1\}} \Pi (dz) $ can also be represented by \eqref{e:3.3a}
with $\varphi_1$ in place of $\varphi$ (but with different $\kappa (r,d\theta)$); that is, we can write
$$
\I_{\{|z|>1\}}\Pi (dz) =
\I_{\{|z|>1\}}\frac{a_0(\theta)+\kappa_1(r,d\theta)}{r^{d+\alpha}\log r}\,\sigma(d\theta)\,dr,
$$
where $\kappa_1(r,\theta)$ satisfies \eqref{e:3.4}.  In particular,
the homogenization result for $X$ holds with both $\varphi$ and $\varphi_1$ as its time scaling function.

(3) Similar to these of Example
\ref{ex4-1}, we can get the assertion when the jumping measure $\Pi (dz) $ enjoys the form
$$\I_{\{|z|>1\}}\Pi (dz) =
\I_{\{|z|>1\}}\frac{a_0(z/|z|)}{|z|^{d+\alpha}\log |z|}\,dz$$
with $\alpha\ge2$.  In details, when $\alpha=2$,  define
$$
Y_t^\e:=\e X_{\e^{-2}|\log \log {\e}|t} - \e^{-1}|\log \log {\e}|(\bar b_\infty+\bar b) t,\quad t\ge0.
$$
Suppose that \eqref{ex4-1-3} holds for some $k_0>0$.
Then, as $\e \to 0$, $(Y_t^\e)_{t\ge 0}$ converges weakly in $\D ([0, \infty); \R^d)$ to
Brownian motion $(B_t)_{t\ge0}$ with the covariance matrix
$A=\{a_{ij}\}_{1\le i,j\le d}$, where $$a_{ij}= k_0 \int_{\bS^{d-1}}\theta_i\theta_ja_0(\theta) \,\sigma (d\theta).$$
When $\alpha>2$, we define
$$
Y_t^\e:=\e X_{t/\e^{2}} - \e^{-1}(\bar b_\infty+\bar b) t,\quad t\ge0.
$$
Then, as $\e \to 0$, $(Y_t^\e)_{t\ge 0}$ converges weakly in $\D ([0, \infty); \R^d)$ to
Brownian motion $(B_t)_{t\ge0}$ with the covariance matrix $A$ defined
by \eqref{ex4-1-2}.
\end{remark}

The following example is concerned
with  the homogenization for jump process with a singular jumping kernel.

\begin{example}\label{ex4-0}
 Suppose that
\begin{equation}\label{ex4-0-0}
\I_{\{|z|>1\}}\Pi (dz) =\sum_{i=1}^d\frac{1}{r^{1+\alpha}}\delta_{e_i}(d\theta)
\I_{\{|r|>1\}}\,dr,
\end{equation}
where $\{e_i\}_{i=1}^d$ is the standard orthonormal basis of $\R^d$
and $\delta_{\theta_0}(d\theta)$ denotes the Dirac measure on $\bS^{d-1}$ concentrated at $\theta_0\in \bS^{d-1}$.
\begin{itemize}
\item [(i)] Suppose that $z\mapsto k(x,z)$ is
multivariate 1-periodic
for each fixed $x\in \R^d$.
For any $\e\in (0,1)$, define
 $(Y_t^\e)_{t\ge 0}$ by
\begin{align*}Y_t^\e=\begin{cases}
\e X_{t/\e^{\alpha}} ,& 0<\alpha<1,\\
\e X_{t/\e^{\alpha}}-(\bar b_{1/\e}+\bar b) t,&\alpha=1,\\
\e X_{t/\e^{\alpha}}-\e^{1-\alpha}(\bar b_\infty +\bar
b)t,&1<\alpha<2.\end{cases}
\end{align*} Then the
process $(Y_t^\e)_{t\ge0}$ converges weakly in $\D ([0, \infty); \R^d)$, as $\e \to 0$, to a non-symmetric $\alpha$-stable
process $(\bar X_t)_{t\ge0}$  with infinitesimal generator $\sL_0$ as follows
\begin{align*}
\sL_0 f(x)=
\begin{cases}
 \displaystyle\sum_{i=1}^d\int_{0}^\infty\left(f(x+z_ie_i)-f(x)\right)
 \frac{\bar k_i}{z_i^{1+\alpha}}\,dz_i,
 &\quad  \alpha\in (0,1),\\
  \displaystyle\sum_{i=1}^d\int_{0}^\infty\left(f(x+z_ie_i)-f(x)-\frac{\partial f(x)}{\partial x_i}\cdot z_i
  \I_{\{0<z_i\le 1\}}\right)
  \frac{\bar k_i}{z_i^{1+\alpha}}\,dz_i
  &\quad\ \alpha=1,\\
  \displaystyle\sum_{i=1}^d\int_{0}^\infty\left(f(x+z_ie_i)-f(x)- \frac{\partial f(x)}{\partial x_i}\cdot z_i\right)
     \frac{\bar k_i}{z_i^{1+\alpha}}\,dz_i
  &\quad \alpha\in (1,2),
  \end{cases}
  \end{align*}
 where $$\bar k_i:=\int_{\T^d}\int_0^1 k\left(x,(0,\cdots,z_i,\cdots,0)\right)\,dz_i\,\mu(dx).$$

\item [(ii)] When $\alpha=2$, we define
$$
 Y_t^\e:=\e X_{\e^{-2}| \log {\e}|t} - \e^{-1}| \log {\e}|(\bar b_\infty+\bar b) t,\quad t\ge0.
$$
Suppose that \eqref{ex4-1-3} holds for some $k_0>0$.
Then, as $\e \to 0$, $(Y_t^\e)_{t\ge 0}$ converges weakly in $\D ([0, \infty); \R^d)$ to
Brownian motion
$(B_t)_{t\ge0}$ with the covariance matrix
$A:=k_0{\rm I}_{d\times d}$,
where ${\rm I}_{d\times d}$ denotes the $d\times d$ identity matrix.

\item [(iii)] When $\alpha>2$, we define
$$
Y_t^\e:=\e X_{t/\e^2} - \e^{-1}(\bar b_\infty+\bar b) t,\quad t\ge0.
$$
Then, as $\e \to 0$, $(Y_t^\e)_{t\ge 0}$ converges weakly in $\D ([0, \infty); \R^d)$ to
Brownian motion $(B_t)_{t\ge0}$ with the covariance matrix $A$ defined
by \eqref{ex4-1-2}.
\end{itemize}
\end{example}
\begin{proof}
By \eqref{ex4-0-0} we know that \eqref{e:3.3a} holds with $\varphi(r)=r^\alpha$,
$\varrho_0(d\theta)=\sigma_0 (d\theta)$
and $\kappa(r,d\theta)\equiv0$.
According to Corollary \ref{P:4} below we know that \eqref{e:3.6} holds with $\bar k(x,e_i)=\int_0^1 k\left(x,(0,\cdots,z_i,\cdots,0)\right)\,dz_i$
for $1\le i \le d$, and $\bar k(x,\theta)=0$ for any $\theta\in \bS^{d-1}\backslash\{e_i\}_{1\le i\le d}$.
Hence the desired assertion in (i) follows from
Theorems \ref{t3-1} and \ref{t3-2}.
The proofs of (ii) and (iii) are similar to these of Example
\ref{ex4-1}.
\end{proof}

\subsection{Comments on assumptions {\bf (A1)}, {\bf (A2)} and {\bf (A3)}}\label{S:7.2}
Assumptions {\bf (A1)}, {\bf (A2)} and {\bf (A3)} are closely related with
recent developments on
the fundamental solution
of the  L\'evy type operators. For example, in \cite{CZ} the
authors considered the following L\'evy-type operator on $\R^d$:
$$\sL f(x)=\lim_{\delta \to0}\int_{\{|z|> \delta\}} (f(x+z)-f(x))
\frac{k(x,z)}{|z|^{d+\alpha}}\,dz,$$ where $0<k_1\le k(x,z)\le k_2$,
$k(x,z)=k(x,-z)$ and $|k(x,z)-k(y,z)|\le k_3|x-y|^\beta$ for some
constants $k_i>0$ $(i=1,2,3)$ and $\beta\in (0,1)$. Later the results of \cite{CZ} are
extended to time-dependent cases in \cite{CZ1} such that the symmetric assumption
in $z$ for the function $k(x,z)$ are not required; moreover, the corresponding results for the
perturbation by a drift term $b(x)$ belonging to some Kato's class when
$\alpha\in (1,2)$ are also considered
there,
see \cite[Theorem 1.5]{CZ1}.
For the critical case (i.e., $\alpha=1$), one can refer to \cite{XZ}. See \cite{Peng, KSV, CCW} and the
references therein for more recent progress on this topic, including the case that
a large class of symmetric L\'evy processes are considered instead of rotationally symmetric $\alpha$-stable processes,
and the case that the index function $\alpha(x)$ depends on $x.$

\begin{proposition}\label{l4-1}
Let $\sL$ be the operator given by \eqref{e4-1} such that the coefficients satisfy all the assumptions below \eqref{e4-1},
$k(x,z)$ is bounded between two positive constants,
 and that there
is a constant $\beta\in(0,1)$ so that
$b(x) \in C_b^\beta(\R^d)$ and
$$
\sup_{z\in \R^d}|k(x,z)-k(y,z)|\le c_0|x-y|^{\beta},\quad x,y\in \R^d
$$ for some $c_0>0$.
Assume that
\begin{equation}\label{l4-1-1}
\I_{\{|z|\le 1\}}\,\Pi (dz) =\I_{\{|z|\le 1\}}\frac{1}{|z|^{d+\alpha_0}}\,dz
\end{equation}
for some $\alpha_0\in (0,2)$.
For $\alpha_0\in (0,1)$, we
assume in addition
that $b(x)=\int_{\{|z|\le 1\}}z\frac{k(x,z)}{|z|^{d+\alpha}}\,dz$; for $\alpha_0=1$, we
assume in addition
that $k(x,z)=k(x,-z)$ for all $x,z\in \R^d$. Then
assumptions ${\bf (A1)}$ and ${\bf (A2)}$ are
satisfied.
\end{proposition}

\begin{proof}  For simplicity, we only prove the case that $\alpha_0\in (1,2)$, since the proofs of the cases $\alpha_0\in (0,1)$ and $\alpha_0=1$ are similar.

(1) We first assume that $\Pi (dz) =|z|^{-d-\alpha_0}\,dz$. According \cite[Theorem
1.5]{CZ1} (see also \cite[Theorem 1.4]{Peng}), there is a unique fundamental solution $p(t,x,y): \R_+\times
\R^d\times \R^d\to \R_+$ of the operator $\sL$. Then, the existence and the uniqueness of the Feller process
$X:=((X_t)_{t\ge0};(\Pp^x)_{x\in \R^d})$ associated with the operator $\sL$ was mentioned in
\cite[Remark 1.6]{CZ1}. In particular, $p(t,x,y)$ is the transition density function of the process $X$ with respect to the Lebesgue measure. By two-sided estimates and
gradient estimates of $p(t,x,y)$ stated in \cite[Theorem 1.5, (i)
and (vi)]{CZ1}, we can easily see that the process $X$ is
irreducible and enjoys the strong Feller property; that is, the
associated semigroup $(P_t)_{t\ge0}$ maps  measurable bounded functions into
continuous bounded functions.

Concerning
assumption {\bf (A1)}, it is clear that the probability law of $X$ solves
the martingale problem for $(\sL,C_c^\infty(\R^d))$ in the sense that
for every $f\in C_c^\infty(\R^d)$ and $x\in \R^d$,
$$f(X_t)-f(x)-\int_0^t \sL f(X_s)\,ds,\quad t\ge0$$ is a $\Pp^x$-martingale, see \cite[Remark 1.2 (iv)]{CZ1}.
By our assumptions on $k(x,z)$ and $b(x)$ again and the
process $X$ being
conservative (see \cite[Theorem 1.1, (iv)]{CZ1}),
$X$ solves the martingale problem for $(\sL,C_b^2(\R^d))$ as well.
If we regard
$X$ as a $\T^d$-valued process,
then the associated semigroup is still irreducible and has the strong Feller property
by the statements above, see, for example,
the proof of \cite[Proposition 1]{F2}
or the argument of \cite[Section 4]{San}. This along with \cite[Theorem 1.1]{Wa} gives us Assumption {\bf(A2)}.

(2)  Let $\sL$ be the operator given in Proposition \ref{l4-1}. Set $\hat \Pi (dz) :=|z|^{-d-\alpha_0}\,dz$ and
\begin{equation}\label{l4-1-2}
\sL_0f(x):=\int_{\R^d}\left(f(x+z)-f(x)-\langle \nabla f(x), z\rangle\I_{\{|z|\le 1\}}\right)k(x,z)\,\hat \Pi (dz) +
\langle b(x),\nabla f(x)\rangle.
\end{equation}
Then,
$$
\sL f(x)=\sL_0f(x)+ \sA f(x),
$$
where
\begin{equation}\label{e:bopeer}
\sA f(x)=\int_{\{|z|>1\}}(f(x+z)-f(x))k(x,z)\,(\Pi (dz) -\hat \Pi (dz) ).\end{equation}
It is clear that, under assumptions on $k(x,z)$ and $\Pi (dz) $, there is
a constant $c_1>0$ such that
$\| \sA f\|_\infty\le c_1\|f\|_\infty$ for all $f\in B_b(\R^d)$.
By bounded perturbation results for martingale problems, one can deduce that assumption {\bf (A1)} holds for the operator $\sL$;
see, e.g., \cite[Chapter 4, Section 10, p.\ 253]{EK}.

By the proof of Proposition 6 and Remark 8 in \cite{CW},
we know that the process $(X_t)_{t\ge 0}$ is irreducible. On the other hand, let $(P_t)_{t\ge0}$ and $(P_t^0)_{t\ge0}$ be the semigroups associated with the operator $(\sL, C_b^2(\R^d))$ and $(\sL_0, C_b^2(\R^d))$, respectively. It holds that
$$P_t f=P_t^0f+\int_0^t P_s^0 \sA P_{t-s}f\,ds,\quad t>0, f\in B_b(\R^d).$$ This along with the fact that $(P_t^0)_{t\ge0}$ has the strong Feller property yields that $(P_t)_{t\ge0}$ also enjoys the strong Feller property. Thus  Assumption {\bf(A2)} holds, thanks to \cite[Theorem 1.1]{Wa} again.
\end{proof}

\begin{proposition}\label{l4-2}
Let $\sL$ be the operator given in Proposition $\ref{l4-1}$. If $\alpha_0\in (1,2)$, then assumption ${\bf (A3)}$ is also satisfied. \end{proposition}

\begin{proof} Similar to the proof of Proposition  \ref{l4-1},
let $\sL_0$ be defined by \eqref{l4-1-2}. We write
$$
\sL_0 f(x) =\int_{\R^d}\left(f(x+z)-f(x)-\langle \nabla f(x), z\rangle\right)k(x,z)\,\hat \Pi (dz) +
\langle \hat b_\infty(x),\nabla f(x)\rangle,
$$
where $$\hat \Pi (dz) :=\frac{1}{|z|^{d+\alpha_0}}\,dz,\quad \hat b_\infty(x):=b(x)+\int_{\{|z|>1\}}z\frac{k(x,z)}{|z|^{d+\alpha_0}}\,dz.$$
Let $p_0(t,x,y)$ and $P_t^0$ be the fundamental solution and Markov semigroup associated with
$\sL_0$ respectively. Note that $\hat b_\infty\in C_b^\beta(\R^d;\R^d)$. By \cite[Theorem 1.5]{CZ1}, for any $t\in (0,1]$ and $x,y\in \R^d$,
\begin{equation}\label{p7-1-1}
\begin{split}
 p_0(t,x,y)\le & \frac{ c_0 t}{\left(t^{1/\alpha_0}+|x-y|\right)^{d+\alpha_0}},\\
|\nabla_x p_0(t,\cdot,y)(x)|\le &\frac{ c_0 t^{1-1/\alpha_0}}{\left(t^{1/\alpha_0}+|x-y|\right)^{d+\alpha_0}}.
\end{split}
\end{equation}
(Note that in our setting we can take $\eta=0$ in \cite[Theorem
1.5]{CZ}, see also \cite[Theorem 1.4]{Peng}.)
For any $\lambda>0$, let $R_\lambda^0$ be the $\lambda$-resolvent of the semigroup $(P_t^0)_{t\ge0}$, i.e.,
$$
R_\lambda^0 f(x):=\int_0^\infty e^{-\lambda t}P_t^0 f(x) \,dt,\quad f\in C(\T^d)\,x\in \R^d.
$$
According to \eqref{p7-1-1}, we can see that $R_\lambda^0$ is an operator
such that $R_\lambda^0: C(\T^d)\to C^1(\T^d)$ so that
$$
\|R_\lambda^0 f\|_\infty+\|\nabla R_\lambda^0 f\|_\infty\le \frac{c_1}{\lambda}\|f\|_\infty,\quad f\in C(\T^d).
$$
where $c_1$ is a positive constant independent of $\lambda$ and $f$. Here, we used the fact that $\alpha_0\in (1,2)$.
Furthermore,
It is well known that $R_\lambda^0=(\lambda-\sL_0)^{-1}$. Thus,
$(\lambda -\sL_0)^{-1}:C(\T^d)\to C^1(\T^d)$ and
\begin{equation}\label{p7-1-3}
\|(\lambda-\sL_0)^{-1} f\|_\infty+\|\nabla (\lambda -\sL_0)^{-1}f\|_\infty\le \frac{c_1}{\lambda}\|f\|_\infty,\quad f\in C(\T^d).
\end{equation}

Let $\sA f$ be defined by \eqref{e:bopeer}. By the assumption on $k(x,z)$, $\sA:C(\T^d)\to C(\T^d)$  satisfies that
\begin{equation}\label{p7-1-4}
\| \sA f\|_\infty\le c_2\|f\|_\infty,\quad f\in C(\T^d).
\end{equation}
Note that $\sL=\sL_0+\sA$. Then, for each $\lambda>0$,
$$
(\lambda -\sL)^{-1}=(\lambda -\sL_0)^{-1}\left(1-\sA(\lambda -\sL_0)^{-1}\right)^{-1}.
$$
Therefore, combining \eqref{p7-1-3} with \eqref{p7-1-4}, we  find  that
for every $\lambda>\lambda_0:=c_1c_2>0$, $(\lambda -\sL)^{-1}:C(\T^d)\to C^1(\T^d)$ is well defined such that
\begin{equation}\label{p7-1-5}
\|(\lambda -\sL)^{-1} f\|_\infty+\|\nabla (\lambda -\sL)^{-1}f\|_\infty\le c_3(\lambda)\|f\|_\infty,\quad \,
\lambda>\lambda_0,\ f\in C(\T^d).
\end{equation}

For every $f\in C(\T^d)$ with $\mu(f)=0$, let $\psi_f:=-\int_0^\infty P_t f\,dt$, which is well defined by
\eqref{a1-2-1}. Moreover, $\psi_f\in \mathscr{D}(\sL)$, $\sL \psi_f=f$, $\mu(\psi_f)=0$
 and $\|\psi_f\|_\infty\le c_4\|f\|_\infty$.
In particular, for every $\lambda>\lambda_0$, it holds that $$\psi_f=(\lambda-\sL)^{-1}\left(\lambda \psi_f-f\right).
$$
Hence by \eqref{p7-1-5},
for any $\lambda>\lambda_0$ we obtain
$$
\|\psi_f\|_\infty+\|\nabla \psi_f\|_\infty\le c_3(\lambda)\|\lambda \psi_f-f\|_\infty\le c_5(\lambda)\|f\|_\infty.
$$
Let $(X_t^x)_{t\ge 0}$ be the process associated with the martingale problem for $(\sL, C_b^2(\R^d))$
with initial value $x$. Let $f\in C(\T^d)$ such that $\mu(f)=0$. Then, for every $\psi\in \mathscr{D}(\sL)$ satisfying $\sL \psi =f$
and $\mu(\psi)=0$,
we have $\Ee[\psi(X_t^x)]=\psi(x)+\int_0^t \Ee[f(X_s^x)]\,ds$. Letting $t\to \infty$ and applying
\eqref{a1-2-1}, we get $\psi(x)=-\int_0^\infty P_s f(x)\,ds$. This means there exists a unique
$\psi\in \mathscr{D}(\sL)$ satisfying $\sL \psi =f$. Therefore, according to all the conclusions above, we prove that Assumption {\bf (A3)} holds.
\end{proof}

\begin{remark}\label{r7-1}
For simplicity, in Proposition \ref{l4-1}
 we require $\Pi (dz) $ and $b(x)$ to have special forms; for instance, $\Pi (dz) $ satisfies \eqref{l4-1-1} and
$b(x)=\int_{\R^d} z\frac{k(x,z)}{|z|^{d+\alpha_0}}\,dz$ when $\alpha_0\in (0,1)$. These conditions
are used to verify Assumption {\bf (A3)} under minimal regularity requirements on $k(x,z)$ and
$b(x)$. Indeed, under more general assumptions on $k(x,z)$ and $b(x)$ (that is, it is not required that $b(x)=\int_{\R^d} z\frac{k(x,z)}{|z|^{d+\alpha_0}}\,dz$
when $\alpha_0\in (0,1)$), we can still verify {\bf (A1)}, {\bf (A2)} (see \cite{Ku} for details)
and weaken \eqref{a3-1-2} into
\begin{equation}\label{a3-1-2--}
\|\psi\|_\infty+\|\nabla \psi\|_\infty\le C_1\|f\|_{C^\beta}.
\end{equation} Then, under the conditions above, Theorems \ref{t3-2}, \ref{t3-3} and \ref{t3-4} still hold true with some small modifications
in their proofs.
We note that \eqref{a3-1-2--} is closely related to
the Schauder estimates
for  L\'evy-type operators,
see \cite{BK,Bass, DK, Ku0, Ku1} and references therein for more details.

Moreover, suppose that $k\in C_b^\infty(\R^d\times \R^d)$ and
$b\in C_b^\infty(\R^d;\R^d)$. Then, by using the theory of pseudo-differential operators, we can prove the existence of
the Feller process
$X:=((X_t)_{t\ge0};(\Pp_x)_{x\in \R^d})$ associated with $(\sL,C_c^\infty(\R^d))$, and moreover the process $X$ can be written explicitly via a solution
of a stochastic differential equation
with jumps; see \cite[Chapter 3]{BSW} for more details. Hence, we may
obtain the following estimates
for the associated semigroup $(P_t)_{t\ge0}$ through the Bismut-type formula (see \cite{Kun} and references therein):\begin{align*}
&\|\nabla P_tf\|_\infty\le c_1\|\nabla f\|_\infty,\quad  0<t\le1,\\
&\|\nabla P_t f\|_\infty\le c_2\|f\|_\infty,\quad 1/2<t<1.
\end{align*}
According to these estimates and \eqref{a1-2-1}, we can find that for every $f\in C^1(\T^d)$ with $\mu(f)=0$, there exists a unique
$\psi \in \mathscr{D}(\sL)$ such that $\sL \psi=f$, $\mu(\psi)=0$ and
$$
\|\psi\|_\infty+\|\nabla \psi\|_\infty\le c_3\left(\|f\|_\infty+\|\nabla f\|_\infty\right).
$$
This also suffices to prove Theorems \ref{t3-2}, \ref{t3-3} and \ref{t3-4} with some modifications
in the proofs. \end{remark}

\medskip

\noindent {\bf Acknowledgements.}
The authors thank Boris Solomyak for the reference \cite{B} on almost periodic functions.
The research of Xin Chen is
supported by the National Natural Science Foundation of China (No.\ 11871338).\ The research of Zhen-Qing Chen is  partially supported
by Simons Foundation grant 520542 and a Victor Klee Faculty
Fellowship at UW.\ The research of Takashi Kumagai is supported by
JSPS KAKENHI Grant Number JP17H01093 and by the Alexander von
Humboldt Foundation.\ The research of Jian Wang is supported by the
National Natural Science Foundation of China (No.\ 11831014), the
Program for Probability and Statistics: Theory and Application (No.\
IRTL1704) and the Program for Innovative Research Team in Science
and Technology in Fujian Province University (IRTSTFJ).

\vskip 0.3truein

\noindent {\bf Xin Chen:}
   Department of Mathematics, Shanghai Jiao Tong University, 200240 Shanghai, P.R. China. \texttt{chenxin217@sjtu.edu.cn}

\bigskip

\noindent{\bf Zhen-Qing Chen:}
   Department of Mathematics, University of Washington, Seattle,
WA 98195, USA. \texttt{zqchen@uw.edu}

\bigskip

\noindent {\bf Takashi Kumagai:}
 Research Institute for Mathematical Sciences,
Kyoto University, Kyoto 606-8502, Japan.
\texttt{kumagai@kurims.kyoto-u.ac.jp}

\bigskip

\noindent {\bf Jian Wang:}
    College of Mathematics and Informatics  \&
Fujian Key Laboratory of Mathematical Analysis and Applications (FJKLMAA)  \& Center for Applied Mathematics of Fujian Province (FJNU), Fujian Normal University, 350007 Fuzhou, P.R. China. \texttt{jianwang@fjnu.edu.cn}

\end{document}